\documentclass[a4paper,11pt]{article}
\usepackage[utf8]{inputenc}
\usepackage{lmodern} 
\usepackage[T1]{fontenc}
\usepackage{textcomp}
\usepackage{enumitem} 
\usepackage[scr=boondoxo,frak=boondox,bb=boondox]{mathalfa}
\title{\normalsize{\textbf{TOPOLOGICAL TRIVIALITY AND LINK-CONSTANCY IN DEFORMATIONS OF INNER KHOVANSKII NON-DEGENERATE MAPS. }}}
\author{Julian D. Espinel Leal \href{https://orcid.org/0009-0004-8604-0624}{\includegraphics[scale=0.2]{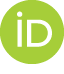}} \and Eder L. Sanchez Quiceno \href{https://orcid.org/0000-0001-6073-2383}{\includegraphics[scale=0.2]{orcid.png}}}
\newcommand{\Addresses}{{
  \bigskip
  \footnotesize
  \noindent  J.~D.~Espinel Leal,\\ \textsc{Institute of Mathematics and Computer Science (ICMC)
Universidade de S\~ao Paulo(USP),\\ Avenida Trabalhador S\~ao-Carlense, 400 - Centro
CEP: 13566-590 - S\~ao Carlos - SP, Brazil.}\\ 
E-mail:\texttt{ jdel941@gmail.com}

\medskip 
  
\noindent  E.~L.~Sanchez Quiceno,\\ \textsc{Departamento de Matemática, Universidade Federal de S\~ao Carlos,\\ Rodovia Washington Lu\'is, Km 235, CEP 13560-905, Caixa Postal 676 S\~ao Carlos- SP, Brazil.}

\medskip 
\noindent \textsc{Instituto Universitario de Matem\'atica Pura y Aplicada,\\ Ed. 8E, Camino de Vera, s/n 46022 València, Spain.}\\
E-mail:\texttt{ ederleansanchez@alumni.usp.br}
\par\nopagebreak
  \medskip 
\textsc{}\par\nopagebreak
}}
\date{}
\usepackage{graphicx}
\usepackage[all]{xy}
\usepackage{mathrsfs}
\usepackage[usenames,dvipsnames]{color}
\usepackage{makeidx}
\usepackage{subcaption}
\usepackage{graphicx}
\usepackage{float}
\usepackage{indentfirst}
\usepackage{amsmath}
\usepackage{mystyles}
\usepackage{graphicx,color}
\usepackage{setspace}
\usepackage{epigraph}
\usepackage{lmodern}
\usepackage{amsfonts,amssymb,amsmath,latexsym}
\usepackage{verbatim}
\usepackage{lscape}
\usepackage{tikz-cd}
\usetikzlibrary{calc,intersections,through,backgrounds}
\usepackage{xfrac}    
\usepackage{pgfplots}
\usepackage{mathtools}
\usepackage{tikz}
\usepackage{faktor}
\usepackage{amsmath}
\usepackage{booktabs}
\usepackage{pinlabel}
\usepackage{titlesec}
\usepackage{lipsum}
\usepackage{soul}
\usepackage{bm}
\usepackage[T1]{fontenc}
\usepackage{tikz,tikz-3dplot}
\providecommand{\MR}[1]{}

\usepackage[text={35pc,256mm}]{geometry}

\tdplotsetmaincoords{60}{120}

\titleformat{\section}
        [block]
        {\large \mdseries\centering}
        {\thesection}
        {1ex}
        {}
        [
        ]%
        \titleformat{\subsection} [block] {\large \mdseries}
        {\thesubsection} {1ex} {}
        [
        ] 

\usepackage[english]{babel}
\usepackage{hyperref}
\hypersetup{
    colorlinks=true,
    linkcolor=blue,
    filecolor=magenta,      
    urlcolor=cyan,
}

\makeindex

\theoremstyle{thmstyleone}%
\newtheorem{theorem}{Theorem}[section]
\newtheorem{proposition}[theorem]{Proposition}%
\newtheorem{claim}[theorem]{Claim}
\newtheorem{lemma}[theorem]{Lemma}
\newtheorem{corollary}[theorem]{Corollary}

\theoremstyle{thmstyletwo}%
\newtheorem{ex}[theorem]{Example}%
\newtheorem{remark}[theorem]{Remark}%

\theoremstyle{thmstylethree}%
\newtheorem{definition}[theorem]{Definition}%

\newcommand{\overbar}[1]{\mkern 2.0mu\overline{\mkern0mu #1 \mkern-1.0mu}\mkern 2mu}

\newcommand{\C}{\mathbb{C}}       
\newcommand{\w}{\boldsymbol{w}}
\newcommand{\x}{\boldsymbol{x}}
\newcommand{\R}{\mathbb{R}}       
\newcommand{\K}{\mathbb{K}}  
\newcommand{\N}{\mathbb{N}}       

\newcommand{\rme}{\mathrm{e}}
\newcommand{\rmi}{\mathrm{i}}

\DeclareMathOperator{\supp}{supp}
\usepackage{lineno}

   \makeatletter
\@namedef{subjclassname@2020}{%
  \textup{2020} Mathematics Subject Classification}
  \setstcolor{red}
\begin{document}
\maketitle

\begin{abstract} 
For real and mixed polynomial maps $f=(f^1,\dots,f^p)$ satisfying $f(0)=0$, we introduce the notion of \textit{Inner Khovanskii Non-Degeneracy} (IKND), that generalizes a previous non-degeneracy condition introduced by Wall for complex polynomial functions (J. Reine Angew. Math. 509 (1999), 1-19). We prove that IKND is a sufficient condition ensuring that the link of the singularity of $f$ at the origin is smooth and well-defined. We study one-parameter deformations of an IKND map $f$, given by $F(\bm{x},\varepsilon)=f(\bm{x})+\theta(\bm{x},\varepsilon)$, with $ F(0,\varepsilon)=0$. We prove that the deformation is \textit{link-constant} under suitable conditions on $f$ and $\theta$, meaning that the ambient isotopy type of the link remains unchanged along the deformation. Furthermore, by employing a strong version of this non-degeneracy, \textit{Strong Inner Khovanskii Non-Degeneracy} (SIKND), we obtain results on topological triviality. 
In the final section, we present link-constant deformations for IKND mixed polynomial functions of two variables. We also explore several applications motivated by the recent findings of Araújo dos Santos, Bode, and Sanchez Quiceno (Bull. Braz. Math. Soc. (N.S.) 55 (2024), no. 3, Paper No. 34).
\vspace{0.2cm}

\noindent \textit{Keywords:} deformations, link-constancy, inner non-degeneracies, mixed maps, topological triviality, real polynomial maps.
   
     \vspace{0.2cm} 
     
\noindent \textit{MSC Classification:} 
Primary 14M25 ; Secondary 32S55, 14M10, 14J17, 57K10, 32S05 . 	

\end{abstract}

\maketitle

\section{Introduction}\label{sec1}
This paper aims to establish results on equisingularity for deformations of real and mixed polynomial maps under a new non-degeneracy setting that extends the condition introduced by Wall, called the inner Newton non-degeneracy (INND)\footnote{Originally denoted NPND* in \cite{wall}, it was later termed INND by other authors.} \cite{wall}.

\subsection{Background}
A key question in the study of complex (or real) analytic singularities is whether a given deformation 
\[ F(\boldsymbol{x},\varepsilon): \mathbb{K}^n \times \mathbb{K} \to \mathbb{K}^p, \ \K=\C \text{ or } \R \]
is equisingular (or trivial). In other words, do the maps 
\[ F_\varepsilon(\boldsymbol{x}):=F(\boldsymbol{x},\varepsilon):\mathbb{K}^n \to \mathbb{K}^p \]
remain locally equivalent, in some suitable sense, for sufficiently small values of \(\varepsilon\)?

A foundational result by L\^e and Ramanujam \cite{LeRamanujam1976} shows that, when \(n \neq 3\), any deformation $F$ of a complex analytic function having constant Milnor number ($\mu$-constant deformation) is \(V\)-trivial, that is, for the associated family of hypersurfaces \(\{F_{\varepsilon}^{\, -1}(0)\}\), there exist germs of homeomorphisms 
\( h_\varepsilon: (\mathbb{C}^n,0) \to (\mathbb{C}^n,0) \)
such that \(h_\varepsilon (F_{\varepsilon}^{\, -1}(0))=F_{0}^{\, -1}(0)\). The case \(n=3\) remains an important open problem in singularity theory. 

In \cite{Timourian}, Timourian improves the result of L\^e and Ramanujam by proving that a $\mu$-constant deformation is topologically trivial for $n \not= 3$, that is, there exist germs of homeomorphisms \(h_\varepsilon : (\C^n,0) \to (\C^n,0)\) such that \(F_\varepsilon \circ h_\varepsilon = F_0\). Later, in \cite{wall}, Wall studied deformations of a complex analytic function $f$, that is, $F(\boldsymbol{x}, 0 ) = f(\boldsymbol{x})$. Under the hypothesis of inner Newton non-degeneracy on \(f\), and imposing additional conditions on \(F_{\varepsilon}-f\), he proved that the deformation is topologically trivial, even in the difficult case \(n = 3\). 

A natural setting to extend Wall's results is the class of deformations of mixed polynomial functions, namely, complex-valued polynomials in complex variables and their conjugates. This class is equivalent to that of real polynomial maps $\R^{2n} \to \R^2$, and it includes the classical complex polynomial functions. Studying these objects as mixed polynomial functions offers significant advantages: it enables the formulation of non-degeneracy conditions in terms of the Newton polyhedron of a mixed polynomial function, as defined by Oka in \cite{Oka2010}, and ensures the resolution of singularities from the perspective of toric modifications, as also demonstrated by Oka in \cite{Oka2010}. This generalizes, to mixed polynomial functions, the previous notions of non-degeneracies of complex polynomial functions due to Kouchnirenko \cite{Kouchnirenko1976}. Within this framework, Araújo dos Santos, Bode and Sanchez Quiceno \cite{AraujoBodeSanchez} extended the notion of inner Newton non-degeneracy (INND) to mixed polynomial functions of two variables and proved that such functions have a weakly isolated singularity at the origin, that is, $f^{-1}(0)\cap \Sigma(f)=\{0\}$, where $\Sigma(f)$ denotes the singular locus. This implies that $f$ defines a smooth and well-defined link of the singularity (in the sense of Milnor \cite{Milnor1968}). They showed that, under $\Gamma$-niceness condition (Definition \ref{def:niceness}), these links depend on the Newton boundary and have a decomposition of the form $\textbf{L}([L_1,L_2,\dots,L_{N-1}],L_{N})$ (see Definition~\ref{nestedlink}), where each link $L_i$ corresponds to a 1-face $\Delta_i$ of the Newton boundary of \(f\). Subsequently, Bode in \cite{BodeSanchez2023} studied the topological properties of the pieces $L_i$ and completely characterized the links of $\Gamma$-nice and INND mixed polynomial functions.
To the author's knowledge, a complete characterization of the links of INND mixed polynomial functions remains open.

For real polynomial maps
\(\mathbb{R}^4 \to \mathbb{R}^2\) with a weakly isolated singularity at the origin, the corresponding characterization is known: every link in the 3-sphere can occur. This was established by Akbulut and King \cite{Akbulut_King1981} and later polished by Bode \cite{Bode2022semih}. For maps with an isolated singularity at the origin, however, the characterization problem remains open. It is conjectured that every fibered link arises in this setting \cite{Benedetti_Shiota1998}. Recent advances further support this conjecture \cite{Araujo_Sanchez2021,Bode2019,bodesat,bode:thomo,Bode2025}. In contrast, for complex polynomial functions of two variables, the corresponding links are completely determined by explicit conditions derived from their Puiseux pairs, following classical results of Brauner, Burau, and Kähler \cite{Brauner1928,Burau1933,Burau1934,kahler}.

In this paper, we generalize the notion of inner Newton non-degeneracy (INND) for real polynomial maps and mixed polynomial maps, introducing the concept of \textit{(strong) inner Khovanskii non-degeneracy} ((S)IKND). Our goal is to obtain results on equisingularity for one-parameter deformations of such maps, while also producing tools for characterizing links of singularities. Further results on non-degenerate maps can be found in Bivià Ausina and Khovanskii \cite{BiviaAusina2007,Khovanskii1977}, and their deformations are studied, among others, by Damon, Eyral, Gaffney, Nguyen and Oka \cite{Damon1989,Eyral2022,Eyral2017,Thang2022}. Since complex polynomial maps are particular cases of mixed polynomial maps, our results extend previous work in that setting. Altogether, our contributions help bridge the gap between the complex and real settings of singularity theory by providing a unified non-degeneracy framework.

\subsection{Main results}
Our new non-degeneracy conditions are based on a collection of \(C\)-face diagrams, which play a similar role to Newton boundaries of the coordinate functions in the classical definition of Khovanskii non-degeneracy (KND) introduced in \cite{Khovanskii1977}. As examples of IKND maps, we present the following classes: convenient KND maps, semi-weighted homogeneous maps (SWH), semi-radially weighted homogeneous mixed polynomial maps (SRWH), and KND polynomial functions in three variables that are not necessarily convenient (see Proposition~\ref{prop:kndsurfaces}). We investigate two important properties to obtain equisingular deformations: (weak) no coalescing of critical points and the $\rho$-uniform Milnor radius. In the case of complex analytic functions, the no coalescing of critical points of $F$ is equivalent to being $\mu$-constant \cite{King1980}. Related results appear in \cite{Menegon2023,Menegon2024}. The notion of $\rho$-uniform Milnor radius, studied by King \cite{King1980} and Bekka \cite{zBekka2015}, is further discussed in \cite{oka1997non,Eyral2017-2,Thang2022}. These properties are relevant because (weak) no coalescing of critical points of $F$ ensures that each $F_\varepsilon$ has a (weakly) isolated singularity at the origin. In the case of complex polynomial maps each $F_\varepsilon$ defines an isolated complete intersection singularity (ICIS). The (weak) no coalesing of critical points implies that the link of each \( F_\varepsilon \), $F^{-1}_{\varepsilon}(0)\cap \mathbb{S}_r $, is smooth and well-defined for $r$ sufficiently small (Here, $\mathbb{S}_r$ denotes a standard sphere of radius $r$ and centered at the origin). However, it is not known if the isotopy type of these links changes with the parameter $\varepsilon$. We address this by showing in Proposition~\ref{Prop:GD} that the existence of a $\rho$-uniform Milnor radius is sufficient to guarantee that the isotopy type is preserved, a property we term \textit{link-constancy}. In earlier work by King \cite{King1980} and Bekka \cite{zBekka2015}, no coalescing of critical points together with  additional conditions were used to obtain topologically trivial deformations. In Propositions~\ref{cor: top trivi} and \ref{Prop:toptrivialwithrho}, we provide examples of such deformations.

We therefore restrict our study to deformations of (S)IKND real or mixed polynomial maps, and we establish our first main result, Theorem~\ref{Th:Main1}. This theorem states that if \( f = (f^1, \dots, f^p) \) is SIKND (resp. IKND) with respect to a set of \( C \)-face diagrams \( D_1, \dots, D_p \), and if the terms of \( F^{\, j}_\varepsilon -f^j \) lie above or on \( D_j \) for each \( j = 1, \dots, p \), then the deformation $F$ has no coalescing (resp. weak no coalescing) of critical points. Under additional hypotheses, these deformations are also link-constant and topologically trivial. These conditions are established in our main results: Theorems~\ref{teo: knd co -> urm} and  \ref{Th:Main2} for maps, and Theorem~\ref{th:nonice} for mixed polynomial functions of two variables.

As a first application of the latter theorem, Corollary~\ref{cor:nicecase}, we yield a full extension, into the mixed setting of two variables, of the classical result that the link type of a convenient non-degenerate complex polynomial function is completely determined by the monomials on its Newton boundary (see \cite{AraujoBodeSanchez}). As a second application, we obtain Proposition~\ref{th:piecewise family}, 
which states that any $\Gamma$-nice and INND mixed polynomial function $f$ can be 
placed in a link-constant piecewise linear analytic family of mixed polynomials 
$\{f_\varepsilon\}_{\varepsilon \in [0,1]}$ such that $f_0=f$ and the link of $f_1$ 
is isotopic to a link of the form
\(
\mathbf{L}([L_1, L_2, \dots, L_{N-1}], L_N).
\)
In particular, this yields \cite[Theorem~1.2]{AraujoBodeSanchez}, with a slightly 
different construction of the links $L_{i}$. We also show that every member of such link-constant piecewise linear analytic family 
is always $\Gamma$-nice. This property is established in Proposition~\ref{prop: niceisopen}, 
where it is proved that, for a fixed $C$-face diagram $D$, both the $D$-nice and the 
non $D$-nice conditions define open properties. As a final application, in Proposition~\ref{cor:classoka} we present a generalization 
of \cite[Theorem~1.9]{Bode2025}, showing that the class of links arising from IKND 
mixed polynomials coincides with that arising from convenient and KND mixed polynomial functions.

\subsection{Outline}
The remainder of the paper is organized as follows.
In Section~\ref{section2}, we introduce in Definition~\ref{def: ind} the non-degeneracy condition (S)IKND for real and mixed polynomial maps, in  Definition~\ref{Def:SWH}, the classes of SWH and SRWH maps. In Lemma~\ref{lemma: IND M -> IKND} motivated by the correspondence between real polynomial maps \( \mathbb{R}^{2n} \to \mathbb{R}^{2p} \) and mixed polynomial maps \( \mathbb{C}^n \to \mathbb{C}^p \), we show that any (S)IKND mixed polynomial map of \(n\)-variables also satisfies the (S)IKND condition when regarded as a real polynomial map of \((2n)\)-variables. However, the converse does not hold: not every real polynomial map that is (S)IKND arises from a (S)IKND mixed polynomial map (see Remark~\ref{contraexemploreciprocalLemma210}). In Section~\ref{section3}, we establish our first main result, Theorem~\ref{Th:Main1}, which gives sufficient conditions to obtain (weak) no coalescing of the critical points of the deformation. This result is relevant for the Section~\ref{section4}, where we study link-constancy and topological triviality. In Corollary~\ref{cor:convenient} and Proposition~\ref{prop:defkndsurfaces}, we highlight the significance of the hypotheses in Theorem~\ref{teo: knd co -> urm}.
In Subsection~\ref{section5}, we focus on mixed polynomial functions of two variables. 
In Proposition~\ref{prop:caracterizationinner}, we show that our class is equivalent 
to the class of INND mixed polynomial functions introduced 
in \cite{AraujoBodeSanchez}. We then present the proof  of our main result Theorem~\ref{th:nonice} 
and give the first application in Corollary~\ref{cor:nicecase}.

The niceness condition splits the class of IKND mixed polynomial functions into two nonempty subclasses: those that are 
$\Gamma$-nice and those that are not \(\Gamma\)-nice. See Examples~\ref{exnonice} and \ref{ex:nonice2} for a non $\Gamma$-nice IKND mixed polynomial function and Example~\ref{ex:nice} for a $\Gamma$-nice one. In Proposition~\ref{prop:characterizationknd}, we provide a 
characterization of the KND condition for certain mixed polynomial functions. This result allows 
one to determine whether a given IKND mixed polynomial function is $\Gamma$-nice. 
It also facilitates the proof of Proposition~\ref{prop: niceisopen}. Finally, we present two further applications in Propositions~\ref{th:piecewise family} 
and \ref{cor:classoka}.

\vspace{0.2cm}

\textbf{Acknowledgments:} Julian Espinel Leal was supported by the CAPES Grant \\ PROEX-11162221/D. Eder L. Sanchez Quiceno was supported by S\~ao Paulo Research Foundation (FAPESP), Brasil, Process Number 2023/11366-8, 2024/17116-6. The authors  are thankful to Professor Benjamin Bode from UPM, Spain, Professor Carles Bivià Ausina from UPV, Spain, and  Professor Raimundo N. Ara\'ujo dos Santos from ICMC-USP, Brazil, for their valuable discussions and comments that contributed to the paper.

\section{Inner Khovanskii non-degenerate maps (IKND)}\label{section2}
Let $\mathcal{J}$ be a set of homogeneous linear functions $\ell $ such that $\ell(\boldsymbol{\nu})>0 $ for all $\boldsymbol{\nu} \in (\R_{>0})^n$. Let $J$ be a finite subset of $\mathcal{J}$, we define a minimal function of $J$ to be $\ell_J: (\R_{\geq 0})^n \rightarrow \R $, $\ell_J(\boldsymbol{\nu}) = \min_{\ell \in J} \ell (\boldsymbol{\nu})$. The $C$-\textit{face diagram} of $J$ is the set $D(J)=\{\boldsymbol{\nu} \in (\R_{\geq 0})^n \mid \ell_J(\boldsymbol{\nu}) = 1 \}$. A \textit{face} of a $C$-face diagram $D$ is a nonempty subset $\Delta \subset D$, such that there exist index sets $ I_\Delta \subseteq [n]:= \{1, \dots, n\}$ and $J_\Delta \subset J$ for which $\Delta $ is defined by equations:
\[
\begin{cases}
	\pi_i(\boldsymbol{\nu}) = 0, & \text{for } i \in (I_\Delta)^c, \\
	\ell(\boldsymbol{\nu}) = 1, & \text{for } \ell \in J_\Delta,
\end{cases}
\]
where $\pi_i$ is the projection onto the $i$-th coordinate. A face \(\Delta\) is called an \textit{inner face} if $I_\Delta = [n]$, equivalently, if it is not contained in any proper coordinate subspace\footnote{The subspaces of the form $\{x_i=0 \}_{i \in [n]}$ } of \(\R^n\).

Let $S$ be a finite set of $\mathbb{Q}_{\geq 0 }^n$, the \textit{Newton polyhedron} on $S$, denoted by $\Gamma_{+}(S)$, is the convex hull of
\[ \bigcup_{\nu \in S} \{ \nu + (\R_{\geq 0})^n\} ,\]
where + denotes the Minkowski sum. The \textit{Newton boundary}  $\Gamma(S)$ is defined as the union of compact faces of $\Gamma_+(S)$. We say that $\Gamma(S)$ is \textit{convenient} if it intersects all  coordinate axes. Observe that a convenient Newton boundary \( \Gamma \) is  a $C$-face diagram.

Given a weight vector $\w \in (\mathbb{Q}_{>0})^n$ and a polyhedron $\Gamma_+$, they define a linear function $\ell_{\w}: \Gamma_+ \to \R_{>0}$ given by 
\(\ell_{\w}(\boldsymbol{\nu}):=\langle \w , \boldsymbol{\nu}  \rangle  \), where $\langle \ , \  \rangle$ is the euclidean inner product. Denote $\Delta(\w;\Gamma)$ the face of the Newton boundary $\Gamma$ of $\Gamma_+$ where $\ell_{\w}$ takes its minimal value, denoted by $d(\w;\Gamma)$. 

For $\Gamma(S)$, define the finite set of weight vectors, $$\mathcal{P}(\Gamma(S)):=\{\w_1,\dots,\w_N\}\subset (\mathbb{Q}_{>0})^n,$$  where $\Delta(\w_i;\Gamma(S)), \ i=1,\dots,N,$ is an $(n-1)$-face, $\Gamma(S)=\cup_{i=1}^{N} \Delta(\w_i;\Gamma(S))$, and  $\Delta(\w_i;\Gamma(S))=\Delta(\w_j;\Gamma(S))$ if and only if $i=j$.  
\vspace{0.2cm}

Let $f:\R^n \to \R$ be a polynomial function of the form
\[f(\boldsymbol{x})= \sum_{\boldsymbol{\nu}} c_{\boldsymbol{\nu}} \boldsymbol{x}^{\boldsymbol{\nu}}, \ c_{\boldsymbol{\nu}} \in \R \]
where, $\boldsymbol{x}^{\boldsymbol{\nu}}=(x_1)^{\nu_1}\cdots (x_n)^{\nu_n}$
and satisfying $f(0)=0$. The \textit{support} of $f$ is defined by 
\(\supp (f):=\{\boldsymbol{\nu} \in (\mathbb{Z_{\geq0}})^n \mid c_{\boldsymbol{\nu}}\neq 0 \}.\)

A complex-valued polynomial function $f : \C^n\to \C$ is called a \textit{mixed polynomial function} if $f(0)=0$, and it is a complex polynomial in complex variables $\x=(x_1,\dots ,x_n)$ and their conjugates $\overbar{\x}=(\overbar{x_1},\dots,\overbar{x_n})$, that is, 
$$f(\x) = \sum_{\boldsymbol{\nu},\boldsymbol{\mu}} c_{\boldsymbol{\nu},\boldsymbol{\mu}} \x^{\boldsymbol{\nu}} \overbar{\x}^{\boldsymbol{\mu}},\ c_{\boldsymbol{\nu},\boldsymbol{\mu}} \in \C.$$ 
The \textit{support} of $f$ is defined, in \cite{Oka2010},  by 
\(\supp (f):=\{\boldsymbol{\nu}+\boldsymbol{\mu} \in (\mathbb{Z_{\geq0}})^n \mid c_{\boldsymbol{\nu},\boldsymbol{\mu}}\neq 0 \}\).

Note that the notion of support of a function \( f \) depends on whether \( f \) is regarded as a real or a mixed polynomial function. For instance, if \( f \) is a mixed real-valued polynomial function, such as \( f(\boldsymbol{x}) = \boldsymbol{x} \overline{\boldsymbol{x}} \), its real support differs from its mixed support. To avoid any ambiguity, we will always specify the nature of \( f \), namely whether it is considered as a real or as a mixed polynomial function.

The \textit{Newton polyhedron} of a real polynomial function (or mixed polynomial function) $f$ is defined as $\Gamma_+(f):=\Gamma_{+}(\supp (f)).$ The \textit{Newton boundary} is defined as $\Gamma(f):=\Gamma(\supp (f)),$ 
and $\mathcal{P}(f):=\mathcal{P}(\Gamma(\supp (f)))$. 
The (mixed) real polynomial function $f$ is \textit{convenient} if $\Gamma(f)$ is a $C$-face diagram. 

Consider $\w \in (\mathbb{Q}_{>0})^n$ and $\Gamma(S)$. If $f$ is a real polynomial function then we  define the \textit{face function} of $f$ respect to $\w$ and $\Gamma(S)$ by
\begin{equation*}
	f_{\Delta(\w;\Gamma(S))}(\boldsymbol{x}):= \sum_{\boldsymbol{\nu} \in \Delta(\w;\Gamma(S)) }c_{\boldsymbol{\nu}} \boldsymbol{x}^{\boldsymbol{\nu}}.
\end{equation*}
Analogously, if $f$ is a mixed polynomial function the \textit{face function} of $f$ respect to $\w$ and $\Gamma(S)$ is defined by 
\begin{equation*}
	f_{\Delta(\w;\Gamma(S))}(\x) = \sum_{\boldsymbol{\nu}+\boldsymbol{\mu} \in \Delta(\w;\Gamma(S))} c_{\boldsymbol{\nu},\boldsymbol{\mu}} \x^{\boldsymbol{\nu}} \overbar{\x}^{\boldsymbol{\mu}}.
\end{equation*}
We denote $\Delta(\w;f):=\Delta(\w;\Gamma(f))$, $d(\w;f):=d(\w;\Gamma(f))$, and 
$f_{\w}:=f_{\Delta(\w;\Gamma(f))}$.

Let \( D_1, \dots , D_p\) be \( C \)-face diagrams. Consider the Minkowski sum
\(
(D_1 + (\R_{\geq 0})^n) + \cdots + (D_p + (\R_{\geq 0})^n).
\)
The union of its compact faces defines a new \( C \)-face diagram, denoted by \( D:=D_1 + \cdots + D_p. \) Note that this is not the Minkowski sum of the diagrams \( D_i \) themselves. However, each face \( \Delta \) of \( D \) admits a unique decomposition \( \Delta = \Delta_1 + \cdots + \Delta_p \), where each $\Delta_i$ is a face of $D_i$, in this case, the sum is indeed the Minkowski sum of the corresponding faces.
\vspace{0.2cm}

For \(I \subseteq [n] \), and $\K=\R \text{ or } \C$, define
\[\K^I:=\{\boldsymbol{x} \in \K^n \mid x_i=0 \text{ if } i \not\in I\} \simeq \K^{|I|}, \]
\[(\K^*)^I:= \left\{\boldsymbol{x} \in \K^n \mid \prod _{i \in I} x_i \neq 0 \right \} \cap \K^I \simeq (\K^* )^{|I|}.\]
Let \( f:\K^n \rightarrow \K^p \) be a real polynomial map (or mixed polynomial map). If $\K = \C$, we can regard $f$ as a map $\tilde{f}: \R^{2n} \rightarrow \R^{2p}$ by identifying $\C $ with $\R^2$. The set of critical points of $f$ is defined by $\Sigma(f): = \{ \boldsymbol{x} \in \K^n \mid \mathrm{rank } (J_{\tilde{f}}(\boldsymbol{x})) < 2p \}$, where $J_{\tilde{f}}(\boldsymbol{x})$ denotes the real Jacobian matrix of $\tilde{f} $. Note that if $f$ is a complex polynomial map then, $\Sigma(f) = \{\boldsymbol{x} \in \C^n \mid \mathrm{rank } (J_{f}(\boldsymbol{x})) < p \}$, where $J_{f}(\boldsymbol{x})$ is the complex Jacobian matrix of $f$.
\begin{definition}\label{def: ind}
	Let \( f = (f^1, \dots, f^p) \) be a real polynomial map (or mixed polynomial map). We say that $f$ is: 
	\begin{itemize}
		\item \textbf{Inner Khovanskii non-degenerate (IKND)} if there exist \( C \)-face diagrams \( D_1, \dots, D_p \) such that
		\begin{enumerate}
			\item[(i)] for each $j=1,\dots,p,$, \(d(\boldsymbol{w}; f^{j}) \geq d(\boldsymbol{w}; D_j), \text{ for all }  \boldsymbol{w} \in \mathcal{P}(D_j).\)
			\item[(ii)] for every inner face \( \Delta \) of \(D=D_1+\cdots+ D_p\) and every nonempty subset \( I \subseteq [n] \), we have:
			\[
			\Delta \cap \R^I \neq \emptyset \quad \Rightarrow \quad \Sigma(f_\Delta) \cap V(f_\Delta) \cap (\K^*)^I = \emptyset,
			\]
			where \( f_\Delta = (f^{\, 1}_{\Delta_1}, \dots, f^{\, p}_{\Delta_p}) \) denotes the corresponding face system related to $\Delta=\Delta_1+\cdots+\Delta_p$, with $\Delta_i$ face of $D_i$.
		\end{enumerate} 
		\item \textbf{Strongly inner Khovanskii non-degenerate (SIKND)} if it is IKND and satisfies the following extra condition
		\begin{enumerate}
			\item[(ii')] for every inner face \( \Delta \) of \(D=D_1+\cdots+ D_p\) and every nonempty subset \( I \subseteq [n] \), 
			\[\Delta \cap \R^I \neq \emptyset \quad \Rightarrow \quad \Sigma (f_\Delta) \cap (\K^*)^{I} = \emptyset.\]
		\end{enumerate}
	\end{itemize}
	
\end{definition}
\begin{remark}
    \begin{enumerate}
        \item[(i)] Condition~(i) of Definition \ref{def: ind} is equivalent to requiring that no point of \(\supp (f^i)\) lies below \(D_i\), i.e., \( \supp (f^i) \subseteq D_i + (\R_{\geq 0})^n \) for each \( i=1,\dots, p \).
	    \item[(ii)] Note that a  $C$-face diagram satisfying (i), (ii) (or (ii')) is not necessarily unique. For example, consider the function $f: \K^2 \rightarrow \K$, given by $f(x_1, x_2) = x_{1}^{\,3} + x_1 x_2 + x_{2}^{\,3}$, and the  $C$-face diagrams $D_1:=D({J_1})$, $D_2:=D({J_2})$, and $D_3:=D({J_3})$, with 
	\[
	\begin{aligned}
		J_1 = \{\ell_1(\nu) = \nu_1 + \nu_2\}, \
		J_2 = \left\{
		\ell_{2}(\nu) = \tfrac{2\nu_1 + \nu_2}{3}, \
		\ell_{3}(\nu) = \tfrac{\nu_1 + 2\nu_2}{3}
		\right\}, \ 
		J_3 = \left\{\ell_3(\nu)  \right\}.
	\end{aligned}
	\]
	In this case, $f$ satisfies both $(i)$ and $(ii)$ for any of these  $C$-face diagrams $D_1$, $D_2$, and $D_3$.
        \end{enumerate}
\end{remark}

\subsection{(S)KND maps}

\begin{definition} \label{def: KND}
	Let $f = (f^1, \dots, f^p): \mathbb{K}^n \to \mathbb{K}^p$ be a real polynomial map (or mixed polynomial map). Fix a strictly positive weight vector $\boldsymbol{w} \in (\mathbb{N}_{>0})^n$, and denote by $f_{\boldsymbol{w}} = (f_{\boldsymbol{w}}^{\,1}, \dots, f_{\boldsymbol{w}}^{\,p})$ the associated face map of $f$ with respect to $\boldsymbol{w}$. We say that $f$ is:
	\begin{itemize}
		\item \textbf{Khovanskii non-degenerate (KND)} if for every strictly positive weight vector $\boldsymbol{w}$,
		\[
		\Sigma (f_{\boldsymbol{w}})\cap V(f_{\boldsymbol{w}}) \cap (\mathbb{K}^*)^n = \emptyset.
		\]
		
		\item \textbf{Strongly Khovanskii non-degenerate (SKND)} if for every strictly positive weight vector $\boldsymbol{w}$, we have
		\[
		\Sigma (f_{\boldsymbol{w}}) \cap (\mathbb{K}^*)^n = \emptyset.
		\]
	\end{itemize}
\end{definition}

For a map $f=(f^1,\dots,f^p):\K^n \to \K^p$, we define 
\(I^j_{\mathrm{nc}}(f):=\{i \in [n] \mid \supp (f^j) \cap \R^{\{i\}}= \emptyset\}\) and   \(I_{\rm{nc}}(f):=\cup_{j=1}^p I^j_{\rm{nc}}(f).\) We say that the map $f$ is \textit{convenient} if \(I_{\rm{nc}}(f)=\emptyset\). In other words, if any coordinate function $f^j$ is convenient.

\begin{remark}
    If $f:\C^n \to \C$ is complex analytic or mixed, then KND reduces to Newton non-degeneracy introduced by Kouchnirenko \cite{Kouchnirenko1976} for complex analytic and by Oka \cite{Oka2010} for mixed functions.
\end{remark}

\begin{proposition} \label{prop: knd co -> iknd}
	Let $f$ be a convenient (S)KND real polynomial map (or mixed polynomial map). Then $f$ is (S)IKND.    
\end{proposition}
\begin{proof}   
	Suppose that $f$ is SKND. Let $D_j := \Gamma(f^j)$ denote the $C$-face diagram of $f^j$, for $j = 1, \dots, p$, and let $D := D_1 + \cdots + D_p$. Suppose that $f$ is not SIKND with respect to these $C$-face diagrams. Then, there exists an inner face $\Theta'$ of $D$, and a subset $I \subseteq [n]$, such that $\Theta := \Theta' \cap \mathbb{R}^I \neq \emptyset$ and 
	$\Sigma(f_{\Theta'}) \cap (\K^*)^I \neq \emptyset.$ Since $I = [n]$ implies that $f$ is not SKND, we may assume that $I \subsetneq [n]$. Let $\boldsymbol{w} \in (\mathbb{Q}_{>0})^n$ be such that $\Theta'=\Delta(\boldsymbol{w};D)$, define $\boldsymbol{q} \in (\mathbb{Q}_{>0})^n$, by $q_i = w_i$ if $i \in I$ and $q_i = w_i+1$ if $i \not\in I$. Then $\Theta = \Delta(\boldsymbol{q};D)$ and $d := d( \boldsymbol{w};D) = d( \boldsymbol{q};D) $. 
	Let $\Theta' = \Theta'_1 + \cdots + \Theta'_p$ and $\Theta = \Theta_1 + \cdots + \Theta_p$ be the unique decompositions corresponding to the Minkowski sum of faces, where $\Theta'_j,\Theta_j \subset D_j$, for each $j = 1, \dots, p$.
	
	Let \( \boldsymbol{x}_1 \in \Sigma(f_{\Theta'}) \cap (\mathbb{K}^*)^I \). Thus, there exists \( (a_1, \dots, a_p) \in \mathbb{K}^p \setminus \{0\} \) such that
	\[
	\sum_{j=1}^p a_j \nabla f^{\, j}_{\Theta'_j} (\boldsymbol{x}_1) = 0.
	\]
	
	Define \( \boldsymbol{x}_0 = (x_{0,1}, \dots, x_{0,n}) \in \mathbb{K}^n \) by setting \( x_{0,i} = x_{1,i} \) if \( i \in I \), and \( x_{0,i} = 1 \) otherwise. 
	\begin{claim} \label{claim:4} $f^{\, j}_{\Theta_j}(\boldsymbol{x}_0)= f^{\, j}_{\Theta'_j}(\boldsymbol{x}_1)$ for $j=1,\dots,p$.        
	\end{claim}
	\begin{proof}
		We have $f^{\, j}_{\Theta'_j} =f^{\, j}_{\Theta_j} +  (f^{\, j}_{\Theta'_j} - f^{\, j}_{\Theta_j})$. Let $h=a_{\boldsymbol{\nu}} \boldsymbol{x}^{\boldsymbol{\nu}}$ be a monomial of $(f^{\, j}_{\Theta'_j} - f^{\, j}_{\Theta_j})$, then $d(\boldsymbol{q}'; a_{\boldsymbol{\nu}} \boldsymbol{x}^{\boldsymbol{\nu}}) > d(\boldsymbol{w}; a_{\boldsymbol{\nu}} \boldsymbol{x}^{\boldsymbol{\nu}})=d(\boldsymbol{w}; D_j)$, this means, that there is $i \in [n] \setminus I$, such that $\nu_i \not = 0$. Hence $h(\boldsymbol{x}_1)=a_{\boldsymbol{\nu}} \boldsymbol{x}_1^{\boldsymbol{\nu}} =0$, this implies that $(f^{\, j}_{\Theta'_j} - f^{\, j}_{\Theta_j})(\boldsymbol{x}_1)=0$. The result follows from $f^{\, j}_{\Theta_j}(\boldsymbol{x}_0) =f^{\, j}_{\Theta_j}(\boldsymbol{x}_1)$.
	\end{proof}
	
	Then, by Claim \ref{claim:4} we have  
	\[
	\frac{\partial f^{\, j}_{\Theta_j}}{\partial x_i}(\boldsymbol{x}_0) =\frac{\partial f^{\, j}_{\Theta_j}}{\partial x_i}(\boldsymbol{x}_1) =
	\begin{cases}
		\frac{\partial f^{\, j}_{\Theta'_j}}{\partial x_i}(\boldsymbol{x}_1) & \text{if } i \in I, \\
		0 & \text{if } i \notin I.
	\end{cases}
	\]      
	Hence,  
	\[
	\sum_{j=1}^p a_j \nabla f^{\, j}_{\Theta_j} (\boldsymbol{x}_0) = \sum_{j=1}^p a_j \nabla f^{\, j}_{\Theta_j} (\boldsymbol{x}_1) = 0,
	\]  
	which implies that \( \boldsymbol{x}_0 \in \Sigma(f_{\Theta}) \cap (\mathbb{K}^*)^n \), contradicting the fact that \( \Sigma (f_{\Theta}) \cap (\mathbb{K}^*)^n = \emptyset \). 
	
	Now, if we have that $f$ is KND then $\boldsymbol{x}_1 \in V(f_{\Theta'})$. Since $f_{\Theta'}(\boldsymbol{x}_1) = f_{\Theta}(\boldsymbol{x}_1) = f_{\Theta}(\boldsymbol{x}_0) $, it follows that \( \boldsymbol{x}_0 \in \Sigma(f_{\Theta}) \cap V(f_{\Theta}) \cap (\mathbb{K}^*)^n \). Contradicting the fact that \( \Sigma (f_{\Theta})\cap V(f_{\Theta}) \cap (\mathbb{K}^*)^n = \emptyset \).
\end{proof}
\subsection{SWH and SRWH maps}
\begin{definition}\label{def:wh}
	A real polynomial function (resp. mixed polynomial function) \(f: \mathbb{K}^n \to \mathbb{K}\) is called \textbf{weighted homogeneous} (WH) (resp. \textbf{radially weighted homogeneous} (RWH)) of weight-type \((\boldsymbol{w};d)=(w_1, \dots, w_n; d)\), $w_i,d \in \mathbb{Q}_{>0}$, if for any \(\lambda \in \mathbb{R} \setminus \{0\}\), the following condition holds:
	\begin{equation*}
		f(\lambda^{w_1}x_1, \dots, \lambda^{w_n}x_n) = \lambda^{d} f(x_1, \dots, x_n).
	\end{equation*}
	
	A real polynomial map (resp. mixed polynomial map) \(f = (f^1, \dots, f^p): \mathbb{K}^n \to \mathbb{K}^p\), where \(n \geq p\), is called \textbf{weighted homogeneous} (WH) (resp. \textbf{radially weighted homogeneous} (RWH)) of weight-type \((\boldsymbol{w};\boldsymbol{d}=(d_1,\dots,d_p))\) if each coordinate function \(f^i\), for \(i = 1, \dots, p\), is WH (resp. RWH) of weight-type \((\boldsymbol{w}; d_i)\). 
	
\end{definition}
\begin{definition}\label{Def:SWH}
	A real polynomial map (resp. mixed polynomial map) \( f = (f^1, \dots, f^p): \mathbb{K}^n  \to \mathbb{K}^p \) is \textbf{semi-weighted homogeneous} (SWH) (resp. \textbf{semi-radially weighted homogeneous} (SRWH)) of weight-type \((\boldsymbol{w};\boldsymbol{d})\) if:
	\begin{itemize}
		\item[(i)] The map \(f_{\w}\) is WH (resp. RWH) of weight-type \((\boldsymbol{w};\boldsymbol{d})\) and \(\Sigma({f_{\w}}) \cap V({f_{\w}})=\{0\}\),
		\item[(ii)] for each $j=1,\dots,p,$ \( d(\boldsymbol{w}; f^j - f_{\boldsymbol{w}}^j) > d_j. \)
	\end{itemize}
\end{definition}
This definition requires that the variety \( V(f_{\boldsymbol{w}}) \) has an isolated singularity at the origin, which is a weaker condition than requiring \( f_{\boldsymbol{w} } \) itself to have an isolated singularity, as in the classical definition.

\begin{proposition} \label{prop:SWH->IKND}
	SWH maps are IKND. If, in addition, $\Sigma(f_{\boldsymbol{w}})=\{0\}$, then it is SIKND. 
\end{proposition}
\begin{proof}
	Let \( f \) be a SWH map of weight type \( (\boldsymbol{w}; (d_1, \dots, d_p)) \). For each \( i = 1, \dots, p \), define the \( C \)-face diagram \( D_i := D(J_i) \), where \( J_i = \{ \ell_{\boldsymbol{w}} / d_i \} \). Since \( \Sigma(f_{\boldsymbol{w}}) \cap V(f_{\boldsymbol{w}})= \{0\} \), it follows that \( f \) is IKND with respect to the collection of \( C \)-face diagrams \( D_1, \dots, D_p \). To get the condition SIKND we use $\Sigma(f_{\boldsymbol{w}})=\{0\}$.
\end{proof}

\begin{ex}\label{ex1:SIKND}
	Consider the real polynomial  
	\(f(x_1,x_2,x_3)= x_{1}^{\,12} + x_1 x_{2}^{\,4} x_3 +( x_{2}^{\,3} - x_{3}^{\,2} )^2.
	\) We claim that \( f \) is SIKND but not SKND. To see that \( f \) is \emph{not} SKND, take the weight vector \( \w = (2,2,3) \). Then the face function related to $\w$ is  
	\(
	f_{\w}(\boldsymbol{x}) = (x_{2}^{\,3} - x_{3}^{\,2} )^2,
	\)  
	which clearly has a singular locus containing the point \( (1,1,1) \in \mathbb{R}^{\{1,2,3\}} \). Therefore,  
	\(
	\Sigma({f_{\boldsymbol{w}}}) \cap \mathbb{R}^{\{1,2,3\}} \neq \emptyset,
	\)
	and \( f \) is not SKND.
	To check whether \( f \) is SIKND, consider the directional function given by
	\(
	\ell(\boldsymbol{\nu}) = \left\langle \left(\tfrac{1}{12}, \tfrac{2}{12}, \tfrac{3}{12} \right), \boldsymbol{\nu} \right\rangle.
	\)
	The corresponding $C$-face diagram \( D(J), \ J=\{\ell\} \) has a unique 1-face, so \( f \) is SIKND if and only if \( \Sigma(f_{D(J)}) = \{0\} \). Since all monomials of \( f \) lie on this face, we have \( f = f_{D(J)} \), and thus it suffices to show that the system of partial derivatives has only the trivial solution:
	\begin{align}
		\frac{\partial f}{\partial x_1}(\boldsymbol{x}) &= 12 x_{1}^{\,11} +  x_{2}^{\,4} x_3 = 0, \label{eq:fx1} \\
		\frac{\partial f}{\partial x_2}(\boldsymbol{x}) &= 2 x_{2}^{\,2} \left(2x_1x_2x_3 + 3( x_{2}^{\,3} - x_{3}^{\,2} )\right) = 0, \label{eq:fx2} \\
		\frac{\partial f}{\partial x_3}(\boldsymbol{x}) &= x_1 x_{2}^{\,4} - 4x_3 ( x_{2}^{\,3} - x_{3}^{\,2} ) = 0. \label{eq:fx3}
	\end{align}
	The origin \( (0,0,0) \) clearly satisfies the system. It is also easy to check that for any $i \in \{1,2,3\}$, if $x_i = 0$, then $\boldsymbol{x} = (0,0,0)$. Therefore, we assume that there is a nontrivial solution \( \boldsymbol{x} \in \mathbb{R}^{\{1,2,3\}} \), i.e., all \( x_i \ne 0 \).
	
	From \eqref{eq:fx2}, dividing by \( 2 x_{2}^{\,2} \ne 0 \), we get:
	\begin{equation}\label{eq:ex17}
		x_1 = \frac{-3( x_{2}^{\,3} - x_{3}^{\,2} )}{2x_2x_3}.
	\end{equation}
	Substitute \eqref{eq:ex17} into \eqref{eq:fx3}:
	\[
	( x_{2}^{\,3} - x_{3}^{\,2} ) \left( \frac{-3 x_{2}^{\,3} }{2x_3} - 4x_3 \right) = 0.
	\]
	Thus, either  \( x_{2}^{\,3} = x_{3}^{\,2} \), or 
	\(-3 x_{2}^{\,3} = 8x_{3}^{\,2} \).
	If \( x_{2}^{\,3} = x_{3}^{\,2}\), then \(x_1 = 0 \) by \eqref{eq:ex17}, contradicting \( x_1 \ne 0 \). Thus, \( x_{2}^{\,3} = -\frac{8}{3}x_{3}^{\,2} \). Then, we get: \[x_1 = \frac{-3(x_{2}^{\,3} - x_{3}^{\,2} ) }{2x_2x_3} = \frac{-3(-\frac{8}{3}x_{3}^{\,2} - x_{3}^{\,2} )}{2x_2x_3} = \frac{9x_3^{\,2}}{2x_2x_3} = \frac{9x_3}{2x_2}. \]
	Substitute this value of \( x_1 \) into \eqref{eq:fx1}:
	\[
	12 \left(\frac{9x_3}{2x_2} \right)^{11} + x_{2}^{\,4} x_3 = 0.
	\]
	Clearing denominators and simplifying gives a contradiction unless \( x_3 = 0 \), which is not allowed. Therefore, no nontrivial solution exists, and we conclude that
	$\Sigma(f) = \{0\}$, so \( f \) is SIKND.
\end{ex}

\subsection{IKND mixed maps}
Let $\tilde{f}:\R^{2n} \rightarrow \R^{2p}$ be a polynomial map, $\sigma$ and $\varsigma$ be permutations of $n$ and $p$ elements, respectively. We identify $\C^n \cong \R^{2n}$ by setting $x_i=x_{\sigma(2i-1) } + \mathrm{i} x_{\sigma(2i) }$ for $i=1,\dots,n$. Note that $ x_{\sigma(2i-1)} = (x_i + \overbar{x_{i}} ) /2$ and $ x_{\sigma(2i)} = (x_i - \overbar{x_{i}} ) /(2\mathrm{i})$.
We denote ${f}_{\sigma,\varsigma} : \C^n \rightarrow \C^p $ the mixed polynomial map defined by  component functions $${f}^j_{\sigma,\varsigma}(\boldsymbol{x}, \boldsymbol{\overbar{x}}) = \tilde{f}^{\varsigma(2j-1)} \left(\boldsymbol{x}_{\sigma}\right)+\mathrm{i} \tilde{f}^{\varsigma(2j)}\left(\boldsymbol{x}_{\sigma}\right), \, j=1, \dots, p, $$
where $\boldsymbol{x}_\sigma = (x_{\sigma(1)},\dots,  x_{\sigma(2n)})$.

\begin{lemma} \label{lemma: IND M -> IKND}
	Let $\tilde{f}: \R^{2n} \rightarrow \R^{2p}$ be a polynomial map. If there are permutations $\sigma$ and $\varsigma$ of $n$ and $p$ elements, respectively, such that ${f}_{\sigma,\varsigma}: \C^n \rightarrow \C^p$ is (S)IKND mixed polynomial map, then $\tilde{f}$ is (S)IKND.
\end{lemma}
\begin{proof}
	Suposse that $\sigma$ and $\varsigma $ the trivial permutations. Let ${D}_1, \dots, {D}_p \subset (\R_{\geq 0})^n$ be the mixed $C$-face diagrams of ${f}^1, \dots, {f}^p$, respectively, and let ${J}_j$ be the set of linear functions defining ${D}_j$, for $j=1\dots,p$. Each ${\ell} \in {J}_j$ can be extended to a linear function $\tilde{\ell}$ on $\R^{2n}$ as follows: 
	$$\tilde{\ell}(\boldsymbol{\nu}) = {\ell} ( \nu_{1} + \nu_2,\nu_{3} + \nu_4,  \dots, \nu_{2n-1} + \nu_{2n}). $$ 
	Let $\tilde{J}_j$ be the set of these extended linear functions, and define the $C$-face diagrams $$\tilde{D}_{2j-1}=\tilde{D}_{2j}= \{ \boldsymbol{\nu} \in \mathbb{R}^{2n}_{\geq 0} \mid \tilde\ell_{{J}_j}(\boldsymbol{\nu}) = 1 \}, \text{ for } j=1, \dots, p.$$
	We have $\supp (\tilde{f}^{k}) \subset \tilde{D}_k+ \R^{2n}_{\geq 0}$, for $k=1, \dots, 2p$. Indeed, consider a monomial $\boldsymbol{x}^{\boldsymbol{\nu}} $ of $\tilde{f}^{k}$. It originates from a real (if $k$ is even) or imaginary (if $k$ is odd) part of a mixed monomial $\boldsymbol{x}^{\boldsymbol{\mu}}\boldsymbol{\overbar{x}}^{\boldsymbol{\xi}} $, so that
	$${\nu}_{2i-1}+{\nu}_{2i} = {\mu}_{i}+{\xi}_{i}, \text{ for } i=1,\dots, n. $$
	Thus, for every linear function $\tilde\ell$ defining a face of $\tilde{D}_k$, we obtain
	\begin{align*}
		\tilde{\ell}(\boldsymbol{\nu}) &= {\ell} (\nu_{1} + \nu_2,\nu_{3} + \nu_4,  \dots, \nu_{2n-1} + \nu_{2n})\\
		&= {\ell} ({\mu}_{1}+{\xi}_{1},{\mu}_{2}+{\xi}_{2},  \dots, {\mu}_{n}+{\xi}_{n}) \\
		& \geq 1,
	\end{align*}
	which implies that $\supp (\tilde{f}^{k}) \subset \tilde{D}_k+ \R^{2n}_{\geq 0}$.
	\begin{claim}\label{claim: mix inner}
		Let $\tilde{\Theta}$ be an inner face of $\tilde{D}=\tilde{D}_1+ \dots + \tilde{D}_{2p}$. If $\boldsymbol{w} \in (\mathbb{Q}_{\geq 0})^{2n}$ is such that ${\tilde\Theta} = \Delta (\boldsymbol{w}; \tilde{D})$, then $w_{2k-1} = w_{2k}$ for $k=1,\dots,n$.
	\end{claim}
	\begin{proof}
		Let $\boldsymbol{w} \in (\mathbb{Q}_{>0})^{2n}$ be such that $\tilde{\Theta} = \Delta (\boldsymbol{w}; \tilde{D})$. Define $K = \{ k \in [n] \mid w_{2k} \not= w_{2k-1} \}$. For each $k \in K$, define $m(k)= 2k-1$ if $w_{2k-1} < w_{2k}$ and $m(k)= 2k$ if $w_{2k-1} >  w_{2k}$, analogously, define $M(k)= 2k-1$ if $w_{2k-1} > w_{2k}$ and $M(k)= 2k$ if $w_{2k-1} < w_{2k}$. Let $m(K):= \{m(k) \in [2n] \mid k \in K \}$.
		
		Suppose by contradiction that $K \not= \emptyset$. Let $k \in K$. Since $\tilde\Theta$ is an inner face, there is $j \in [2p]$ such that $M(k) \in I_{\tilde\Theta_j}$. Choose $\boldsymbol{\nu} \in \tilde\Theta_j$ with $\nu_{M(k)} \not= 0$ and define $\boldsymbol{\nu}'$ by setting $\nu'_{M(k)}= 0$, $\nu'_{m(k)}= \nu_{2k-1} + \nu_{2k} $ and $ \nu_i'= \nu_{i}$ for $ i \in [n] \setminus \{m(k),M(k)\}$. Note that $\boldsymbol{\nu}' \in \tilde{D}_j$ since for all linear function $\tilde\ell$ defining $\tilde{D}_j$ we have that $\tilde\ell(\boldsymbol{\nu})=\tilde\ell(\boldsymbol{\nu'})$. But
		\begin{align*}
			\langle \boldsymbol{w};\boldsymbol{\nu}'\rangle < \langle\boldsymbol{w}; \boldsymbol{\nu} \rangle=d(\boldsymbol{w};\tilde{D}_j),
		\end{align*}
		which is a contradiction.
	\end{proof}
	Thus, by Claim \ref{claim: mix inner} there is a one-to-one correspondence between the inner faces of \( D \) and \( \tilde{D} \), such that if \( \tilde{\Theta} \) is an inner face of \( \tilde{D} \) and \( \Theta \) is the corresponding inner face of \( D \), then \( \tilde{f}_{\tilde{\Theta}} = f_\Theta \).
	For $\tilde{I} \subseteq [2n]$ define $$ {I}:= \{ i \in [n] \mid k=2i \text{ or } k=2i-1 \text{ for some } k \in \tilde{I} \}  .$$
	Suppose that ${\tilde\Theta} \cap \R^{\tilde{I}} \not= \emptyset$, then we have that ${\Theta} \cap \R^{{I}}  \not= \emptyset $. If ${f}$ is an IKND mixed polynomial map then $\Sigma ({f}_{{\Theta}})\cap V({f}_{{\Theta}}) \cap (\C^*)^{{I}} = \emptyset $. Note that $(\R^*)^{\tilde{I}} \subseteq (\C^*)^{{I}} $. Therefore
	$\Sigma(\tilde{f}_{{\tilde{\Theta}}})\cap V(\tilde{f}_{{\tilde{\Theta}}}) \cap   (\R^*)^{\tilde{I}}  = \emptyset$. Analogously, if ${f}$ is SIKND mixed polynomial map, $\Sigma (\tilde{f}_{{\tilde{\Theta}}}) \cap   (\R^*)^{\tilde{I}}  = \emptyset$.
\end{proof}
\begin{remark}\label{contraexemploreciprocalLemma210}
	The reciprocal of Lemma~\ref{lemma: IND M -> IKND} is not true. In fact, the polynomial map  $\tilde{f}: \R^n \to \R^2$ given by $\tilde{f}(\boldsymbol{x})=(2x_1,2x_2(\sum_{i=1}^n x_{i}^{\,2} ))$ is IKND with respect to $\tilde{D}=D(J_1)+D(J_2)$, where $J_1=\{ \ell_1(\nu):=\sum_{i=1}^n \nu_i\}$ and $J_2=\{\ell_2(\nu):=\frac{\ell_1(\nu)}{3}\}$. However, for any permutation $\sigma$ and $\varsigma$, the mixed polynomial function ${f}_{\sigma,\varsigma}$ is not IKND. This is because $(\C^*)^{\{i\}}$, where $\sigma(2i-1)=1$ or $\sigma(2i)=1$, is always a set of critical points for any $({f}_{\sigma,\varsigma})_{\Delta}$.
\end{remark}
\subsection{IKND functions in three variables}
For a real polynomial function (or mixed polynomial function) $f$, we define \(D_{\boldsymbol v}(f):=\Gamma(S_{\boldsymbol v}(f))\) where, 
\[S_{\boldsymbol v}(f):=\supp (f) \cup \{v_{i} \boldsymbol{e}_i \mid i\in I_{\mathrm{nc}}(f),\ v_{i} \in \mathbb{Q}\},\  \boldsymbol v=\sum_{i \in I_{\rm{nc}}(f)}v_{i} \boldsymbol{e}_i.\]
\begin{lemma}\label{lemma:diagramnonconvcase} Let $f:\K^3 \to \K$ be a real polynomial function (or mixed polynomial function). If $f$ satisfies:
	\begin{itemize}
		\item[(i)] for any  $i \in I_{\rm{nc}}(f)$, there exists an index $j\neq i$ and $\boldsymbol{p} \in \supp (f) \cap \Gamma(f)\cap \R^{\{i,j\}}$ with $p_j=1$,
		\item[(ii)] $\supp (f) \cap \Gamma(f) \cap \R^{\{i,j\}}\neq \emptyset$ for all distinct indices $i,j \in \{1,2,3\}$. 
	\end{itemize}
	Then, there exists $\boldsymbol{v} \in \mathbb{Q}^n$ such that, for any $i \in I_{\rm{nc}}(f)$, there exists a unique inner face $\Delta(i)$ of $D_{\boldsymbol{v}}(f)$ satisfying $\Delta(i)\cap \R^{\{i\}} \neq \emptyset$. Moreover, $\Delta(i)$ is a 2-face with $\Delta \cap \R^{\{i,j\}}\neq \emptyset$ for any $j\neq i$ and  $f_{\Delta(i)}(\boldsymbol{x})=M_i(x_i,x_j)+N_i(x_i,x_k),$ where $i,j,k$ are distinct and $\supp (M_i)=\bm{p}$.  
\end{lemma}
\begin{proof}
	We will define \(\boldsymbol{v}\). Let $i \in I_{\rm nc}(f)$, by Condition~(i) there is $j \not= i$ and a point $\boldsymbol{p} \in \R^{\{i,j\}}$ such that $p_j=1$. By Condition~(ii) there is $\supp (f) \cap \R^{\{i,k\}} \not= \emptyset$ with $k \in \{1,2,3\} \setminus \{i,j\}$. Let $\boldsymbol{q} \in \Gamma(f) \cap \R^{\{i,k\}} \not= \emptyset$ be such that $q_k = \min \{ \nu_k \mid \boldsymbol{\nu} \in \Gamma(f) \cap \R^{\{i,k\}}\}$. Define $v_i$ such that
	\begin{align} 
		v_i &> \frac{\nu_k q_i + \nu_j p_iq_k-\nu_i q_k}{\nu_k+q_k(\nu_j-1)}, \label{eq: faceunique}
	\end{align}
	for all $\boldsymbol{\nu} \in \supp (f)$ with $(\nu_k,\nu_j) \not= (0,1)$.
	Let $\boldsymbol{w}_i  \in (\R_{>0})^3 $ be such that $w_i = 1, w_k=\frac{v_i - q_i}{q_k}$, $w_j=v_i - p_i$ and $\Delta(i):= \Delta(\boldsymbol{w}_i ; D_{\boldsymbol{v}}(f))$. The we have that for all $\boldsymbol{\nu}\in \supp (f)$, $\langle \boldsymbol{w}_i,\boldsymbol{\nu} \rangle \geq v_i$ and $v_i = \langle \boldsymbol{w}_i , \boldsymbol{p} \rangle = \langle \boldsymbol{w}_i , \boldsymbol{q}\rangle$. In fact, if $\boldsymbol{\nu} \in (\R^*)^3 \cap \supp (f)$
	\begin{align*}
		\langle \boldsymbol{\nu} , \boldsymbol{w} \rangle &= \nu_j w_j + \nu_k w_k + \nu_i \\
		&= (\nu_j -1 ) w_j + \nu_k w_k + (\nu_i - p_i) + w_j + p_i \\
		&=v_i\left((\nu_j-1)+\frac{\nu_k}{q_k}\right) - p_i\nu_j - \frac{\nu_k}{q_k}q_i+\nu_i+ v_i \\
		&> v_i.
	\end{align*} 
	The last inequality follows from \eqref{eq: faceunique}.
\end{proof}

\begin{proposition}\label{prop:kndsurfaces}
	Let $f: \K^3 \to \K$ be a KND real polynomial function (or mixed polynomial function) as Lemma~\ref{lemma:diagramnonconvcase} such that their associated polynomials $M_i, \ i \in I_{\rm nc}(f)$, satisfy $V(M_i)\cap (\K^*)^{3}=\emptyset$. Then,  $f$ is IKND.  \end{proposition}
\begin{proof}
	We first prove the case where $f$ is a mixed polynomial function. The proof for the real case is analogous.
	
	If $I_{\rm{nc}}(f)=\emptyset$, then by Proposition~\ref{prop: knd co -> iknd}, $f$ is IKND. We therefore suppose that $I_{\rm{nc}}(f)\neq\emptyset$.
	
	Take an index $i \in I_{\rm{nc}}(f)$. By Lemma~\ref{lemma:diagramnonconvcase} there is a unique 2-face $\Delta(i)$ in the $C$-face diagram $D$ satisfying $\Delta(i) \cap \R^{\{i\}}\neq \emptyset$ and  \(f_{\Delta(i)}(\boldsymbol{x})=M_i(x_i,x_j)+N_i(x_j,x_k),\) where $i,j,k$ are distinct and $M_i$ has the form $M_i(x_i,x_j)=x_jA(x_i)+\overbar{x_{j}}B(x_i)$. The face  $\Delta(i)$ also satisfies $\Delta(i) \cap \R^{\{i,j\}}\neq \emptyset$, $\Delta(i) \cap \R^{\{i,k\}}\neq \emptyset$, and $\Delta(i) \cap \R^{\{k\}}\neq \emptyset$ (assuming $N_i$ does not depend on $x_i$ or $\overbar{x_{i}}$). 
	
	We must show that $f$ satisfies (ii) in Definition~\ref{def: KND} for $\Delta(i)$. 
	We will proceed by contradiction. Assume that $\boldsymbol{a} \in  \Sigma (f_{\Delta(i)})\cap V(f_{\Delta(i)})$, by the characterization of critical points of mixed polynomial functions (see. \cite{Oka2010}) that means, there is $\lambda \in S^1$ such that:
	\begin{equation}\label{eq:1propikndsurfaces}
		f_{\Delta(i)}(\boldsymbol{a})=0 \text{ and } \left( \overbar{\frac{\partial f_{{\Delta(i)}}}{\partial x_1}}(\boldsymbol{a}),\overbar{\frac{\partial f_{{\Delta(i)}}}{\partial x_2}}(\boldsymbol{a}),\overbar{\frac{\partial f_{{\Delta(i)}}}{\partial x_3}}(\boldsymbol{a}) \right)=\lambda \left( \frac{\partial f_{{\Delta(i)}}}{\partial \overbar{x_{1}}}(\boldsymbol{a}),\frac{\partial f_{{\Delta(i)}}}{\partial \overbar{x_{2}}}(\boldsymbol{a}),\frac{\partial f_{{\Delta(i)}}}{\partial \overbar{x_{3}}}(\boldsymbol{a}) \right).
	\end{equation}
	We suppose that $\boldsymbol{a}$ violates the IKND criterion for $f$ on the face $\Delta(i)$, thus it satisfies some of the following cases:
	\begin{enumerate}
		\item  $\boldsymbol{a} \in (\C^*)^{\{i,j,k\}}$, then 
		\eqref{eq:1propikndsurfaces} implies that $\boldsymbol{a}$ belongs to $\Sigma(f_{\Delta(i)}) \cap V(f_{\Delta(i)}) \cap (\C^{*})^3$. This contradicts the non-degeneracy of $f$ for the face $\Delta^*$ of $\Gamma(f)$ that satisfies $f_{\Delta^*}=f_{\Delta(i)}$.   
		
		\item $\boldsymbol{a} \in (\C^*)^{\{i\}}$, then \eqref{eq:1propikndsurfaces} implies $|A(a_i)|=|B(a_i)|$. This implies that $V(M_i)\cap (\C^*)^{3}\neq \emptyset$, which is a contradiction. 
		\item  $\boldsymbol{a} \in (\C^*)^{\{i,j\}}$, then \eqref{eq:1propikndsurfaces} implies that  $\boldsymbol{a}$ belongs to $\Sigma(M_{{\Delta(i)}}) \cap V(M_{{\Delta(i)}})$.
		Thus $\boldsymbol{b} \in (\C^*)^3$, with $b_i=a_i$, $b_j=a_j$, belongs to $\Sigma(M) \cap  V(M) \cap (\C^{*})^3$. This contradicts the non-degeneracy of $M$.  
		
		\item  $\boldsymbol{a} \in (\C^*)^{\{i,k\}}$, then arguing analogously to the case (3), the point 
		$\boldsymbol{b} \in (\C^*)^3$, with $b_i=a_i$, $b_k=a_k$, belongs to $\Sigma(N) \cap V(N) \cap (\C^{*})^3$. This contradicts the non-degeneracy of $N$.
		
		\item  $\boldsymbol{a} \in \C^{\{k\}}$ and ${\Delta(i)} \cap \R^{\{k\}}\neq \emptyset$ (thus $N$ does not depend on $x_i$ or $\overbar{x_{i}}$). As the case (4), the point $\boldsymbol{b} \in (\C^*)^3$, with $b_k=a_k$, belongs to $\Sigma(N) \cap V(N) \cap (\C^{*})^3$. This contradicts the non-degeneracy of $N$. 
	\end{enumerate}
	Therefore, $f$ satisfies the condition of IKND for the inner face ${\Delta(i)}$. Hence, $f$ is IKND with respect to $D$.
\end{proof}
\section{Deformations of (S)IKND maps}\label{section3}
We consider deformations of a map $f$ of the form $F: \K^n \times \K \rightarrow \K^p$, with $F(\boldsymbol{x}, 0 ) = f(\boldsymbol{x}) $, $F(0, \varepsilon)=0$. Namely, if $\K=\R$ then $f$ is a real polynomial map, if $\K=\C$ then $f$ is a mixed polynomial map. For each value $\varepsilon\in \K$, let $F_\varepsilon(\boldsymbol{x}) := F(\boldsymbol{x} , \varepsilon)$.
\begin{definition}\label{def:nocoalescing}
	Let $K\subset \K$ be a connected set containing the origin. We say that the deformation has \textbf{no coalescing of critical point along $K$} if there exists $\epsilon > 0$ such that $\Sigma(F_\varepsilon) \cap B_\epsilon(0) \subseteq \{0\}$ for all $\varepsilon \in K$. Similarly, we say that the deformation $F$ has \textbf{weak no coalescing of critical points along $K$} if there exists $\epsilon > 0$ such that $\Sigma({F_\varepsilon}) \cap V({F_\varepsilon}) \cap B_\epsilon(0) \subseteq \{0\}$ for all $\varepsilon \in K$. If the deformation has (weak) no coalescing of critical point along any compact connected set containing the origin $K \subset \K$ then we simply say that the deformation has (weak) no coalescing of critical point.
\end{definition}
For $I \subset [n]$ the map $\pi^I: \mathbb{R}^n \to \mathbb{R}^{|I|}$ denotes the projection onto the coordinates in $I$, with kernel $(\pi^I)^{-1}(0)=\mathbb{R}^{I^c}$.
\begin{lemma} \label{lemma: find inner}
	Let $\boldsymbol{q} \in (\mathbb{Q}_{\geq 0})^n$. Let $I:= \{i \in [n] \mid q_i \not= 0\}$. Define $d:=d(\pi^I(\boldsymbol{q}); \pi^I(D\cap \R^I))$ and $\Theta:=(\pi^I)^{-1}(\Delta(\pi^I(\boldsymbol{q}); \pi^I(D\cap \R^I)))\cap \R^I$. There exists a weight vector $\boldsymbol{q'} \in (\mathbb{Q}_{>0})^n$ with $q'_i=q_i$ for $i \in I$, such that $d(\boldsymbol{q}';D)=d$ and $\Delta(\boldsymbol{q}'; D)=\Theta$. Furthermore, there is a weight vector $\boldsymbol{w} \in (\mathbb{Q}_{>0})^n$ with $w_i=q'_i$ for $i \in I_{\Theta}$, and $w_i < q'_i$ for $i \in I_{\Theta}^{\,c} $, such that $d(\boldsymbol{w};D)=d$ and $\Theta':=\Delta(\boldsymbol{w};D)$ is an inner face of $D$ satisfying $\Theta=\Theta' \cap \R^{I_{\Theta}}$.
\end{lemma}
\begin{proof}
	Take $S \subset (\mathbb{Q}_{>0})^n$ satisfying $\Gamma(S)=D$. We define $\boldsymbol{q}' \in (\mathbb{Q}_{>0})^n$ such that $q'_i=q_i$ for $i \in I$ and  $q'_i>\frac{d}{\alpha}$  for $i \in I^c$, where $\alpha= \min \{ \pi_i(\boldsymbol{s}) \mid \boldsymbol{s} \in  S\cap \R^{I^c}, \, i \in I^c,\, \pi_i(\boldsymbol{s})>0\}$.    
	
	Let $\boldsymbol{p} \in D \setminus \Theta$. Then, there exist $a_1,\dots, a_k \in \R_{>0}$ and $\boldsymbol{s}_1,\dots,\boldsymbol{s}_k \in S$ such that $\boldsymbol{p}=a_1 \boldsymbol{s}_1+ \cdots +a_k \boldsymbol{s}_k \text{ and } \sum_{i=1}^k a_i=1.$
	If $\boldsymbol{p} \in \R^{I}$, then
	\(\langle \boldsymbol{p},\boldsymbol{q}' \rangle =\sum_{i \in I}p_iq_i>d\) as $\boldsymbol{p} \not\in \Theta$. 
	If $\boldsymbol{p} \not\in \R^{I}$, then $\sum_{i \not\in I}p_i\geq \alpha.$ Thus,  
	\[ \langle \boldsymbol{p},\boldsymbol{q}'\rangle > \sum_{i \in I}p_iq_i+\left(\sum_{i \not\in I}p_i\right)\frac{d}{\alpha} \geq d.\]
	Therefore, $\Delta(\boldsymbol{q}';D)=\Theta$ and $d(\boldsymbol{q}';D)=d$.
	
	If $I_\Theta = [n]$ then $\Theta $ is an inner face and we can take $\boldsymbol{w}=\boldsymbol{q}$. Suppose that $I_\Theta \not= [n]$. Let $q'_{\min}=\min_{i \in I_{\Theta}^{\,c} } \{q'_i\}  $. For $r \in [0, q'_{\min})$, define $\boldsymbol{q}'(r)= (q'_1(r), \dots, q'_n(r))$, where $q'_i(r)= q_i$ if $i \in I_\Theta$ and $q'_i(r) = q'_i - r$ if $i \not\in I_\Theta$.
	Define the following continuous function $\tilde{d}: [0, q'_{\min}) \cap \mathbb{Q} \rightarrow \R_{\geq0} $, $\tilde{d}(r)= d(\boldsymbol{q}'(r) ;D).$ Observe that $\tilde{d}(0)= d$, $\tilde{d}$ is non-increasing, and for all $r>0$ small enough $\tilde{d}(r)=d$ and $\tilde{d}(q'_{\min}-r)<d$. 
	There is a unique continuous extension $\tilde{\tilde{d}} $ of $\tilde{d}$ to the real numbers, clearly $\tilde{\tilde{d}}(r) = \min \{\langle \boldsymbol{q}'(r) ,\boldsymbol{\nu} \rangle \mid \boldsymbol{\nu} \in D\}$. Let $r_1 := \max \{r \in [0, q'_{\min}) \mid \tilde{\tilde{d}}(r)=d \}$. Since $\tilde{\tilde{d}}(r_1) = d$ is a rational number and $\boldsymbol{\nu} \in (\mathbb{Q}_{>0})^n$, it follows that $r_1 \in \mathbb{Q}$, which implies that $\boldsymbol{q}'_1:=\boldsymbol{q}'(r_1)$ has rational components. Hence $d(\boldsymbol{q}'_1; D)=d$ and clearly $q'_i = q'_{1,i}$ for $i \in I_\Theta$ and $q'_i > q'_{1,i}$ for $i \in I_{\Theta}^{c} $.
	
	Let $\Theta_1 = \Delta(\boldsymbol{q}'_1; D) $. Then $I_{\Theta} \subsetneq I_{\Theta_1}$. In fact, if $I_\Theta = I_{\Theta'}$, then $\Theta = \Theta_1$, and for all vertex $\boldsymbol{\nu}$ of $D \setminus \Theta$, $\langle \boldsymbol{q}(r_1); \boldsymbol{\nu} \rangle > d$, then for rational $\epsilon>0$ small enough $\langle \boldsymbol{q}(r_1+ \epsilon); \boldsymbol{\nu} \rangle > d$, this implies that $d(\boldsymbol{q}'(r_1+\epsilon); D)=d$, hence $\tilde{d}(r_1+\epsilon) =d$, which is a contradiction. Clearly $\Theta = \Theta_1 \cap \R^{I_\Theta}$.
	
	By recursion, we obtain $\boldsymbol{q}'_1,\boldsymbol{q}'_2, \dots, \boldsymbol{q}'_l \in (\mathbb{Q}_{>0})^n$ such that $\Theta_i:=\Delta(\boldsymbol{q}'_i ; D) $, $\Theta_{k} = \Theta _{k+1}\cap \R^{I_{\Theta_k}}$, $I_{\Theta_k} \subsetneq \Theta_{k+1}$, $k=1, \dots, l-1$ and $I_{\Theta_l} = [n]$. Then $\Theta'=\Theta_l$ is an inner face such that $\Theta = \Theta' \cap \R^{I_\Theta}$. Taking $\boldsymbol{w}= \boldsymbol{q}_l$ we have that $q'_i = w_{i}$ for $i \in I_\Theta$ and $q'_i > w_{i}$ for $i \in I_\Theta^c$.
\end{proof}
\begin{theorem} \label{Th:Main1}
	Let \( F: \K^n \times \K \to \K^p \) be a deformation of a real polynomial map or mixed polynomial map \( f \). If there exist $C$-face diagrams $D_1, \dots, D_p$ satisfying:
	\begin{itemize}
		\item[(i)] the map \( f \) is (S)IKND with respect to the \(C\)-face diagrams \( \{D_j\}_{j=1,\dots,p} \),
		\item[(ii)] for each $j=1,\dots,p,$ and $\varepsilon$, \(d(\boldsymbol{w}; F^{\, j}_\varepsilon-f^j) \geq d(\boldsymbol{w}; D_j), \text{ for all }  \boldsymbol{w} \in \mathcal{P}(D_j).\)\end{itemize}
        Then, the deformation $F$ has (no coalescing) weak no coalescing of critical points along some neighborhood $K \subset \K$ of the origin. If the inequalities in (ii) are strict, then the deformation has (no coalescing) weak no coalescing of critical points. 
\end{theorem}
\begin{proof}
	By \cite[Proposition~3.1]{chen_tibar_2012} there is a connected compact set $K \subset \K$ containing the origin such that $F_\varepsilon $ is IKND with respect to the \(C\)-face diagrams \(D_1,\dots,D_p\) for each $\varepsilon \in K $. Suppose that $\{F_\varepsilon\} $ has coalescing of critical points then there is a sequence of $(\boldsymbol{x}_k,\varepsilon_k) \in \R^n \times K$ such that, $\boldsymbol{x}_k \rightarrow 0$, $\varepsilon_k \rightarrow \varepsilon_0 $ when $k \rightarrow \infty$ and
	$
	\mathrm{rank} \begin{pmatrix}
		J_{F_{\varepsilon_k}}(\boldsymbol{x}_k)
	\end{pmatrix} < p,
	$
	which holds if and only if there exist \( (\lambda_{k,1}, \dots, \lambda_{k,p}) \in \mathbb{R}^{p} \setminus \{0\} \) such that 
	$
	\sum_{j=1}^p \lambda_{k,j} \nabla F^{\, j}_{\varepsilon_k}(\boldsymbol{x}_k) = 0.
	$    
	Hence, by the curve selection lemma, there exist curves \( (\gamma, \varepsilon): [0, \epsilon) \to \mathbb{R}^n \times [0,1] \) and \( \lambda_j: [0, \epsilon) \to \mathbb{R},\, j=1,\dots,p,\) such that \( (\gamma(s), \varepsilon(s)) \to (0, \varepsilon_0) \) as \( s \to 0 \), and
	\begin{equation}\label{eq: sing}
		\sum_{j=1}^p \lambda_j(s) \nabla F^{\, j}_{\varepsilon(s)}(\boldsymbol{\gamma}(s)) =0.
	\end{equation}
	
	We can write $\varepsilon(s)$ and $\lambda_j(s),\ j=1,\dots,p,$ as
	\begin{align*}
		\varepsilon(s)&=\varepsilon_0 + \varepsilon_1 s^r + \text{h.o.t.},  \text{ and } \ 
		\lambda_j(s) = c_j s^{\beta_j} + \text{h.o.t.}, 
	\end{align*}   
	where, $r, \beta_j \in \N$ and $ (c_1,\dots,c_p) \neq 0$. Let $I:= \{i \in [n]  \mid \gamma_i \not\equiv 0 \}$. If $i \in I$, then we can write $\gamma_i (s)$ as 
	$$\gamma_i(s)=x_{0,i} s^{q_i} + \text{h.o.t.},$$
	and if $i \not\in I$, define $q_i=0$ and $x_{0,i}=0$. Set $\boldsymbol{q} = (q_1,\dots,q_n)$, $\boldsymbol{x}_0=(x_{0,1},\dots, x_{0,n})$, $d:=d(\pi^I(\boldsymbol{q}); \pi^I(D\cap \R^I))$, and  $\Theta:=(\pi^I)^{-1}(\Delta(\pi^I(\boldsymbol{q}); \pi^I(D\cap \R^I)))\cap \R^I$.
	
	By Lemma \ref{lemma: find inner} there exist $\boldsymbol{q'},\boldsymbol{w} \in \mathbb{Q}_{>0}$ such that $q'_i=w_i=q_i$ for $i \in I_\Theta $, $q_i'>w_i$ for $i \in I_\Theta^c$, $\Theta':= \Delta( \boldsymbol{w} ; D)$ is an inner face and $\Theta =\Delta(\boldsymbol{q}' ; D)=\Theta' \cap \R^{I_\Theta}$.
	
	Let $d_j= d(\boldsymbol{w}; D_j)$, $\Theta_j = \Delta( \boldsymbol{q}' ; D_j) $  and $\Theta'_j = \Delta (\boldsymbol{w}; D_j) $ for $j=1,\dots,p$. Then $\Theta = \Theta_1+ \cdots+ \Theta_p $ and $\Theta' = \Theta'_1+ \cdots+ \Theta'_p $. Hence, $(F_\varepsilon)_\Theta  = ((F_{\varepsilon}^{\,1} )_{\Theta_1}, \dots, (F_{\varepsilon}^{\,p} )_{\Theta_p})$, $(F_\varepsilon)_{\Theta'}  = ((F_{\varepsilon}^{\,1} )_{\Theta'_1}, \dots, (F_{\varepsilon}^{\,p} )_{\Theta'_p})$. Define $\boldsymbol{x}_1=(x_{1,1}, \dots, x_{1,n})$ where $x_{1,i}=x_{0,i}$ for $i \in I_\Theta$ and $x_{1,i}=0$ for $i \in I_\Theta^c$.
	\begin{claim}\label{claimexpansion}
		For all $i \in [n]$ and $j \in [p]$:
		\begin{align}
			\frac{\partial F_{\varepsilon(s)}^{\,j} }{ \partial x_i} (\gamma (s)) 
			= \frac{\partial (F_{\varepsilon_0}^{\,j} )_{\Theta'_j} }{ \partial x_i } (\boldsymbol{x}_1) s^{d_j-w_i} +  \text{h.o.t.} \label{eq: 9}
		\end{align}  
	\end{claim}
	\begin{proof}
		
		Let $i \in [n]$ and $j \in [p]$. Suppose that $\frac{\partial F_{\varepsilon}^{\,j} }{\partial x_i} \not\equiv 0$. Let $h$ be a monomial of $ F^{\, j}_\varepsilon$ such that $\frac{\partial h}{\partial x_i} \not\equiv 0$. 
		
		\underline{Case I:} If $h$ is a monomial of $F^{\, j}_\varepsilon -(F_{\varepsilon}^{\,j} )_{\Theta_j'}$,
		$$d\left( \boldsymbol{q}' ; \frac{\partial h}{\partial x_i} \right) \geq d \left( \boldsymbol{w} ; \frac{\partial h}{\partial x_i}  \right) = d(\boldsymbol{w} ; h) - w_i > d_j - w_i. $$
		The first inequality follows from Lemma \ref{lemma: find inner}, since $q_k \geq w_k$ for all $k=1,\dots,n.$
		
		\underline{Case II:} If $h= f^{\, j}_{\boldsymbol{\nu}}$, where $\boldsymbol{\nu} \in S_1 := \left\lbrace  \boldsymbol{\nu} \in \text{supp}( F_{\varepsilon}^{\,j} ) \cap \Theta'_j: \frac{\partial \boldsymbol{x}^{\boldsymbol{\nu}}}{\partial x_i}(\boldsymbol{x}_1) = 0  \right\rbrace$. Then, there is $k \in I_{\Theta}^c$ such that $x_k$ is a factor of $\frac{\partial h}{\partial x_i}$. Hence, by Lemma \ref{lemma: find inner}, $q_k > w_k$ and 
		$$d \left( \boldsymbol{q}' ;  \frac{\partial h}{\partial x_i} \right)> d \left( \boldsymbol{w} ;  \frac{\partial h}{\partial x_i} \right)= d \left( \boldsymbol{w} ; h \right)  - w_i = d_j - w_i.$$
		
		\underline{Case III:} If $h= f^{\, j}_{\boldsymbol{\nu}}$, where $\boldsymbol{\nu} \in S_2 := \left\lbrace  \boldsymbol{\nu} \in \text{supp}( F_{\varepsilon}^{\,j} ) \cap \Theta'_j: \frac{\partial \boldsymbol{x}^{\boldsymbol{\nu}}}{\partial x_i}(\boldsymbol{x}_1) \not= 0  \right\rbrace$. Then, $\frac{\partial h}{\partial x_i}$ only depends on variables indexed by $I_{\Theta}$. Hence $\frac{\partial h}{\partial x_i} (\boldsymbol{x}_0) = \frac{\partial h}{\partial x_i}(\boldsymbol{x}_1)$ and by Lemma \ref{lemma: find inner}
		$$d \left( \boldsymbol{q}' ;  \frac{\partial h}{\partial x_i} \right) = d \left( \boldsymbol{w} ;  \frac{\partial h}{\partial x_i} \right) = d_j - w_i. $$
		
		Hence $d\left(\boldsymbol{q}' ; \frac{\partial F^{\, j}_\varepsilon}{\partial x_i}\right) \geq d_j - w_i$. Note that $( F_{\varepsilon}^{\,j} )_{\Theta'_j} = ( F_{\varepsilon}^{\,j} )_{S_2} +( F_{\varepsilon}^{\,j} )_{S_1} $. If $S_2 = \emptyset$, then $\frac{\partial ( F_{\varepsilon}^{\,j} )_{\Theta'_j}}{\partial x_i} (\boldsymbol{x}_1) =0$ and $ d\left(\boldsymbol{q}' ; \frac{\partial F^{\, j}_\varepsilon}{\partial x_i}\right) > d_j - w_i$. If $S_2 \not= \emptyset$ then $$\left(\frac{\partial F^{\, j}_\varepsilon}{\partial x_i}\right)_{\boldsymbol{q}'} (\boldsymbol{x}_0) = \frac{\partial ( F_{\varepsilon_0}^{\,j} )_{S_2}}{\partial x_i}(\boldsymbol{x}_0) =  \frac{\partial ( F_{\varepsilon_0}^{\,j} )_{S_2}}{\partial x_i}(\boldsymbol{x}_1) = \frac{\partial ( F_{\varepsilon}^{\,j} )_{S_2}}{\partial x_i}(\boldsymbol{x}_1) + \frac{\partial ( F_{\varepsilon}^{\,j} )_{S_1}}{\partial x_i}(\boldsymbol{x}_1) = \frac{\partial ( F_{\varepsilon}^{\,j} )_{\Theta'_j}}{\partial x_i}(\boldsymbol{x}_1).$$ 
	\end{proof}
	Let $m :=\min_{j\in[p] } \{d_j+\beta_j\} $ and  $J := \left\lbrace j \in [p] \mid d_j + \beta_j = m \right\rbrace$. Then, combining \eqref{eq: sing} and \eqref{eq: 9} we obtain 
	\begin{align*}
		\sum_{j \in J} c_j \frac{\partial ( F_{\varepsilon_0}^{\,j} )_{\Theta'_j}}{\partial x_i} (\boldsymbol{x}_1) s^{m-w_i} + \text{h.o.t.}= 0, \text{ for all } i = 1,\dots ,n.
	\end{align*}
	Hence, $\sum_{j \in J} c_j \nabla ( F_{\varepsilon_0}^{\,j} )_{\Theta'_j}(\boldsymbol{x}_1) =0. $ Therefore $\boldsymbol{x}_1 \in \Sigma((F_{\varepsilon_0})_{\Theta'}) \cap (\R^*)^{I_{\Theta}}$, which is a contradiction since $(F_{\varepsilon_0})$ is IKND, and then $F$ has no coalescing of critical points. 
	
	If we consider the condition $ F(\gamma(s), \varepsilon(s)) = 0 $ in the hypothesis of contradiction, using Claim \ref{claim:4} we get:
	\begin{align*}
		0= F(\gamma(s), \varepsilon(s)) &= (( F_{\varepsilon_0}^{\,1} )_{\Theta_1}(\boldsymbol{x}_0) s^{d_1}+\text{h.o.t.} ,\cdots,( F_{\varepsilon_0}^{\,p} )_{\Theta_p}(\boldsymbol{x}_0) s^{d_p} + \text{h.o.t.}   )\\
		&= (( F_{\varepsilon_0}^{\,1} )_{\Theta'_1}(\boldsymbol{x}_1) s^{d_1}+ \text{h.o.t.},\cdots,( F_{\varepsilon_0}^{\,p} )_{\Theta'_p}(\boldsymbol{x}_1) s^{d_p}+\text{h.o.t.}).
	\end{align*}
	Therefore, $\boldsymbol{x}_1 \in V((F_{\varepsilon_0})_{\Theta'})$, which produces a contradiction.
    
    If the inequalities in (ii) are strict, then \(F_\varepsilon\) is (S)IKND with respect to the \(C\)-face diagrams \(D_1,\dots,D_p\) for all \(\varepsilon \in \mathbb{K}\). In this case, we may choose any compact connected set \(K \subset \mathbb{K}\) containing the origin and carry out the same proof.
\end{proof}
Observe that each \( F_\varepsilon \) (resp. \( V(F_\varepsilon) \)) has an isolated singularity if \( f \) is SIKND (resp. if \( f \) is IKND)\footnote{In \cite{BodeSanchez2023}, the authors study, for mixed polynomial functions, different conditions that imply a (weak) isolated singularity (see \cite[Fig. 1]{BodeSanchez2023}), as well as the relationships between them.}. Therefore, by Milnor's result \cite{Milnor1968}, the link  
\[
L_{F_\varepsilon, r_\varepsilon} := V(F_\varepsilon) \cap \mathbb{S}_{r_\varepsilon}
\]  
is well-defined for all sufficiently small \( \rho_\varepsilon > 0 \), where \( \mathbb{S}_{r_\varepsilon} \) denotes the standard sphere of radius \( r_\varepsilon \) centered at the origin in \( \K^n \). Since the link does not depend on $r_\varepsilon>0$ being small enough, we denote the link of $F_\varepsilon$  by $L_{F_\varepsilon}$.  
However, in general, there is no reason to expect that the links \( L_{F_\varepsilon, r_\varepsilon} \) are equivalent as \( \varepsilon \) varies.
\begin{remark}
	It is important to note that SIKND maps constitute a wider class than convenient SKND maps, as SIKND includes maps that are KND but not SKND, as well as maps that are neither KND nor SKND.
\end{remark}

\section{Topological triviality and link-constancy in deformations}\label{section4}
In this section, we study invariants associated with the deformations introduced in the previous section. Our goal is to determine when the deformation is topologically trivial or when the corresponding links are ambient isotopic.
Let $F:\K^n \times \K \rightarrow \K^p $ be a deformation of a real polynomial map (or mixed polynomial map) $f$ with $F(\boldsymbol{x},0) = f(\boldsymbol{x})$ and  $F(0,\varepsilon) = 0$.

\begin{definition} \label{def:link-const}
	Let $K \subset \K$ be a connected set containing the origin. Suppose that for all $\varepsilon \in K $ we have that $F_\varepsilon$ has a well-defined link (in the sense of Milnor \cite{Milnor1968}). We say that the deformation $F$ is \textbf{link-constant along $K$} if for all $\varepsilon \in K$, $L_{F_\varepsilon}$ is ambient isotopic to $L_{F_0}$, that is, there exists a continuous map	$
	H : \mathbb{S}_\epsilon \times [0,1] \to \mathbb{S}_\epsilon$, where $L_{F_0,\epsilon}$ and  $L_{F_\varepsilon,\epsilon}$ are well-defined for sufficiently small $\epsilon>0$, 
	such that for each $t \in [0,1]$, the map $H_t := H(\cdot, t)$ is a homeomorphism,
	$H_0 = \text{id}$, and
	$
	H_1(L_{F_0}) = L_{F_\varepsilon}.
	$    If $K = \K$, we simply say that $F$ is link-constant.
\end{definition}
\begin{definition}
	Let $K \subset \K $ be a connected set containing the origin. We say that the deformation $F$ is \textbf{topologically trivial along $K$} if there is continuous family of homeomorphism germs $h_\varepsilon : (\K^n,0) \rightarrow (\K^n,0)$ such that $F_0 = F_\varepsilon \circ h_\varepsilon$ for all $\varepsilon \in K$. If $K = \K $, we simply say that the deformation is topologically trivial.
\end{definition}
We present two methods to determine the link-constancy of a deformation. The first one relies on the implicit function theorem, as used in Theorems \ref{Th:Main2}. The second method, Proposition~\ref{Prop:GD}, consists in proving the existence of a $\rho$-uniform Milnor radius, as done in Theorem~\ref{teo: knd co -> urm} and Theorem~\ref{th:nonice}.

Let $U$ be an open subset of $\mathbb{R}^n$ containing the origin, and let $\rho : U \rightarrow \mathbb{R}_{\geq 0}$ be a proper, nonnegative, real analytic function such that $\rho^{-1}(0) = \{0\}$.

\begin{definition}
	Let $K$ be a connected compact subset of $\K$ containing the origin. We say that the deformation $F$ has \textbf{$\rho$-uniform Milnor radius along $K$} if there is $\epsilon_0>0 $ such that for all $ 0<\epsilon \leq \epsilon_0$ and $\varepsilon \in K$, $V(F_\varepsilon)$ is transversal to $\rho^{-1}(\epsilon)$. If $F$ has $\rho$-uniform Milnor radius along any connected compact subset $K$ containing the origin, we simply say that $F$ has a $\rho$-uniform Milnor radius.
\end{definition}
\begin{proposition} \label{Prop:GD}
	Let $K$ be a connected compact subset of $\K$ containing the origin. If the deformation $F$ has $\rho$-uniform Milnor radius along $K$, then for any small enough $\epsilon>0$, there exists, $0<\eta \ll \epsilon $, such that
	\begin{equation} \label{eq: fibration} ({F} \times \mathrm{id})_| : (\rho^{-1}(\epsilon) \times K) \cap ({F} \times \mathrm{id})^{-1} (B_\eta^p \times K)  \rightarrow B_\eta^p \times K
	\end{equation}  
	is a locally trivial smooth fibration, where $({F} \times \mathrm{id})(\boldsymbol{x},\varepsilon) = (F_\varepsilon(\boldsymbol{x}),\varepsilon)$. In particular, $F$ is link-constant along $K$.
\end{proposition}
\begin{proof}
	Let $\epsilon_0>0 $ be such that for all $ 0<\epsilon < \epsilon_0$ and $\varepsilon \in K $, $F_{\varepsilon}^{\,-1} (0) $ is transverse to $\rho^{-1}(\epsilon)$. We have that 
	\[J_{({F} \times \mathrm{id})_|}(\boldsymbol{x},\varepsilon) = 
		\begin{pmatrix}
			J_{({F} \times \mathrm{id})}(\boldsymbol{x},\varepsilon) \\
			\begin{array}{cc}
				\nabla \rho (\boldsymbol{x})   &  0
			\end{array}
		\end{pmatrix} =\begin{pmatrix}
			J_{F_\varepsilon}(\boldsymbol{x}) & \partial F / \partial \varepsilon \\
			0 & 1 \\
			\nabla \rho (\boldsymbol{x}) & 0
		\end{pmatrix}.\]
	Then, for all $\varepsilon \in K$, $F_{\varepsilon}^{\,-1}(0)$ is transverse  $\rho^{-1}(\epsilon)$ if and only if $({F} \times \mathrm{id})^{-1}(0,\varepsilon)$ is transverse to $ \rho^{-1}(\epsilon) \times K$.
	Since transversality is an open property, for each $(0,\varepsilon) \in \{0\} \times K$, there exists $\eta_\varepsilon > 0$ such that for all $\boldsymbol{y} \in B^{p+1}_{\eta_\varepsilon}(0,\varepsilon) \subset \R^p \times K$, we have $$
	({F} \times \mathrm{id})^{-1}(\boldsymbol{y} ) \pitchfork_{\boldsymbol{y}} \rho^{-1}(\epsilon) \times K.
	$$
    
	By the compactness of $K$, there exists $\varepsilon_1, \dots, \varepsilon_k \in  K$ such that $\{0\}\times K$ is covered by $ \cup_{i=1}^k B^{p+1}_{\eta_{\varepsilon_i}} (0,\varepsilon_i)$. Let $\eta > 0$ be the Lebesgue number of this cover. Then, $B^p_\eta(0) \times K \subset \cup_{i=1}^k B^{p+1}_{\eta_{\varepsilon_i}}(0,\varepsilon_i)$. Thus, $({F} \times \mathrm{id})_| $ is a submersion on $({F} \times \mathrm{id})^{-1}(B_\eta^p \times K)$ and by Ehresmann's fibration theorem we have that \eqref{eq: fibration} is a locally trivial smooth fibration.
	
	For the second part, fix $\varepsilon \in K \setminus \{0\}$ note that \eqref{eq: fibration} gives a smooth isotopy between $\rho^{-1}(\epsilon) \cap V(F_0)$ and $\rho^{-1}(\epsilon) \cap  V(F_\varepsilon)$ in $\rho^{-1}(\epsilon) $, by the isotopy extension theorem \cite{Hirsch1976} we obtain the ambient isotopy $H: \rho^{-1}(\epsilon) \times [0,1] \rightarrow \rho^{-1}(\epsilon) $. From this, we will construct an ambient isotopy in the sphere between the link $L_{F_0} $ and $L_{F_\varepsilon}$. Denote $\rho_E: \K^n \rightarrow \R$ square of the distance function, $\rho_E(\boldsymbol{x}) = ||\boldsymbol{x} ||^2 $. 
	
	Let $\mathcal{S}$ be a Whitney stratification of $V= V(F_0F_\varepsilon)$.
	Let $\epsilon_0>0$ such that each strata of $V \cap B_{\epsilon_0}^*$ is transversal to $\rho_E|_{B_{\epsilon_0}^*}$ and $\rho|_{B_{\epsilon_0}^*}$, where $B_{\epsilon_0}^* = B_{\epsilon_0} \setminus \{0\}$. It is possible to refine the stratification $\mathcal{S}$ of $V$ in such a way that $\mathcal{S}' = \mathcal{S} \cup \{ B_\varepsilon \setminus V \}$ defines a Whitney stratification of the  $B_\varepsilon$. Shrinking $\epsilon_0 $ if necessary, we can suppose that no $\boldsymbol{x} \in B_{\epsilon_0}^
	* \setminus \{0\}$, $\nabla \rho (\boldsymbol{x})$ and $\nabla \rho_E (\boldsymbol{x})$ point in opposite directions (see. \cite[Proposition~2.5]{zzLooijenga_1984}).
	
	For $\epsilon_0>0$ small enough, we have that for all $(a,b) \in (\R_{\geq 0})^2 \setminus \{0\}$ each strata of $V(F_0 F_\varepsilon) \setminus \{0\}$ is transversal to $a\rho_E + b \rho $. Indeed, suppose that for some strata $X$ in $\mathcal{S}$ we have for all $n \in \N$ there is a point $x_{n} \in B_{1/n}^* \cap X$ such that $X$ is not transversal to $a \rho_E + b \rho $ at $x_{n}$, then by curve selection lemma there is an analytic curve $\gamma:[0,\eta ) \rightarrow \K^n$, with $\gamma (0)=0$ and $\gamma((0,\eta)) \subset X$, such that $X$ is not transversal to $a \rho_E + b \rho $ in $\gamma(t)$, $t \in (0,\eta)$. Note that $a\nabla \rho_E(\gamma(t)) + b \nabla \rho (\gamma(t)) \not= 0$. Since $ \frac{d}{dt}(\rho_E \circ \gamma)(t) =  \langle \nabla \rho_E (\gamma(t)), \gamma'(t) \rangle > 0$ and $ \frac{d}{dt}(\rho \circ \gamma)(t) =  \langle \nabla \rho (\gamma(t)), \gamma'(t) \rangle > 0$, $ \langle a\nabla \rho_E(\gamma(t)) + b \nabla \rho (\gamma(t)) , \gamma'(t)  \rangle > 0$, $t \in (0, \eta)$. Hence, $a \nabla \rho_E(\gamma(t)) + b \nabla \rho (\gamma(t))$ is not in the normal space of $X$ which implies that $a \rho_E + b \rho$ is transverse to $X$ at $\gamma(t)$, $t \in (0,\eta)$. 
	
	Let $0< \delta'< \delta < \epsilon_0$ be such that $\rho^{-1}([0,\delta']) \subset \rho_{E}^{\,-1}([0,\delta])$. Consider the open cover of $B_{\epsilon_0}$ given by $ \mathcal{U} = \{ U_1= \rho^{-1}((\delta',\infty)), U_2= \rho_{E}^{\,-1}([0,\delta)) \}$. Let $\{\psi_1 ,\psi_2\}$ be a partition of unity subordinate to $\mathcal{U}$ and define $\widehat{\rho} : B_{\epsilon_0}^* \rightarrow \R$, $\widehat{\rho} = \psi_1 \rho_E + \psi_2 \rho $. Since $(\psi_1(\boldsymbol{x}),\psi_2(\boldsymbol{x}) ) \in (\R_{\geq 0})^2$, then $\widehat{\rho}$ is tranversal to each strate $X$ of $\mathcal{S}$. Then, $\widehat{\rho}|_X$ is a submersion for each stratum $X$ of $\mathcal{S}$. Hence, by the Thom's first isotopy lemma we obtain a homeomorphism $h: \rho_{E}^{\,-1}(\delta) \rightarrow \rho^{-1}(\delta')$ such that $h(V \cap \rho_{E}^{\,-1}(\delta)) = V \cap \rho^{-1}(\delta')$. Taking $\epsilon = \delta$ we have that $$(h^{-1} \times \mathrm{id}) \circ H \circ h: \rho_{E}^{\,-1}(\delta) \times [0,1] \rightarrow \rho_{E}^{\,-1}(\delta)  $$
	is an ambient isotopy\footnote{If $V(F_0F_\epsilon) \setminus \{0\}$ is non-singular, $h$ is a diffeomorphism, and hence we obtain a smooth ambient isotopy.} between $L_{F_\varepsilon}$ and $L_{F_0}$.
\end{proof}
We establish, in the following two propositions, conditions for topological triviality by combining our main result, Theorem~\ref{Th:Main1}, with \cite[Corollary~1]{King1980} and subsequently with  \cite[Theorem~2.4]{zBekka2015}. In the first case, the no coalescing critical points is a sufficient condition to guarantee topological triviality. Recall that, in general, the fact that a deformation $F$ has no coalescing of critical points does not imply topological triviality (see \cite{King1980}). In the second case, no coalescing of critical points is combined with a condition of a $\rho$-uniform Milnor radius. These results in topological triviality extend the study explored in \cite{Damon1983,SaiaandRuas1997}.
\begin{proposition} \label{cor: top trivi}
	Let \( F: \K^n \times \K \to \K^p \) be a deformation of a polynomial map \( f \). Assume that $n-p\leq 2$ if $f$ is a real map and $n-p\leq 1$ if $f$ is a mixed polynomial map. If there exist $C$-face diagrams $D_1, \dots, D_p$ satisfying:
	\begin{itemize}
		\item[(i)] the map \( f \) is SIKND with respect to the \(C\)-face diagrams \( \{D_j\}_{j=1,\dots,p} \),
		\item[(ii)] for each $j=1,\dots,p,$ and $\varepsilon$, \(d(\boldsymbol{w}; F^{\, j}_\varepsilon -f^j) \geq d(\boldsymbol{w}; D_j), \text{ for all }  \boldsymbol{w} \in \mathcal{P}(D_j).\)\end{itemize}
	Then, the deformation is topologically trivial along some neighborhood $K \subset \K$ of the origin. If the inequalities in (ii) are strict, then the deformation is topologically trivial.
\end{proposition}
\begin{proof}
	By Theorem~\ref{Th:Main1} \(F\) has no coalescing of critical points along some compact connected subset \(K \subset \mathbb{K}\) containing the origin. If the inequalities in (ii) are strict, then \(F\) has no coalescing of critical points. Thus, the result follows by applying \cite[Corollary~1]{King1980}. The mixed case follows by an additional application of Lemma~\ref{lemma: IND M -> IKND}.
\end{proof}
\begin{proposition} \label{Prop:toptrivialwithrho}
	Let \( F: \K^n \times \K \to \K^p \) be a deformation of a real polynomial map (or mixed polynomial map) \( f \).  If there exist $C$-face diagrams $D_1,\dots,D_p$ satisfying:
	\begin{itemize}
		\item[(i)] the map \( f \) is SIKND with respect to the \(C\)-face diagrams \( \{D_j\}_{j=1,\dots,p} \),
		\item[(ii)] for each $j=1,\dots, p,$ and $\varepsilon$, \(d(\boldsymbol{w}; F^{\, j}_\varepsilon -f^j) \geq d(\boldsymbol{w}; D_j), \text{ for all }  \boldsymbol{w} \in \mathcal{P}(D_j)\),
		\item[(iii)] the deformation $F$ has a $\rho$-uniform Milnor radius.
	\end{itemize} 
    Then, the deformation is topologically trivial along some neighborhood $K \subset \K$ of the origin. If the inequalities in (ii) are strict, then the deformation is topologically trivial.
\end{proposition}
\begin{proof}
	By Theorem~\ref{Th:Main1} \(F\) has no coalescing of critical points along some compact connected subset \(K \subset \mathbb{K}\) containing the origin. If the inequalities in (ii) are strict, then \(F\) has no coalescing of critical points. Then, the result follows by applying \cite[Theorem~2.4]{zBekka2015}. The mixed case follows by an additional application of Lemma~\ref{lemma: IND M -> IKND}. 
\end{proof}

The following result describes certain deformations of IKND maps (including convenient KND maps) that are link-constant. Moreover, under the SIKND condition, these deformations are topologically trivial (see \cite{Raimundo2005,saia2008} for related results).
\begin{theorem} \label{teo: knd co -> urm}
	Let \( F: \K^n \times \K \to \K^p \) be a deformation of a real polynomial map (or mixed polynomial map) \( f \). Suppose that there exist $C$-face diagrams $D_1,\dots,D_p$ and a connected compact set $K \subset \K$ containing the origin such that:
	\begin{itemize}		
		\item[(i)] the map $f_D$ is KND, where $D=D_1+\cdots+D_p,$
		\item[(ii)] the map $f$ is IKND  with respect to the \(C\)-face diagrams \( \{D_j\}_{j=1,\dots,p} \),
      \item[(iii)]  for each $j=1,\dots, p,$ and $\varepsilon$, \(d(\boldsymbol{w}; F^{\, j}_\varepsilon-f^j) \geq d(\boldsymbol{w}; D_j), \text{ for all }  \boldsymbol{w} \in \mathcal{P}(D_j)\),
		\item[(iv)] For all $I \subset [n]$ and inner face $\Delta(\boldsymbol w;D), \ \boldsymbol w \in (\mathbb{Q}_{>0})^n$,  such that $\Delta(\boldsymbol w;D) \cap \R^{I} \not= 0$ and $(f_{\Delta(\boldsymbol w;D)})^I \equiv 0$ there is $i \in I$ satisfying 
		\(w_{i} \leq w_{j},\text{ for all } j.\)
	\end{itemize}
      Then, the deformation is link-constant along some neighborhood $K \subseteq \K$ of the origin. If the inequalities in (ii) are strict, then the deformation is link constant. Furthermore, if, in addition, $f$ is SIKND with respect to \( \{D_j\}_{j=1,\dots,p} \), then the deformation is topologically trivial along the respective set $K$ or $\K$.
\end{theorem}
\begin{proof}
    By (i), (ii) and \cite[Proposition~3.1]{chen_tibar_2012} there is a connected compact set $K \subset \K$ containing the origin such that $F_\varepsilon $ is IKND with respect to the \(C\)-face diagrams \(D_1,\dots,D_p\) for each $\varepsilon \in K $ and \((F_\varepsilon)_D\) is KND.
	By Proposition~\ref{Prop:GD}, the link-constancy follows by proving the $\rho$-uniform Milnor radius for some analytic function $\rho$ defining the origin. Suppose by contradiction that the deformation $F$ has no $\rho_E$-uniform Milnor radius, thus by applying the curve selection lemma we find:
	\begin{equation}\label{eq:zeros}
		F^{\, 1}_{\varepsilon(s)}(\gamma(s))=\cdots= F^{\, p}_{\varepsilon(s)}(\gamma(s))=0,
	\end{equation}
	\begin{equation}\label{eq: milnorset}
		\sum_{j=1}^p \lambda_j(s) \nabla F^{\, j}_{\varepsilon(s)}(\gamma(s)) =\lambda_{p+1}(s) \gamma(s),
	\end{equation}
	where $(\gamma(s), \varepsilon(s)) \to (0, \varepsilon_0) \text{ as }  s \to 0$, $ \lambda_j(s) = c_j s^{\beta_j} + \text{h.o.t.}, \ j=1,\dots,p+1,\ (c_1,\dots,c_{p+1}) \neq 0$, and 
	\begin{align*}\gamma_i(s)= \begin{cases}
			x_{0,i} s^{q_i} + \text{h.o.t.}, &\text{if } i \in I= \{i \in [n] \mid \gamma_i \not\equiv 0 \},\\
			0 , & \text{if } i \not\in I.
		\end{cases}
	\end{align*}
	Set $\boldsymbol{q} = (q_1,\dots,q_n)$  and  $\boldsymbol{x}_0=(x_{0,1},\dots, x_{0,n})$  with
	\begin{align*}q_i= \begin{cases}
			q_i, &\text{if } i \in I,\\
			0 , & \text{if } i \not\in I,
		\end{cases}
		\text{ and } x_{0,i}= \begin{cases}
			x_{0,i}, &\text{if } i \in I,\\
			0 , & \text{if } i \not\in I.
		\end{cases}
	\end{align*}
	By Lemma~\ref{lemma: find inner}, we can find positive weight vectors $\boldsymbol{q}'$ and $\boldsymbol{w}$ such that $\Theta := \Delta(\boldsymbol{q}';D) = \Theta_1+\cdots+\Theta_p$ and $\Theta' := \Delta(\boldsymbol{w}; D) = \Theta'_1+\cdots+\Theta'_p$ is an inner face of $D$. These faces satisfy $\Theta' \cap \mathbb{R}^{I_\Theta} = \Theta$. Furthermore, $d := d(\boldsymbol{q}'; D) = d(\boldsymbol{w}; D)$, where $d = d_1+\cdots+d_p$ and $d_j= d(\boldsymbol{q}'; D_j) = d(\boldsymbol{w}; D_j)$ for $j=1,\dots,p$. A key property of these weight vectors is that $w_i=q'_i=q_i$ for $i \in I_{\Theta}$ and $w_i < q_i'$ for $i \in I_{\Theta}^c$. 
	
	Define the point $\boldsymbol{x}_1$ by setting  $x_{1,i}=x_{0,i}$ for $i \in I_\Theta$ and $x_{1,i}=0$ for $i \in I_\Theta^c$.
	By Condition~(iii) applied in \eqref{eq:zeros}, we find that: 
	\begin{align*}
		F_{\varepsilon(s)}(\gamma(s)) &= (( F_{\varepsilon_0}^{\,1} )_{\Theta_1}(\boldsymbol{x}_0) s^{d_1}+\text{h.o.t.} ,\cdots,( F_{\varepsilon_0}^{\,p} )_{\Theta_p}(\boldsymbol{x}_0) s^{d_p} + \text{h.o.t.}   )\\
		&= (( F_{\varepsilon_0}^{\,1} )_{\Theta'_1}(\boldsymbol{x}_1) s^{d_1}+ \text{h.o.t.},\cdots,( F_{\varepsilon_0}^{\,p} )_{\Theta'_p}(\boldsymbol{x}_1) s^{d_p}+\text{h.o.t.}   ).
	\end{align*}
	
	Let $m :=\min \{d_j+\beta_j \mid j \in  I \}$ and $J := \left\lbrace j \mid d_j + \beta_j = m \right\rbrace$. By applying Claim~\ref{claimexpansion} in \eqref{eq: milnorset} we obtain 
	\begin{align}\label{eq:unif.Milnor}
		\sum_{j \in J} c_j \frac{\partial ( F_{\varepsilon_0}^{\,j} )_{\Theta'_j}}{\partial x_i} (\boldsymbol{x}_1) s^{m-w_i} +\text{h.o.t.} = c_{p+1} x_{0,i} s^{q_i + \beta_{p+1}} + \text{h.o.t.}, \quad \text{for all } i.
	\end{align}
	Set $I' := \{i \mid m- \beta_{p+1} = q_i + w_i \}$. By Theorem~\ref{Th:Main1}, $F$ has weak no coalescing of critical points, thus $c_{p+1} \not=0$, $m- \beta_{p+1} \leq q_i + w_i$ for all $i$, and $I' \not= \emptyset$. We consider two cases based on whether $I_\Theta$ intersects $I'$ or not: 
	
	\underline{Case I: $I' \cap I_\Theta \not= \emptyset $}. By \eqref{eq:unif.Milnor} we get: 
	\begin{equation}\label{eq:case11}
		\sum_{j \in J} c_j  \frac{\partial ( F_{\varepsilon_0}^{\,j} )_{\Theta'_j}}{\partial x_i}(\boldsymbol{x}_1) = c_{p+1}x_{0,i}, \text{ for all } i \in I' \cap I_{\Theta}
	\end{equation}
	and 
	\begin{equation}\label{eq:case12}
		\sum_{j \in J} c_j  \frac{\partial ( F_{\varepsilon_0}^{\,j} )_{\Theta'_j}}{\partial x_i}(\boldsymbol{x}_1) = 0, \text{ for all } i \in I_{\Theta} \setminus I'.
	\end{equation}
	For $j \in J$ we have the Euler's identities: 
	$$d_j ( F_{\varepsilon_0}^{\,j} )_{\Theta'_j}(\boldsymbol{x}_1)=\sum_{i \in I_\Theta}  \frac{\partial ( F_{\varepsilon_0}^{\,j} )_{\Theta'_j}}{\partial x_i}(\boldsymbol{x}_1) x_{1,i} w_i.$$
	Given that $\boldsymbol{x}_1 \in V((F_{\varepsilon_0})_{\Theta'})$ and $x_{1,i}=x_{0,i}$ for $i \in I_{\Theta}$, we can apply Euler's identities along with \eqref{eq:case11} and \eqref{eq:case12} to derive:
	\begin{align*}
		0&=\sum_{j \in J }c_jd_j ( F_{\varepsilon_0}^{\,j} )_{\Theta'_j}(\boldsymbol{x}_1) = \sum_{j \in J} c_j \sum_{i \in I_\Theta}  \frac{\partial ( F_{\varepsilon_0}^{\,j} )_{\Theta'_j}}{\partial x_i}(\boldsymbol{x}_1) x_{1,i} w_i
		=\sum_{i \in I' \cap I_\Theta} c_{p+1} (x_{1,i})^2 w_i. 
	\end{align*} 
	This leads to a contradiction, since the right side of the equation is nonzero.
	
	\underline{Case II:  $I' \cap I_\Theta = \emptyset $}. Then $m- \beta_{p+1} < q_i + w_i$ for $i \in I_{\Theta}$. By \eqref{eq:unif.Milnor} we get:
	\begin{align}\label{eq:unif.Milnor2}
		\sum_{j \in J} c_j \frac{\partial ( F_{\varepsilon_0}^{\,j} )_{\Theta'_j}}{\partial x_i} (\boldsymbol{x}_1)=0, \quad \text{for all } i \in I_{\Theta}.
	\end{align}
	If $(F_{\varepsilon_0})_{\Theta} \equiv 0$, by Condition~(iv), there is $i \in I_\Theta$ such that $w_i \leq w_k,\text{ for all } k,$ thus by Lemma~\ref{lemma: find inner}, 
	$w_i+q'_i=2w_i\leq 2w_k\leq w_k+q'_k \text{ for all } k$, which implies $I' \cap I_\Theta \not= \emptyset$, which is a contradiction. We can suppose that $(F_{\varepsilon_0})_{\Theta} \not\equiv 0$, note that $(F_{\varepsilon_0})_{\Theta} = ((F_{\varepsilon_0})_D)_{\boldsymbol{q}'}.$

	Thus, for $\boldsymbol{x}_2$ defined by setting $x_{2,i}=x_{1,i}$ for $i \in I_\Theta$ and  $x_{2,i}=1$ for $i \not \in I_\Theta$, using \eqref{eq:unif.Milnor2} and Claim \ref{claim:4} we find:  
	$$  \sum_{j \in J} c_j \frac{\partial (F_{\varepsilon_0}^j)_{\Theta_j}}{\partial x_i}(\boldsymbol{x}_2)= \sum_{j \in J} c_j \frac{\partial (F_{\varepsilon_0}^j)_{\Theta'_j}}{\partial x_i}(\boldsymbol{x}_1)=0, \text{ for all } i \in I_{\Theta}.$$ This implies that $\boldsymbol{x}_2 \in \Sigma(((F_{\varepsilon_0})_D)_{\boldsymbol{q}'}) \cap V(((F_{\varepsilon_0})_D)_{\boldsymbol{q}'}) \cap (\K^*)^n$. Hence, $(F_{\varepsilon_0})_D$ is not KND, which contradicts Condition~(i).   
	
	Therefore, the deformation $F$ has $\rho_E$-uniform Milnor radius along $K$. 

    If the inequalities in (iii) are strict, then \(F_\varepsilon\) is (S)IKND with respect to the \(C\)-face diagrams \(D_1,\dots,D_p\) and $(F_\varepsilon)_D $ is KND for all \(\varepsilon \in \mathbb{K}\). In this case, we may choose any compact connected set \(K \subset \mathbb{K}\) containing the origin and carry out the same proof.

	The topological triviality follows directly by applying Proposition~\ref{Prop:toptrivialwithrho}.
\end{proof}
\begin{corollary}\label{cor:convenient}
	Let \( F: \K^n \times \K \to \K^p \) be a deformation of a real polynomial map (or mixed polynomial map). If the following conditions hold:
	\begin{enumerate}
		\item[(i)] \(f\) is convenient and KND,
		\item[(ii)] for each $j=1,\dots,p,$ and \(\varepsilon\),   \(d(\w;F^{\, j}_\varepsilon-f^j) \geq d(\w;f^j), \text{ for all } \w \in \mathcal{P}(f^j).\)
	\end{enumerate}
      Then, the deformation is link-constant along some neighborhood $K \subset \K$ of the origin. If the inequalities in (ii) are strict, then the deformation is link-constant. Furthermore, if, in addition, $f$ is SIKND with respect to $\{\Gamma(f^j)\}_{j=1,\dots,p}$, then the deformation is topologically trivial along the set $K$ or $\K$.
\end{corollary}
\begin{proof}
	The results follow directly from Theorem~\ref{teo: knd co -> urm} and Proposition~\ref{prop: knd co -> iknd}.
\end{proof}
\begin{proposition}\label{prop:defkndsurfaces}
	Let $f: \K^3 \to \K$ be a KND real polynomial function (or mixed polynomial function) as Proposition~\ref{prop:kndsurfaces}. Then, there exists $\boldsymbol{v}\in \mathbb{Q}$ such that any deformation $F$ of $f$ is link-constant, along some neighborhood $K \subset \K$ of the origin, whenever, \(d(\boldsymbol{w}; F_\varepsilon-f) \geq d(\boldsymbol{w}; D_{\boldsymbol{v}}(f)), \text{ for all }  \boldsymbol{w} \in \mathcal{P}(D_{\boldsymbol{v}}(f)).\)
\end{proposition}
\begin{proof}
	By Theorem~\ref{teo: knd co -> urm}, it is sufficient to demonstrate the existence of $\boldsymbol{v} \in \mathbb{Q}^3$ such that the $C$-face diagram $D_{\boldsymbol{v}}(f)$ and the deformation $F$ satisfy Conditions~(i-iv) of Theorem~\ref{teo: knd co -> urm}.
	The verification proceeds as follows: Conditions~(i) and (iii) are satisfied directly by the initial hypotheses that define $f$ and $F_\varepsilon-f$, respectively. Condition~(ii) is guaranteed by choosing $\boldsymbol{v}$ as Lemma~\ref{lemma:diagramnonconvcase} and applying Proposition~\ref{prop:kndsurfaces}. To prove Condition~(iv), note that for $I\neq\{i\}, \ i\in I_{\rm{nc}}(f),$ and an inner face $\Delta$ with $\Delta \cap \R^I \neq \emptyset$, we always have \((f_{|{\K^I}})_\Delta \not \equiv 0\). Thus, consider $I=\{i\}, \ i \in I_{\rm{nc}}(f)$. By Lemma~\ref{lemma:diagramnonconvcase}, there is a unique inner face $\Delta(i)=\Delta(\w;D_{\boldsymbol{v}}(f)),\ \w \in \mathbb{Q}^3$, which intersects $\R^{I}$ and   \((f_{|{\K^I}})_{\Delta(i)} \not \equiv 0\). Without loss of generality, we can suppose $d(\boldsymbol{w};D_{\boldsymbol{v}}(f))=1$. To achieve (iv) for this face, i.e., $w_i \leq w_j, \ j\neq i$, we choose $v_i$ sufficiently large. Since $1=v_iw_i$, increasing $v_i$ forces $w_i$ to decrease, thus establishing  $w_i\leq w_j, \ j\neq i$. This selection of $v_i$ yields Condition~(iv) while preserving the validity of Condition~(ii). This construction ensures all four conditions of Theorem~\ref{teo: knd co -> urm} hold, completing the proof.
\end{proof}

\begin{ex}
	Consider the real polynomial   $f(x_1,x_2,x_3)= (x_{1}^{\,6} + x_1 x_{2}^{\,5} + x_2 x_{3}^{\,6} )( x_{1}^{\,2} + x_{2}^{\,2} + x_{3}^{\,2} ).$
	The Newton boundary of $f$ is given by the union of the 2-faces ABDE and BCD (see Fig.\ref{fig:ncexc}) and $I_{\mathrm{nc}}(f)=\{2,3\}$. 
	\begin{itemize}
		\item[(i)] Let $\boldsymbol{v}=(0,v_2,v_3)$ be a vector. By applying Proposition~\ref{prop:kndsurfaces}, we establish that if $\boldsymbol{v}$ satisfies the conditions $v_2>8$ and $v_3>46/5$, the polynomial $f$ is IKND for the C-face diagram $D_{\boldsymbol{v}}(f)$ (see Fig.\ref{fig:ncexc}). Consider $\boldsymbol{w}_2$ and $\boldsymbol{w}_3$ the weight vectors that satisfy $\Delta(2)=\Delta(\boldsymbol{w}_2;D_{\boldsymbol{v}}(f))=CEG$ and $\Delta(3)=\Delta(\boldsymbol{w}_3;D_{\boldsymbol{v}}(f))=BCF$ as Lemma~\ref{lemma:diagramnonconvcase}. If  $8<v_2<9$ and  $v_3>46/5$, then the weight vector $\boldsymbol{w}_2$ is given by $\w_2=(v_2-7,1,\frac{v_2-1}{8})$. Consequently, the C-face diagram $D_{\boldsymbol{v}}(f)$ does not satisfy (iv) in Theorem~\ref{teo: knd co -> urm}. If $v_2\geq 9$ and  $v_3>46/5$, then the C-face diagram $D_{\boldsymbol{v}}(f)$ satisfies  (iv) in Theorem~\ref{teo: knd co -> urm}. This implies that the vector $\boldsymbol{v}$ fits the vector in Proposition~\ref{prop:defkndsurfaces}. Therefore, if the vector $\boldsymbol{v}$ satisfies the conditions  $v_2\geq 9$ and  $v_3>46/5$, then any deformation $F$ of $f$, with $F_\varepsilon-f$ above the $C$-face diagram $D_{\boldsymbol{v}}(f)$ is link-constant.   
		\item[(ii)] 
		Consider the $C$-face diagram $D$ shown in Fig.\ref{fig:ncexb}. This $C$-face diagram is defined by $\ell_1(\boldsymbol{\nu}) = \left\langle \left(\tfrac{1}{8}, \tfrac{1}{8}, \tfrac{1}{8} \right), \boldsymbol{\nu} \right\rangle$  and $\ell_2(\boldsymbol{\nu}) = \left\langle \left(\tfrac{6}{46}, \tfrac{6}{46}, \tfrac{5}{46} \right), \boldsymbol{\nu} \right\rangle$.
		
		This $C$-face diagram satisfies (iv) in Theorem~\ref{teo: knd co -> urm} and the polynomial $f$ is IKND with respect to it. Consequently, by the direct application of Theorem~\ref{teo: knd co -> urm}, any deformation $F$ of $f$, with $F_\varepsilon-f$ above the $C$-face diagram $D$ is link-constant. The existence of the $C$-face diagram $D$ demonstrates the potential to find more refined diagrams, such as the family of diagrams $D_{\boldsymbol{v}}$, by exploring alternative constructions.  
	\end{itemize}
    \begin{figure}[h]
		\begin{subfigure}[a]{.4\textwidth}
			\centering
			\begin{tikzpicture}[scale=0.4,>=stealth,tdplot_main_coords]
				\coordinate (O) at (0,0,0);
				\coordinate[label=above right:{C}] (C) at (0,1,8);
				\coordinate[label=above left:{B}] (B) at (6,0,2);
				\coordinate[label=above left:{D}] (D) at (1,5,2);
				\coordinate[label=above left:{A}] (A) at (8,0,0);
				\coordinate[label=below left:{E}] (E) at (1,7,0);
				\coordinate[label=above right:{F}] (F) at (0,0,9.3);
				\coordinate[label=above right:{G}] (G) at (0,8.5,0); 
				\draw[->] (O) -- (11,0,0);
				\draw[->] (O) -- (0,11,0);
				\draw[->] (O) -- (0,0,11);
				\filldraw[draw=blue,fill=blue,fill opacity=0.1] (G) -- (E) -- (0,1,8) -- cycle;
				\filldraw[draw=blue,fill=blue,fill opacity=0.1] (C) -- (D) -- (E) -- cycle;
				\filldraw[draw=blue,fill=blue,fill opacity=0.1] (B) -- (C) -- (F) -- cycle;
				\filldraw[draw=red,fill=red,fill opacity=0.1] (C) --(B)--(D)-- cycle;
				\filldraw[draw=purple,fill=purple,fill opacity=0.3] (B) -- (A)  -- (E)  --(1,5,2) -- cycle;
				\fill[black] (A) circle (5pt);
				\fill[black] (B) circle (5pt);
				\fill[black] (C) circle (5pt);
				\fill[black] (D) circle (5pt);
				\fill[black] (E) circle (5pt);
				\fill[black] (F) circle (5pt);
				\fill[black] (G) circle (5pt);				
			\end{tikzpicture}
			\caption{}
			\label{fig:ncexc}
		\end{subfigure}  
		\hspace{1.8cm}
		\begin{subfigure}[a]{.4\textwidth}
			\centering
			\begin{tikzpicture}[scale=0.4,>=stealth,tdplot_main_coords]
				\coordinate (O) at (0,0,0);
				\coordinate[label=above right:{C}] (C) at (0,1,8);
				\coordinate[label=above left:{B}] (B) at (6,0,2);
				\coordinate[label=above left:{D}] (D) at (1,5,2);
				\coordinate[label=above left:{A}] (A) at (8,0,0);
				\coordinate[label=below left:{E}] (E) at (1,7,0);
				\coordinate[label=above right:{I}] (I) at (0,0,9.2);
				\coordinate[label=above right:{J}] (J) at (0,8,0); \coordinate[label=above right:{H}] (H) at (0,6,2);
				\draw[->] (O) -- (11,0,0);
				\draw[->] (O) -- (0,11,0);
				\draw[->] (O) -- (0,0,11);
				\filldraw[draw=red,fill=red,fill opacity=0.1] (B) --(I) -- (H) -- cycle;
				\filldraw[draw=purple,fill=purple,fill opacity=0.3] (B) -- (A)  -- (J) --(H) -- cycle;
				\fill[black] (A) circle (5pt);
				\fill[black] (B) circle (5pt);
				\fill[black] (C) circle (5pt);
				\fill[black] (D) circle (5pt);
				\fill[black] (E) circle (5pt);
				\fill[black] (I) circle (5pt);
				\fill[black] (J) circle (5pt);
				\fill[black] (H) circle (5pt);
			\end{tikzpicture}
			\caption{}
			\label{fig:ncexb}
		\end{subfigure}
		\caption{$A=(8,0,0)$, $B=(6,0,2)$, $C=(0,1,8)$, $D=(1,5,2)$, $E=(1,7,0)$, $F=(0,0,v_3)$, $G=(0,v_2,0)$, $H=(0,6,2)$, $I=(0,0,46/5)$, $J=(0,8,0)$.}
		\label{fig:ncex}
	\end{figure}
\end{ex}
\subsection{SWH and SRWH maps}
\begin{theorem}\label{Th:Main2}
	Let \( F: \K^n \times \K \to \K^p \)  be a deformation of a real polynomial map (resp. mixed polynomial map) \( f \). Assume the following conditions hold:
	\begin{itemize}
		\item[(i)] \(f\) is SWH (resp. SRWH) of weight-type \((\boldsymbol{w}; \boldsymbol{d})\),
		\item[(ii)] for each  $j=1,\dots,p,$ and $\varepsilon$, \(d({\boldsymbol{w}}; F_{\varepsilon}^{\,j} -f^j) \geq d_j.\)
	\end{itemize}
	Then, the deformation is link-constant along some neighborhood $K \subset \K$ of the origin. If the inequalities in (ii) are strict,  then the deformation is link-constant. Furthermore, if $\Sigma(f_{\boldsymbol{w}})=\{0\}$, then the deformation is also topologically trivial along the set $K$ or $\K$. 
\end{theorem}
\begin{proof}
	Let $(\boldsymbol{w};\boldsymbol{d})=(w_1,\dots,w_n;d_1,\dots,d_p)$ denote its weight-type. Let $D_j$ be the $C$-face diagrams defined from $J_j=\{\ell_{\w}(\nu)/{d_j}\}.$
	
	Without loss of generality, assume that $f$ is a real polynomial map, the result for mixed polynomial maps follows by regarding it as real polynomial map by identifying $\C$ with $\R^2$.
	
	It is enough to suppose that $f$ is WH, that is, $f=f_{\w}$, or equivalently, $\tilde{f}\equiv 0$.   
	
	Let $\theta^{\, j}_\varepsilon := F_{\varepsilon}^{\,j} - f^j$, $j = 1, \dots,p$ and $I=\{j\in [p] \mid d(\boldsymbol{w};\theta^{\, j}_\varepsilon)= d_j\}$. We divide the proof into two cases: 
	$I=\emptyset$ and $I\neq \emptyset$.
	\vspace{0.2cm}

	\noindent \underline{Case I: $I=\emptyset$.}
	\noindent Consider the analytic function
	$\rho:\R^n \to \R$ defined as:
	\[\rho(\boldsymbol{x})=\sum_{i=1}^n |x_i|^{\frac{2w}{w_i}}, \ w=w_1 w_2 \cdots w_p.\]
	
	Consider the map \(\pi: \rho^{-1}(1) \times \R \to \R^n\) defined as: 
	\begin{equation}\label{eqsem4}
		\pi(\boldsymbol{s}, r)=(r^{w_1}s_1,\dots,r^{w_n}s_n).
	\end{equation}
	This map is a diffeomorphism outside the origin. This map is also used in \cite[Theorem~IV]{fukui1998}.
	
	\noindent Since $f$ is WH of weight-type $(\boldsymbol{w};\boldsymbol{d})$, we have for any $j$:
	\begin{align}\label{eqsem5}
		F_{\varepsilon}^{\,j} (\pi(\boldsymbol{s},r))&=f^j(\boldsymbol{s})r^{d_j} +\theta^{\, j}_\varepsilon(\pi(\boldsymbol{s},r)),
	\end{align}
	where, by $d(\w;\theta^{\, j}_{\varepsilon}) > d_j,\ j=1,\dots,p$,  
	\(\lim _{r \to 0} r^{-d_j} \theta^{\, j}_\varepsilon (\pi(\boldsymbol{s},r))=0.\)
	Thus, $F_{\varepsilon}(\pi(\boldsymbol{s},r))=0$ 
	if and only if: 
	\[r^{d_j}(f^j(\boldsymbol{s})+ r^{-d_j}\theta^{\, j}_\varepsilon (\pi(\boldsymbol{s},r)))=0, \ \text{for all } j =1,\dots p, \]
	which implies $r=0$ (corresponding to the origin) or: 
	\begin{equation}\label{Sys:1thmain2}
		f^j(\boldsymbol{s})+ r^{-d_j}\theta^{\, j}_\varepsilon (\pi(\boldsymbol{s},r))=0, \ \text{for all } j=1,\dots,p.
	\end{equation}
	The system provided by \eqref{Sys:1thmain2} can be seen as the zero set of the map
	\[G(\boldsymbol{s},r,\varepsilon):\rho^{-1}(1) \times \R_{\geq 0}\times \R_{\geq 0} \to \R^p\] which is a deformation with respect to the parameter $r$ of the map $G_0(\boldsymbol{s},\varepsilon):\rho^{-1}(1) \times \R_{\geq 0} \to \R^p$, where $$G_0(\boldsymbol{s},\varepsilon):=G(\boldsymbol{s},0,\varepsilon)=f_{|\rho^{-1}(1)}(\boldsymbol{s}).$$
	It follows that:
	\[\mathrm{rank}(J_G(\boldsymbol{s},0,\varepsilon)=\mathrm{rank} (J_{f_{|\rho^{-1}(1)}}(\boldsymbol{s})).\]
	Since $\Sigma(f) \cap V(f) =\{0\}$ and $f$ is WH, for any point $(\boldsymbol{s}_0,\varepsilon_0)\in \rho^{-1}(1) \times \R_{\geq 0}$ satisfying $G(\boldsymbol{s}_0,0,\varepsilon_0)=f(\boldsymbol{s}_0)=0$, we find
	\(\mathrm{rank}(J_G(\boldsymbol{s}_0,0,\varepsilon_0))=p.\)
    
	If the link of $f$ is empty, then the link of each member of the deformation is empty as well, which implies the result. If $f$ has non-empty link, by the implicit function theorem, we can locally, at a neighborhood of $(\boldsymbol{s}_0,0,\varepsilon_0)$, parametrize $\pi^{-1}(V(F))$  (after possible changes of variables) as:  
	\begin{equation}\label{eqsem1}
		(s_1(s_{p+2},\dots,s_n,r,\varepsilon),\dots,s_{p+1}(s_{p+2},\dots,s_n,r,\varepsilon),s_{p+2}, \dots s_n,r,\varepsilon),
	\end{equation}
	with $r$ and $\varepsilon$ moving in $[0,\delta_{s_0,\varepsilon_0})$ and $(\varepsilon_0-\tau_{\boldsymbol{s}_0,\varepsilon_0},\varepsilon_0+\tau_{\boldsymbol{s}_0,\varepsilon_0})$, respectively, and \[  s_i(s_{0,p+2}, \dots, s_{0,n},0,\varepsilon_0) =s_{0,i}, \ i=1,\dots,p+1. \]
	
	By compactness of $\rho^{-1}(1)\times\{\varepsilon_0\}$, there exists a finite family $\mathcal{F}$ of parametrizations of the form \eqref{eqsem1}, covering the set 
	\begin{equation}\label{eqsem2}
		(\rho^{-1}(1) \times [0,\delta_{\mathrm{min}}) \times (\varepsilon_0-\tau_{\mathrm{min}},\varepsilon_0+\tau_{\mathrm{min}}))  \cap G^{-1}(0),
	\end{equation}
	where $\delta_{\mathrm{min}}>0$ and $\tau_{\mathrm{min}}>0$ are choosing as the minimum values of $\rho$'s and $\tau$'s associated with the parametrizations in $\mathcal{F}$, respectively. 
	
	By the uniqueness of the parametrizations coming from the implicit function theorem, we can glue the parametrizations of the family $\mathcal{F}$ to get a parametrization of the entire set \eqref{eqsem2}. 
	
	Therefore, for any fixed \( 0 < r < \rho_{\mathrm{min}} \), this parametrization provides an isotopy, obtained by varying the parameter \( \varepsilon \), contained in \( \rho^{-1}(1) \times {r} \), between the following links
	$$(\rho^{-1}(1) \times \{r\}) \cap \pi^{-1}(V(F_{\varepsilon_0})) \text{ and } (\rho^{-1}(1) \times \{r\}) \cap \pi^{-1}(V(F_{\tau})),$$ with $\ \tau \in (\varepsilon_0-\tau_{\mathrm{min}},\varepsilon_0+\tau_{\mathrm{min}}).$ Then, we have a smooth isotopy between 
	\begin{equation}\label{eq:isotopy}
		\pi(\rho^{-1}(1) \times \{r\}) \cap V(F_{\varepsilon_0}) \text{ and } \pi(\rho^{-1}(1) \times \{r\}) \cap V(F_{\tau}),\end{equation}
	by the isotopy extension theorem \cite{Hirsch1976}, it extends to an ambient isotopy of $\pi(\rho^{-1}(1) \times \{r\})$. Thus, as in the proof of Proposition~\ref{Prop:GD}, we can construct an isotopy between $L_{F_0}$ and $L_{F_\tau}$. It is clear that finitely many such isotopies suffice to produce an isotopy between \( L_{F_0} \) and \( L_{F_1} \).
	
	\noindent \underline{Case II: $I\neq \emptyset$.} \noindent Consider the map $\pi$ defined in \eqref{eqsem4}. Differently of \eqref{eqsem5}, 
	for $ F_{\varepsilon}^{\,j} $, we have now:
	\begin{align*}
		F_{\varepsilon}^{\,j} (\pi(\boldsymbol{s},r))&=(f^j(\boldsymbol{s})+ (\theta^{\, j}_{\varepsilon})_{D_j}(\boldsymbol{s}))r^{d_j} +h^j(\pi(\boldsymbol{s},r),\varepsilon),
	\end{align*}
	where $(\theta^{\, j}_{\varepsilon})_{D_j}$ (possibly null not for all $i$) denotes the restriction of $\theta^{\, j}_\varepsilon$ to the $C$-face diagram $D_j$, which in this case this restriction is a WH polynomial of  weight-type  \((\boldsymbol{w};d_j)\), and $h$ satisfies:
	\(\lim _{r \to 0} r^{-d_j}h^j(\pi(\boldsymbol{s},r),\varepsilon)=0.\)
	Thus, $F_\varepsilon(\pi(\boldsymbol{s},r))=0$ 
	if and only if: 
	\[r^{d_j}( F_{0}^{\,j} (\boldsymbol{s})+(\theta^{\, j}_{\varepsilon})_{D_j}(\boldsymbol{s})+ r^{-d_j}h^j(\pi(\boldsymbol{s},r),\varepsilon))=0, \ j=1, \dots , p \]
	which implies $r=0$ (corresponding to the origin) or: 
	\begin{equation}\label{Sys:2thmain2}
		f^j(\boldsymbol{s})+(\theta^{\, j}_{\varepsilon})_{D_j}(\boldsymbol{s})+ r^{-d_j}h^j(\pi(\boldsymbol{s},r),\varepsilon)=0, \ j=1, \dots , p.
	\end{equation}
	We describe \eqref{Sys:2thmain2} as the zero set of  the map
	\[H(\boldsymbol{s},r,\varepsilon):\rho^{-1}(1) \times \R_{\geq 0}\times \R_{\geq 0} \to \R^p ,\] 
	 which is a deformation with respect to the parameter $ \varepsilon$ (differently to the system~\eqref{Sys:1thmain2} that we see as a deformation in the parameter $r$) of the map $H(\cdot,0):\rho^{-1}(1) \times \R_{\geq 0} \to \R^p$, with $H(\boldsymbol{s},r,0)=f_{|\rho^{-1}(1)}(\boldsymbol{s}).$
	 
	It follows that:
	\[\mathrm{rank}(J_H(\boldsymbol{s},r,0))=\mathrm{rank} (J_{f_{|\rho^{-1}(1)}}(\boldsymbol{s})).\]
	Since $\Sigma(f)\cap V(f)=\{0\}$ and $f$ is WH, for any point  $(\boldsymbol{s}_0,0)$ satisfying $H(\boldsymbol{s}_0,0,0)=0$, we find \(\mathrm{rank}(J_H(\boldsymbol{s}_0,0,0))=p.\)
	Analogously to the case before we find $\rho_{\mathrm{min}}$ and $\tau_{\mathrm{min}}$, such that for any $r \in (0, \rho_{\mathrm{min}})$ fixed, there is an isotopy contained in $\rho^{-1}(1)\times \{r\}$, between the links 
	$$(\rho^{-1}(1) \times \{r\}) \cap \pi^{-1}(V(F_{0})) \text{ and } (\rho^{-1}(1) \times \{r\}) \cap \pi^{-1}(V(F_{\tau})),\ \tau \in [0,\tau_{\mathrm{min}}),$$
	and by compose with $\pi$, we get an isotopy between 
	\begin{equation}\label{eq:rholinkson}\pi(\rho^{-1}(1)\times \{r\}) \cap V(F_{0}) \text{ and } \pi(\rho^{-1}(1)\times \{r\}) \cap V(F_{\tau}),\ \tau \in [0,\tau_{\mathrm{min}}).
	\end{equation}
	As in Case I, we obtain an ambient isotopy between the links $L_{F_0}$ and $L_{F_\tau}$.
	
	Now, if the inequalities in (ii) are strict and $\Sigma(f_{\boldsymbol{w}})=\{0\}$, then from Case I we can conclude that the deformation $F$ has a $\rho$-uniform Milnor radius. Hence, the topological triviality follows by Propositions~\ref{prop:SWH->IKND} and \ref{Prop:toptrivialwithrho}.
\end{proof}
\begin{remark}
	The topologically trivial deformations studied in Theorem~\ref{Th:Main2} are specific examples of the deformations detailed in \cite[Theorem~2.14]{zBekka2015}. We confirm their existence using the polynomial maps constructed by Looijenga in \cite{Looijenga1971} and Bode \cite{Bode2019}.
\end{remark}
\begin{ex}\label{ex:semiradial}
	\begin{itemize}
		\item[(i)] Consider \(
		f(x_1,x_2,x_3)= x_{1}^{\,12} +x_1 x_{2}^{\,4} x_3 +( x_{2}^{\,3} -x_{3}^{\,2} )^2
		\) the real polynomial function of Example~\ref{ex1:SIKND} which is SIKND. It is weighted homogeneous of weight-type $(1,2,3;12)$. Thus, by Theorem~\ref{Th:Main2}, for any deformation $F$ satisfying $d((1,2,3); F_\varepsilon-f)\geq 12$, $F$ is link-constant along some neighborhood $K \subset \R$ of the origin. If $d((1,2,3); F_\varepsilon-f)\geq 12$, then the deformation is topologically trivial and link-constant.      
		\item[(ii)] Consider 
		$f(\x)= x_{1}^{\,6} +\overbar{x_{2}} x_{1}^{\,4} + x_1 \overbar{x_{2}}^3 + x_{2}^{\,6} .$ A direct computation shows that \( f \) is SRWH of weight-type $(\w=(2,3);11)$. Thus, by Theorem~\ref{Th:Main2}, we have that the link of $f$ is ambient isotopic to the link of $f_{\w}(\x)=\overbar{x_{2}} x_{1}^{\,4} + x_1 \overbar{x_{2}}^3$.  
	\end{itemize}
\end{ex}
\subsection{Inner non-degenerate mixed polynomial functions of two variables}\label{section5}
Inspired also by \cite{wall}, the concept of inner non-degeneracy for mixed polynomial functions of two variables was introduced in \cite{AraujoBodeSanchez}. Although that definition relies on weight vectors associated with the 1-faces of the Newton boundary and non-extreme vertices, we show in this section that it coincides with our own Definition~\ref{def: ind}. Thus, we refer to a mixed polynomial function that satisfies either (and hence both) of these non-degeneracy conditions as an inner non-degenerate mixed polynomial function. In this section, we show that the inner Newton non-degeneracy defined in \cite{AraujoBodeSanchez} is equivalent to our definition of inner Khovanskii non-degeneracy. Furthermore, we prove that the link of an inner non-degenerate mixed polynomial function depends only on certain monomials of the radial Newton boundary.

Consider the set of weight vectors
\begin{equation*}
	\mathcal{P}(f)=\{\w_1,\dots,\w_N\}.
\end{equation*}  
For each \(\w_i\) in $\mathcal{P}(f)$, define \(k_i:=\frac{w_{i,1}}{w_{i,2}}\), the slope of the line passing through $\Delta(\w_i;f)$. We order $\mathcal{P}(f)$ so that $i<j$ if and only if $k_i >k_j$.

The vertices of $\Gamma(f)$ are classified as either \textit{non-extreme} or \textit{extreme}. A vertex $\Delta$ is called \textit{non-extreme} if there exits $\w_j\in \mathcal{P}(f)$ such that $\Delta=\Delta(\w_j;f)\cap \Delta(\w_{j+1};f)$, for consecutive weight vectors $\w_j, \w_{j+1} \in \mathcal{P}(f)$. Otherwise, $\Delta$ is called \textit{extreme}.  
\begin{definition}[{\cite{AraujoBodeSanchez}}]\label{Newtoncond}
	Let \( f:\C^2 \to \C \) be a mixed polynomial function. We say that \( f \) is:
	\begin{itemize}
		\item \textbf{Inner Newton non-degenerate (INND)} if either:
		\begin{enumerate}[label=\normalfont(\arabic*)]
			\item \( \mathcal{P}(f) \neq \emptyset \), and:
			\begin{itemize}
				\item[(i)] \( \Sigma(f_{\w_1}) \cap V(f_{\w_1})  \cap (\C^2 \setminus \{x_2=0\}) = \emptyset \), and \( \Sigma(f_{\w_N}) \cap V(f_{\w_N}) \cap (\C^2 \setminus \{x_1=0\}) = \emptyset \);
				\item[(ii)] for every 1-face or non-extreme vertex \( \Delta \), \( \Sigma(f_\Delta) \cap V(f_\Delta) \cap (\C^*)^2 = \emptyset \);
			\end{itemize}
			\item \( \mathcal{P}(f) = \emptyset \), and \( \Sigma(f_\Delta) \cap V(f_\Delta) \cap (\C^2\setminus \{0\}) = \emptyset \), where \( \Delta \) is the unique face of \( \Gamma(f)=\{\Delta\} \), a vertex.
		\end{enumerate}
		\item \textbf{Strongly inner Newton non-degenerate (SINND)} if either:
		\begin{enumerate}[label=\normalfont(\arabic*)]
			\item \( \mathcal{P}(f) \neq \emptyset \), and:
			\begin{itemize}
				\item[(i')] \( \Sigma(f_{\w_1}) \cap (\C^2 \setminus \{x_2=0\}) = \emptyset \), and \( \Sigma(f_{\w_N}) \cap (\C^2 \setminus \{x_1=0\}) = \emptyset \);
				\item[(ii')] for every 1-face or non-extreme vertex \( \Delta \), \( \Sigma(f_\Delta) \cap (\C^*)^2 = \emptyset \);
			\end{itemize}
			\item \( \mathcal{P}(f) = \emptyset \), and \( \Sigma(f_\Delta) \cap (\C^2\setminus \{0\}) = \emptyset \).
		\end{enumerate}
	\end{itemize}
\end{definition}
\begin{proposition}\label{prop:caracterizationinner} Let $f:\C^2 \to \C$ be a mixed polynomial function, then $f$ is (S)INND if and only if $f$ is (S)IKND. 
\end{proposition}
\begin{proof} 
	We will prove the case of inner non-degeneracies, the stronger version follows analogously. We divide the proof into two cases: 
	
	\underline{Case I: $\mathcal{P}(f) = \emptyset$}. Assume that $f$ is INND. Then, $\Gamma(f)$ has a unique compact face, namely the vertex $(1,1)$. It follows that $D(J),\ J=\{\ell(\nu)=\frac{1}{2}(\nu_1+\nu_2)\},$ satisfies (i)-(ii) of Definition~\ref{def: ind}.
	
	Conversely, let $D$ be a $C$-face diagram satisfying (i)-(ii) of Definition~\ref{def: ind}. For any inner face $\tilde{\Delta}$ of $D$, we have $f_{\tilde{\Delta}} = f_{\Delta}$, where $\Delta$ is the unique compact face (a vertex) of $\Gamma(f)$. Then, by (ii) of Definition~\ref{def: ind}, $f$ is INND.
	
	\underline{Case II: $\mathcal{P}(f) \neq \emptyset$}. Assume that $f$ is INND. Let \(
	J =\left\{ \frac{\ell_{\w_i}(\nu)}{d(\w_i;f)} \mid \w_i \in \mathcal{P}(f) \right\}.
	\) From (i)-(ii) of Definition~\ref{Newtoncond}, we get that $f$ is IKND with respect to $D(J)$.
    
	Conversely, let \(D\) be a $C$-face diagram satisfying (i)-(ii) of Definition~\ref{def: ind}. Denote by \\ \(\tilde{\Delta}_1, \tilde{\Delta}_2, \dots, \tilde{\Delta}_{\tilde{N}}\) the \(1\)-faces of \(D\), indexed by decreasing slope so that \(\tilde{\Delta}_1\) has the largest slope.  Denote $\Delta_i=\Delta(\w_i;f), \ i=1,\dots, N,$  where  $\{\w_1,\dots,\w_N\}=\mathcal{P}(f)$.
	
	We have that 
	\[\{\Delta_2,\dots, \Delta_{N-1}\} \subset \{\tilde{\Delta}_1,\dots, \tilde{\Delta}_N\}.\]
	Indeed, otherwise either (i) of Definition~\ref{def: ind} fails, or $\tilde{\Delta}_j \cap \tilde{\Delta}_{j+1} \cap \supp (f) = \emptyset$ for some $j = 1, \dots, \tilde{N}-1$, which would violate (ii) of Definition~\ref{def: ind} for the inner face $\tilde{\Delta}_j \cap \tilde{\Delta}_{j+1}$. Thus, $\Delta_i, \ i=2,\dots, N-1$ and $\Delta_i \cap \Delta_{i+1}, \ i=2,\dots, N-2,$  are inner faces of \(D\), and hence $f$ satisfies (ii) in Definition~\ref{Newtoncond} for these faces. It remains to verify the following: 
	\begin{equation}\label{eq:newton10}
		\Sigma(f_{\Delta_1 \cap \Delta_2  }) \cap V(f_{\Delta_1 \cap \Delta_2  }) \cap (\C^*)^2 = \emptyset, \text{ and } \Sigma({f_{\w_1}}) \cap V({f_{\w_1}}) \cap (\C^2 \setminus \{x_2 = 0\}) = \emptyset.
	\end{equation}
	\begin{equation} \label{eq:newton11}
		\Sigma(f_{\Delta_{N-1} \cap \Delta_N}) \cap V(f_{\Delta_{N-1} \cap \Delta_N})\cap (\C^*)^2 = \emptyset, \text { and } \Sigma({f_{\w_N}}) \cap V({f_{\w_N}}) \cap (\C^2 \setminus \{x_1 = 0\}) = \emptyset.
	\end{equation}
	We will prove \eqref{eq:newton10}. If \(\Delta_1 \subseteq \tilde{\Delta}_1\), then $\Delta_1 \cap \Delta_2$ is an inner face of $D$ and $f_{\Delta_1}=f_{\tilde{\Delta}_1}$; hence \eqref{eq:newton10} follows by applying (ii) of Definition~\ref{def: ind}. 
	If \(\Delta_1 \not\subset \tilde{\Delta}_1\), then \(\tilde{\Delta}_1\cap \supp (f)\) does not contain a point of the form \((0,b)\), but it does contain a point of the form \((1,b)\); otherwise,
	\[
	\Sigma({f_{\tilde{\Delta}_1}}) \cap V({f_{\tilde{\Delta}_1}}) \cap (\C^*)^{\{2\}} \neq \emptyset,
	\]
	which contradicts (ii) of Definition~\ref{def: ind}. Hence, \(\Delta_0 := (1,b)\) is an extreme vertex of \(\Gamma(f)\) such that \(f_{\Delta_0}=f_{\tilde{\Delta}_1}\) and 
	\begin{equation} \label{eq:INDandD2}
		\Sigma({f_{\Delta_0}}) \cap V({f_{\Delta_0}}) \cap (\C^*)^{\{2\}} = \emptyset.
	\end{equation}
	There are two possible cases:
	\[ 
	f_{\w_1}(\x) = f_{\Delta_0}(\x) +  M(x_2,\overbar{x_{2}}) \text{ or } f_{\w_1}(\x) = f_{\Delta_0}(\x) + \sum_{\nu_1+\mu_1\geq 2} c_{\boldsymbol{\nu},\boldsymbol{\mu}}\x^{\boldsymbol{\nu}}\overbar{\x}^{\boldsymbol{\mu}},\ c_{\boldsymbol{\nu},\boldsymbol{\mu}}\in \C^*.
	\]
	In the first case, $\Delta_1 \cap \Delta_2=\Delta_0$, and \eqref{eq:INDandD2} implies \eqref{eq:newton10}. In the second case,  we have \(
	f_{\w_1}(\x) = f_{\tilde{\Delta}_2}(\x)\)
	and $\Delta_1 \cap \Delta_2=\tilde{\Delta}_2 \cap \tilde{\Delta}_3$. Applying (ii) of Definition \ref{def: ind} we conclude
	\begin{equation} \label{eq:INDandD3}
		\Sigma({f_{\Delta_1 \cap \Delta_2}}) \cap V({f_{\Delta_1 \cap \Delta_2}}) \cap (\C^*)^{\{1,2\}}=\Sigma({f_{\w_1}}) \cap V({f_{\w_1}}) \cap (\C^*)^{\{1,2\}} = \emptyset.
	\end{equation}  
	From \eqref{eq:INDandD2}, it follows that 
	\begin{equation} \label{eq:INDandD4}
		\Sigma({f_{\w_1}}) \cap V({f_{\w_1}}) \cap (\C^*)^{\{2\}} = \emptyset.
	\end{equation}
	Combining \eqref{eq:INDandD3} and \eqref{eq:INDandD4} yields \eqref{eq:newton10}. 
	
	Similarly, \eqref{eq:newton11} holds by applying the same argument with $\Delta_N$ instead of $\Delta_1$. Therefore, \(f\) is INND.  
\end{proof} 
Now, we present results in link-constancy of deformations of IKND mixed polynomial functions. For this purpose, we are not concerned with establishing our result for any diagram that satisfies certain properties, as in Theorem~\ref{teo: knd co -> urm}; instead, we aim to define the optimal diagram associated with \(f\) that admits the wider class of link-constant deformations for it.  
\begin{definition}\label{def:Gammainn}
Let $f:\C^2\to \C$ be an IKND mixed polynomial function. Define $\Gamma_{\mathrm{inn}}(f)$ as a $C$-face diagram $D$ satisfying: 
\begin{itemize}
    \item[(i)] $f$ is IKND with respect to $D$,
    \item[(ii)] for $i \in I_{\rm{nc}}(f_D)$ and any inner face $\Delta(\boldsymbol w;D), \ \boldsymbol w \in (\mathbb{Q}_{>0})^2$, of $D$ satisfying  \(\Delta(\boldsymbol w;D) \cap \R^{\{i\}}\neq \emptyset,\) we get 
		\(    w_{i} \leq w_{j},\text{ for all } j,\)
        \item[(iii)] if $D'$ is another $C$-face diagram satisfying (i) and (ii) then  $$D'+(\R_{\geq 0})^{2} \subseteq D+(\R_{\geq 0})^{2}.$$
\end{itemize}
Define the set of positive weight vectors
$	\mathcal{P}_{\rm inn}(f)=\{\boldsymbol{w}_1,\dots,\boldsymbol{w}_N\}$ associated with the 1-faces of $\Gamma_{\rm inn}(f)$ so that $i<j$ if and only if  $k_i >k_j$, where \(k_i:=\frac{w_{i,1}}{w_{i,2}}\). 
\end{definition}

	\begin{figure}[h]
		\begin{subfigure}[a]{.4\textwidth}
			
			\begin{tikzpicture}[scale=0.65]
				\begin{axis}[axis lines=middle,axis equal,yticklabels={0,,2,4,6,8}, xticklabels={0,,2,4,6,8},domain=-10:10,     xmin=0, xmax=10,
					ymin=0, ymax=10,
					samples=1000,
					axis y line=center,
					axis x line=center]  
					\fill[yellow!90,nearly transparent] (0,60) -- (10,30) -- (40,10) -- (60,0) -- (120,0) -- (120,120) -- (0,120) --cycle;
					
					\filldraw[blue] (0,700) circle (2.3pt)  node[anchor=east,,dashed] {};
					\filldraw[blue] (100,400) circle (2.3pt)  node[anchor=east,,dashed] {};
					\filldraw[blue] (200,200) circle (2.3pt)  node[anchor=east,,dashed] {};
					\filldraw[red] (0,300) circle (2.3pt)  node[anchor=east,,dashed] {};
					\filldraw[blue] (300,0) circle (2.3pt)  node[anchor=east,,dashed] {};
					\filldraw[black] (60,60) node[anchor=north ] {$\Gamma_+(f)$};
					\addplot[mark=none, line width=1.2pt, blue] coordinates{(6,0) (4,1) (1,3) (0,6)}; 
					\addplot[only marks, mark=*, blue, mark size=2pt] coordinates {(6,0)};
					\addplot[only marks, mark=*, blue, mark size=2pt] coordinates {(4,1)};
					\addplot[only marks, mark=*, blue, mark size=2pt] coordinates {(1,3)};
					\addplot[only marks, mark=*, blue, mark size=2pt] coordinates {(0,6)};
				\end{axis}
			\end{tikzpicture}
			\caption{Newton boundary of $f$.}
			\label{fig:1ex1}
		\end{subfigure}
		\hfil \hfil
		\begin{subfigure}[a]{.4\textwidth}
			\begin{tikzpicture}[scale=0.65]
				\begin{axis}[axis lines=middle,axis equal,yticklabels={0,,2,4,6}, xticklabels={0,,2,4,6},domain=-10:10,     xmin=0, xmax=10,
					ymin=0, ymax=10,
					samples=1000,
					axis y line=center,
					axis x line=center]  
					\fill[yellow!90,nearly transparent] (0,40) -- (10,30) -- (110/2,0) -- (140,0) -- (140,120) -- (0,120) --cycle;
					
					\addplot[mark=none, line width=1.2pt, blue] coordinates{(11/2,0) (4,1) (1,3) (0,4)}; 
					\addplot[only marks, mark=*, blue, mark size=2pt] coordinates {(1,3)};
					\addplot[only marks, mark=*, blue, mark size=2pt] coordinates {(0,4)};
					\addplot[only marks, mark=*, blue, mark size=2pt] coordinates {(11/2,0)};
					\filldraw[black] (60,60) node[anchor=north ] { $\Gamma_{\rm inn}(f)+(\R_{\geq 0})^2$};
					\filldraw[blue] (0,500) circle (2.3pt)  node[anchor=east,,dashed] {};
					\filldraw[blue] (100,400) circle (2.3pt)  node[anchor=east,,dashed] {};
					\filldraw[blue] (200,200) circle (2.3pt)  node[anchor=east,,dashed] {};
					\filldraw[blue] (100,300) circle (2.3pt)  node[anchor=east,,dashed] {};
					\filldraw[black] (110,105) node[anchor=north ] {$D$};
				\end{axis}
			\end{tikzpicture}
			\caption{$C$-face diagram $\Gamma_{\rm inn}(f)$.}
			\label{fig:2ex1}\end{subfigure}
		\caption{$C$-face diagrams associated with $f$.}
		\label{fig:Ddiagrams}
	\end{figure}
\begin{ex}\label{ex:nice}
	Consider $f(\x)=x_{1}^{\,6} + \overbar{x_{2}} x_{1}^{\,4} + x_1 \overbar{x_{2}}^3 + x_{2}^{\,6}$.
	Let $D(J_1)$ and $D(J_2)$ be the $C$-face diagrams defined by
     \begin{equation*}\label{eq:J1}
		J_1=\{
		\ell_{1,1}(\boldsymbol{\nu}) = \left\langle \left(\tfrac{2}{11},\, \tfrac{3}{11} \right), \boldsymbol{\nu} \right\rangle \} \text{ and }
		J_2=\{
		\ell_{1,1}(\boldsymbol{\nu}), \ell_{1,2}(\boldsymbol{\nu}) = \left\langle \left(\tfrac{1}{4}, \tfrac{1}{4} \right), \boldsymbol{\nu} \right\rangle \}.
	\end{equation*}
	A direct computation confirms that the function $f$ is IKND with respect to the $C$-face diagram $D(J_1)$, $D(J_2)$ and $\Gamma(f)$. Since $D(J_1)$ only has one compact 1-face, $f$ is semi-radially weighted homogeneous. 
    In other hand, $\Gamma_{\rm inn}(f)=D(J_2)$.  
\end{ex}
To achive our aim we focus our analysis on the $\rho$-uniform Milnor radius, which will be instrumental in proving the following result.
\begin{theorem} \label{th:nonice}
	Let \( F: \C^2 \times \C \to \C \) be a deformation of an IKND mixed polynomial function \( f \).  If for each \(\varepsilon\),   \(d(\w;F_\varepsilon-f) \geq d(\w;f), \text{ for all } \w \in \mathcal{P}_{\rm inn}(f),\) then the deformation $F$ is link-constant along some neighborhood $K \subset \C$ of the origin. If the inequality is strict, then $F$ is link-constant. Furthermore, if $f$ is SKND then the deformation is also topologically trivial along the set $K$ or $\K$.  
\end{theorem}
\begin{proof}
	By Proposition~\ref{Prop:GD} it is enough to prove the $\rho_E$-uniform Milnor radius. We will consider $F_\varepsilon$ as a map $F_\varepsilon: (\R^4,0) \rightarrow (\R^2,0)$, then $F_\varepsilon$ has component functions $F_{\varepsilon}^{\,1} = \Re (F_\varepsilon) $ and $F_{\varepsilon}^{\,2} = \Im (F_\varepsilon) $. By Lemma \ref{lemma: IND M -> IKND} we have that $f$ is IKND, then there is a compact set $K \subset \K$ containing the origin such that $F_\varepsilon $ is also IKND for $\varepsilon \in K$. Let $D_1,D_2$ and $D=D_1+D_2$ be the associated $C$-face diagrams. 
	Supposing by contradiction that the deformation does not have a $\rho_E$-uniform Milnor radius and following the same arguments and notations as in the proof of Theorem~\ref{teo: knd co -> urm} we get:
	\begin{align}\label{eq:uniMilnor0}
		\sum_{j \in J} c_j \frac{\partial ( F_{\varepsilon_0}^{\,j} )_{\Theta'_j}}{\partial x_i} (\boldsymbol{x}_1) s^{m-w_i} +\text{h.o.t.} = c_3 x_{0,i} s^{q_i + \beta_3} + \text{h.o.t.}, \quad \text{for all } i \in [4]
	\end{align}
	Here, $m =\min \{d_j+\beta_j \mid j \in \{1,2\}\}$ and $J = \left\lbrace j \mid d_j + \beta_j = m \right\rbrace$ with $d_j= d(\boldsymbol{q}; \Theta_j)$ and  $\Theta_j = \Delta( \boldsymbol{q} ; D_j) $ and $\Theta'_j = \Delta (\boldsymbol{w}; D_j)$. Also, $\Theta=\Delta(\boldsymbol{q}; D \cap \R^I)= \Theta_1+\Theta_2 $ and $\Theta' =\Delta(\boldsymbol{w}; D)= \Theta'_1+ \Theta'_2 $.  Since $f^1$ and $f^2$ arise from an IKND mixed polynomial function, by Claim \ref{claim: mix inner} we have that $w_1=w_2 $ and $w_3=w_4 $.  
	
	Set $I' = \{i \mid m- \beta_{3} = q_i + w_i \}$. By Theorem~\ref{Th:Main1}, $F$ has weak no coalescing of critical points, thus $c_{3} \not=0$, $m- \beta_{3} \leq q_i + w_i$ for all $i$, and $I' \not= \emptyset$. Now we consider two cases depending on whether $I_\Theta$ intersects $I'$ or not.
	
	\underline{Case I: $I' \cap I_\Theta \not= \emptyset $.} 
	We find the same contradiction as Case~I in the proof of Theorem~\ref{teo: knd co -> urm}.
	
	\underline{Case II: \( I' \cap I_\Theta = \emptyset \)}. 
	Observe that if $1 \in I'$ then $2 \not\in I_{\Theta}$. Indeed, if \( 2 \in I_\Theta \), then
	\[
	q_2 + w_2 = w_2 + w_2 = w_1 + w_1 < q_1 + w_1 = m - \beta_3,
	\]
	which contradicts the minimality of \( q_1 + w_1 \).
	Analogously, we have have that if $2 \in I'$ then $1 \not\in I_\Theta$; if $3 \in I'$ then $4 \not\in I_\Theta$; and if $4 \in I'$ then $3 \not\in I_\Theta$.
	
	Without lost of generality, we may assume that $1 \in I'$. Let $q_m = \min\{q_3,q_4\}$, then $ I_\Theta= \{i \in \{3,4\} \mid q_i = q_m\}$.  Define $\overbar{\Theta} = \Delta((q_1,q_2,q_m,q_m); D)$. Then, if $q_3 = q_4$, clearly $\Theta = \overbar{\Theta}$; if $q_3 \not= q_4$, then $\Theta \subset \overbar{\Theta}$. Note that $(F_{\varepsilon_0})_{\overbar{\Theta}} = (F_{\varepsilon_0})_{\tilde{{\Theta}}}$ where $\tilde{\Theta} = \Delta((q_1,q_1,q_m,q_m), \tilde{D})$. If $(F_{\varepsilon_0})_{\tilde{{\Theta}}} \equiv 0$ then by Condition~(iii), $w_3=w_4 \leq w_1=w_2$, which implies that $I_\Theta \cap I' \not= \emptyset$. Hence $(F_{\varepsilon_0})_{\tilde{{\Theta}}} \not\equiv 0$.
	
	We can express $( F_{\varepsilon_0}^{\,1} )_{\Theta'_1}$ as:
	\begin{align*}
		( F_{\varepsilon_0}^{\,1} )_{\Theta'_1} &= ( F_{\varepsilon_0}^{\,1} )_{\tilde{ \Theta}_1} + M_{e_1} x_1 + \cdots
	\end{align*}  
	where $M_{e_1}$ is a polynomial in the variables $x_i$ for $i \in I_{\tilde{\Theta}} $. Suppose that $M_{e_1} \not\equiv 0 $, we have that
	\begin{align*}
		d_1= d(\boldsymbol{w}; x_1 M_{e_1})= w_1 + d(\boldsymbol{w}; M_{e_1}), \quad \text{and} \quad
		d_1= d(\boldsymbol{q}; D_1).
	\end{align*}
	Since $f_{\tilde{{\Theta}}} \not\equiv 0$, \( d_1 \) is an integer number. Given that \( I_{\Theta} \subset \{3,4\} \), we have that \( w_3 \) divides both \( d(\boldsymbol{w}; M_{e_1}) \) and \( d_1 \), so \( w_3 \leq w_1 \)  which implies that $I_\Theta \cap I' \not= \emptyset$. Hence $M_{e_1} \equiv 0 $.
	
	We obtain
	\begin{equation}\label{eq:uniMilnor1}
		\frac{\partial ( F_{\varepsilon_0}^{\,1} )_{\Theta'_1}}{\partial x_1}(\boldsymbol{x}_1) = \frac{\partial ( F_{\varepsilon_0}^{\,1} )_{\Theta_1}}{\partial x_1}(\boldsymbol{x}_0) = 0.
	\end{equation}
	Analogously, we have
	\begin{equation}\label{eq:uniMilnor2}
		\frac{\partial ( F_{\varepsilon_0}^{\,2} )_{\Theta'_1}}{\partial x_1}(\boldsymbol{x}_1) = \frac{\partial ( F_{\varepsilon_0}^{\,2} )_{\Theta_1}}{\partial x_1}(\boldsymbol{x}_0) = 0.
	\end{equation}
	Therefore, substituting  \eqref{eq:uniMilnor1} and \eqref{eq:uniMilnor2} in \eqref{eq:uniMilnor0}, we get  \( c_3 = 0 \), which is a contradiction.
	
	Therefore, the deformation has a $\rho_E$-uniform Milnor radius.  
	
	The topological triviality follows by applying Proposition~\ref{cor: top trivi}.
\end{proof}

\begin{corollary}\label{cor:nicecase}
    Let $f$ be an IKND mixed polynomial function. Then the link of $f$ is isotopic to the link of $f_{\Gamma_{\rm inn}}$, the restriction of $f$ to the $C$-face diagram $\Gamma_{\rm inn}(f)$. 
\end{corollary}
\begin{proof}
    We consider the deformation $F(\bm{x}, \varepsilon) = f(\bm{x}) - \varepsilon (f(\bm{x})-  f_{\Gamma_{\rm inn}}(\bm{x})) $. Hence, by Theorem~\ref{th:nonice} $F$ is link-constant. Therefore, the link of $f$ is isotopic to the link of $f_{\Gamma_{\rm inn}}$. 
\end{proof}
This result is the extension of \cite[Corollary~4.9]{AraujoBodeSanchez} which says that the link of a mixed polynomial function that is INND and $\Gamma$-nice depends only in their Newton boundary. 

We adapt the niceness condition defined for the Newton boundary in \cite{AraujoBodeSanchez} to a niceness condition defined for a $C$-face diagram.
\begin{definition}\label{def:niceness}
Let $f: \C^2 \to \C$ be a mixed polynomial function, then we say that $f$ is:
\begin{itemize}
    \item[(i)]  $\Gamma$-nice (or it has a $\Gamma$-nice Newton boundary) if for every non-extreme vertex $\Delta$ of $\Gamma(f)$, we have $V(f_{\Delta})\cap (\C^*)^2 = \emptyset$.
    \item[(ii)]\textit{nice} with respect to a $C$-face diagram $D$ (or $D$-nice) if, for every inner $0$-face $\Delta$ of $D$, the face function $f_{\Delta}$ satisfies $V({f_{\Delta}})\cap (\C^*)^{2}=\emptyset$.
If $f$ is nice with respect to $D=\Gamma_{\rm inn}(f)$, we say that $f$ is $\Gamma_{\rm inn}$-nice.
\end{itemize}
\end{definition}
If $f$ is IKND, then being $\Gamma$-nice is equivalent to being
$\Gamma_{\mathrm{inn}}$-nice. It is clear that for any face $\Delta$
satisfying \(
V(f_{\Delta})\cap(\mathbb{C}^*)^{2}=\emptyset,
\) the function $f$ is KND on $\Delta$. The converse does not hold, even in
the case of $0$-faces.

We now present a necessary and sufficient condition for a face function associated with a vertex of a $C$-face diagram to be KND. This result helps to better understand the vertices that made to fail the niceness condition.

Let $M:\C^2 \to \C$ be a mixed polynomial function satisfying $\supp (M)=\{(a,b)\}$. Let $x_2=r_2 \rm e ^{\rm i t_2}$. Then, $M$ admits a decomposition of the form: 
\[M(x_1,r_2 \rme^{\rm i t_2})=\overbar{x_{1}}^a r_{2}^{\,b} \sum_{i=0}^a c_{i}(t_2)\left(\frac{x_1}{\overbar{x_{1}}}\right)^i.\]
Define the family of univariate complex polynomials:   \[p_{t_2}(z):=\sum_{i=0}^a c_{i}(t_2)z^i.\]
\begin{proposition}\label{prop:characterizationknd}
    Let $M:\C^2 \to \C$ be a mixed polynomial function with $\supp (M)= \{(a,b)\}$. 
Then, $M$ is KND if and only if for all $z_0 \in \{z \in S^1\mid p_s(z)=0 \text{ for some } s\in [0,2\pi]\}$ the following condition holds: $$ \Im\left( \rmi z_0 \frac{\partial p_{t_2}}{\partial z}(z_0)\overbar{\frac{\partial p_{t_2}}{\partial {t_2}}(z_0)}\right)\neq 0. $$            
\end{proposition}
\begin{proof}
($\Rightarrow$) 
Suppose by contradiction that $$ \Im\left( \rmi z_0 \frac{\partial p_{t_2}}{\partial z}(z_0)\overbar{\frac{\partial p_{t_2}}{\partial t_2}(z_0)}\right) = 0 $$ for some unitary root $z_0$ of $p_{s}$. Let $w \in S^1$, with $w^2=z_0$ and $\lambda_1, \lambda_2 \in \R_{>0}$. Considering variables $(r_1,t_1,r_2,t_2)$, $x_i=r_i \rme^{\rmi t_i}, \ i=1,2$, we find:
$$\frac{\partial M}{\partial r_1}(\lambda_1 w,\lambda_2 \rme^{\rmi s})
=\frac{\partial M}{\partial r_2}(\lambda_1 w,\lambda_2 \rme^{\rmi s})=M(\lambda_1 w,\lambda_2 \rme^{\rmi s})=0,$$ 
$$\frac{\partial M}{\partial t_2}(\lambda_1 w, \lambda_2\rme^{\rmi s})= \lambda_{1}^{\,a} (\overbar{w})^a \lambda_{2}^{\,b} \frac{\partial p_{t_2}}{\partial t_2}(z_0)
$$
and
$$\frac{\partial M}{\partial t_1}(\lambda_1w,\lambda_2 \rme^{\rmi s})
=2 \rmi z_0 \lambda_{1}^{\,a}  (\overbar{w})^a \lambda_{2}^{\,b} \frac{\partial p_{t_2}}{\partial z}(z_0).$$
Therefore, $\Im\left(\frac{\partial M}{\partial t_1}(\lambda_1 w, \lambda_2 \rme^{\rmi s})\overbar{\frac{\partial M}{\partial t_2}(\lambda_1 w,\lambda_2 \rme^{\rmi s})}\right)=0.$ This implies $\Sigma(M)\cap V(M) \cap (\C^*)^2\neq \emptyset$ which contradicts that $M$ is KND.

($\Leftarrow$) We find that 
\begin{equation} \label{eq: nice equivalence}
V(M)\cap(\mathbb{C}^*)^2
=
\Bigl\{
(\lambda_1 w,\lambda_2 \rme^{\rmi s})\in\mathbb{C}^2
\mid
\lambda_1, \lambda_2\in\mathbb{R}_{>0},\
s\in [0,2\pi],\,w \in S^1, \,
p_s(w^2)=0
\Bigr\}.
\end{equation}
Using the calculations from ($\Rightarrow$) for $\bm{a} \in V(M)\cap (\mathbb{C}^*)^2$, we get:
$\Im\left(\frac{\partial M}{\partial t_1}(\bm{a})\overbar{\frac{\partial M}{\partial t_2}(\bm{a})}\right)\neq 0.$ Therefore,  $\Sigma(M)\cap V(M) \cap (\C^*)^2= \emptyset$, thus $M$ is KND.
\end{proof}
\begin{remark}\label{remark:charKNDvertex}
From the characterization in Proposition~\ref{prop:characterizationknd}, we get directly the following implications for a mixed polynomial function $M$ with $\supp (M)= \{(a,b)\}$:
    \begin{itemize}
    \item[(i)] Using the polar coordinate representation $x_1 = r_1 \rme^{\mathrm{i} t_1}$, we analogously derive a family of univariate polynomials $q_{t_1}(\boldsymbol{z})$. Proposition~\ref{prop:characterizationknd} remains valid when $p_{t_2}(\boldsymbol{z})$ is replaced by $q_{t_1}(\boldsymbol{z})$.

        \item[(ii)] If $M$ is ($x_i$ or $\overbar{x_{i}}$)-semiholomorphic, then $M$ is KND if and only if $ V(M) \cap (\C^*)^2= \emptyset$. In particular, if  $a=0$ or $b=0$. 
        \item[(iii)] If $M$ is KND, then each unitary root of $p_t, t\in [0,2\pi],$ is simple.
    \end{itemize}
\end{remark}
\begin{ex}\label{exnonice}
Consider \( f(\bm{x}) = x_{1}^{\,3} + x_2x_1 + \left( x_2 + \frac{1}{2} \overbar{x_{2}} \right) \overbar{x_{1}} + x_{2}^{\,2} \). The mixed polynomial function \( f \) is IKND and not \(\Gamma\)-nice. Furthermore, we have
\(
\Gamma_{\text{inn}}(f) = \Gamma(f).
\)

To illustrate this, define the mixed polynomial function \( M(\bm{x}) = x_2 x_1 + \left( x_2 + \frac{1}{2} \overbar{x_{2}} \right) \overbar{x_{1}} \). A direct application of Proposition~\ref{prop:characterizationknd} shows that \( M \) is KND, and \( V(M) \cap (\mathbb{C}^*)^2 \) is given by:
\[
V(M) \cap (\mathbb{C}^*)^2 = \left( \bigcup_{i=1}^4 l_i \right) \setminus \{ 0 \},
\]
where each \( l_i \) (for \( i = 1, 2, 3, 4 \)) is a line parametrized by \( \lambda \bm{a}_i \), \( \lambda \in \mathbb{R} \). Moreover, a direct calculation shows that if \( \bm{a} \notin V(M) \cap (\mathbb{C}^*)^2 \), then:

\[
|M_{x_1}(\bm{a})| \neq |M_{\overbar{x_{1}}}(\bm{a})| \quad \text{and} \quad |M_{x_2}(\bm{a})| \neq |M_{\overbar{x_{2}}}(\bm{a})|.
\]

Let \( \mathcal{P}(f) = \{ \bm{w}_1, \bm{w}_2 \} \), where \( \bm{w}_1 = (1, 1) \) and \( \bm{w}_2 = (1, 2) \). Thus,
\[f_{\bm{w}_1}(\bm{x}) = x_2x_1 + ( x_2 + \tfrac{1}{2} \overbar{x_{2}} ) \overbar{x_{1}} + x_{2}^{\,2} \text{  and  } f_{\bm{w}_2}(\bm{x}) = x_{1}^{\,3} + x_2x_1 + ( x_2 + \tfrac{1}{2} \overbar{x_{2}} ) \overbar{x_{1}}. \]
Since the mixed monomials \( x_{2}^{\,2} \) and \( x_{1}^{\,3} \) are both KND, we have the following:
\[
\Sigma(f_{\bm{w}_1}) \cap V(f_{\bm{w}_1}) \cap (\mathbb{C}^*)^{\{ 1 \}} = \Sigma(f_{\bm{w}_2}) \cap V(f_{\bm{w}_2}) \cap (\mathbb{C}^*)^{\{ 2 \}} = \emptyset.
\]

To prove that \( f \) is IKND, we still need to show that:

\[
\Sigma(f_{\bm{w}_1}) \cap V(f_{\bm{w}_1}) \cap (\mathbb{C}^*)^{\{1, 2\}} = \Sigma(f_{\bm{w}_2}) \cap V(f_{\bm{w}_2}) \cap (\mathbb{C}^*)^{\{1, 2\}} = \emptyset.
\]

A direct calculation yields that for any \( \bm{a} \notin V(M) \cap (\mathbb{C}^*)^2 \), the following relations hold:

\[
|(f_{\bm{w}_1})_{x_1}(\bm{a})| = |M_{x_1}(\bm{a})| \neq |M_{\overbar{x_{1}}}(\bm{a})| = |(f_{\bm{w}_1})_{\overbar{x_{1}}}(\bm{a})|.
\]
Thus, we conclude that
\(
\Sigma(f_{\bm{w}_1}) \cap (\mathbb{C}^*)^2 \subset V(M) \cap (\mathbb{C}^*)^2.
\) Now, if \( \bm{a} \in V(M) \cap (\mathbb{C}^*)^2 \), we have
\(
f_{\bm{w}_1}(\bm{a}) = M(\bm{x}) + a_{2}^{\,2} = a_{2}^{\,2} \neq 0.
\)
Therefore, \( \bm{a} \notin V(f_{\bm{w}_1}) \). This shows that:
\(
\Sigma(f_{\bm{w}_1}) \cap V(f_{\bm{w}_1}) \cap (\mathbb{C}^*)^2 = \emptyset.
\)
Similarly, we can apply the same reasoning for \( \bm{w}_2 \), yielding:
\(
\Sigma(f_{\bm{w}_2}) \cap V(f_{\bm{w}_2}) \cap (\mathbb{C}^*)^2 = \emptyset.
\)
Therefore, we conclude that \( f \) is IKND with respect to its Newton boundary. To see that $\Gamma_{\rm inn}(f)=\Gamma(f)$, note that $f$ is not IKND with respect to any other $C$-face diagram $D$, since $\Sigma(M)\cap V(M)\cap \C^{\{1\}}\neq \emptyset$ and $\Sigma(M)\cap V(M)\cap \C^{\{2\}}\neq \emptyset$. 
\end{ex}
\begin{ex}\label{ex:nonice2}
Consider the mixed polynomial 
 $$f(\x)= x_{1}^{\,4} + (x_{2}^{\,2} - \overbar{x_{2}}^2) x_1 \overbar{x_{1}} + (\rmi + 1)  {x_2} \overbar{x_{2}} ( x_{1}^{\,2} + \overbar{x_{1}}^2)+ x_{2}^{\,6} .$$
    It can be shown that \(f\) is IKND with respect to the Newton boundary \(\Gamma(f)\). Note that \(f\) does not satisfy the niceness condition. In fact, the unique inner vertex is $\Delta=(2,2)$, for $r \in \R^* $ taking $(x_1,x_2) = (r+ \rmi r , r)  \in (\C^*)^2$ we have that $ x_{1}^{\,2} =2\rmi r^2$ and $x_2^{\,2}= r^2$ then $(r+ \rmi r , r) \in  V(f_{\Delta}) \cap (\C^*)^2$.
\end{ex}

The following proposition implies that both the nice and the non-nice conditions are open properties.
\begin{proposition} \label{prop: niceisopen}
	Let \(F:\C^2 \times \C \to \C\) be a deformation of an IKND mixed polynomial function \(f\) with respect to a \(C\)-face diagram \(D\).
	\begin{itemize}
		\item[(i)] If \(f\) is not \(D\)-nice, then there exists a connected neighborhood \(K \subset \C\) of the origin such that, for every \(\varepsilon \in K\), the function \(F_\varepsilon\) is not \( D\)-nice.
		\item[(ii)] If \(f\) is \(D\)-nice, then there exists a connected neighborhood of the origin \(K \subset \C\) such that, for every \(\varepsilon \in K\), the function \(F_\varepsilon\) is \(D\)-nice.
	\end{itemize}
\end{proposition}
\begin{proof}
    Let $ \Delta = \{(a,b)\}$ be a non-extreme vertex of $D$. Let $\theta(\bm{x},\varepsilon) $ such that $ (F_{\varepsilon})_{\Delta}(\bm{x}) = f_\Delta(\bm{x} ) + \theta(\bm{x},\varepsilon) $. Then

    \[f_\Delta (x_1,r_2\rme^{\rmi t_2}) =\overbar{x_{1}}^a r_{2}^{\,b} \sum_{i=0}^a c_{i}(t_2)\left(\frac{x_1}{\overbar{x_{1}}}\right)^i , \, \theta_\varepsilon(x_1,r_2\rme^{\rmi t_2}) = \overbar{x_{1}}^a r_{2}^{\,b} \sum_{i=0}^a d_{i}(t_2,\varepsilon)\left(\frac{x_1}{\overbar{x_{1}}}\right)^i.\]
    Define the family of univariate complex polynomials:   \[p_{t}(z):=\sum_{i=0}^a c_{i}(t)z^i, \, q_{t,\varepsilon}(z) = \sum_{i=0}^a d_{i}(t,\varepsilon)z^i. \]
    
    Suppose that $V(f_{\Delta}) \cap (\C^*)^2 \not= \emptyset $. Then, by \eqref{eq: nice equivalence}, there exist $t_0$  and unitary $z_0$ such that $p_{t_0}(z_0) = 0$. Define the map $G:\R^5 \rightarrow \R^3, \, G(z,t,\varepsilon) = (p_t(z)+q_t(z,\varepsilon),z\overbar{z}-1) $. Then $G(z_0,t_0,0)=0$ and 
    \[J_G(z_0,t_0,0) = \begin{pmatrix}
        \frac{\partial p_{t_0}}{\partial z}(z_0)  & \frac{\partial p_{t_0}}{\partial t}(z_0) & \frac{\partial q_{t_0}}{\partial \varepsilon}(z_0,0) \\
        z_0 & 0& 0
    \end{pmatrix}. \]    
    Note that  
    \[\det \begin{pmatrix}
        \frac{dp_{t_0}}{dz}(z_0) & \frac{dp_{t_0}}{dt}(z_0)\\
        z_0 & 0
    \end{pmatrix} = -\Im\left( \rmi z_0 \frac{dp_t}{dz}(z_0)\overbar{\frac{dp_t}{dt}(z_0)}\right) .\]
    By Proposition~\ref{prop:characterizationknd} the right-hand side of the previous equation is nonzero. Hence, the matrix on the left-hand side is invertible. Thus, by implicit function theorem we have that there exist a neighborhood $K$ of the origin, and functions $z(\varepsilon)$ and $t(\varepsilon)$ such that $z(0)=z_0$, $t(0) = t_0$ and $G(z(\varepsilon),t(\varepsilon),\varepsilon)=0 $ for all $\varepsilon \in K$. Hence, for any $\varepsilon_0 \in K$ we have that $z(\varepsilon_0)$ is a simple root of $p_{t_0}(z)+q(z,\varepsilon_0)$. Therefore, by \eqref{eq: nice equivalence}, $V((F_\varepsilon)_{\Delta}) \cap (\C^*)^2 \not= \emptyset $ for all $\varepsilon \in K$.

    Suppose that $V(f_{\Delta}) \cap (\C^*)^2 = \emptyset $. Assume, by contradiction, that there is a sequence $\varepsilon_n \rightarrow 0$ such that $V((F_{\varepsilon_n})_{\Delta}) \cap (\C^*)^2 \not= \emptyset $. Then, by \eqref{eq: nice equivalence}, there exist $z_n \in S^1$ and $t_n \in [0, 2 \pi]$ such that $p_{t_n}(z_n)+q_{t_n,\varepsilon_n}(z_n) = 0$. Since $S^1$ and $[0, 2\pi]$ are compact there is a subsequence such that $z_n \rightarrow z_0 \in S^1$, $t_n \rightarrow t_0 \in [0 , 2 \pi]$, and by continuity of $p$ and $q$ in the variables $z$, $t$ and $\varepsilon$ we have that $0 =\lim_{n \rightarrow 0}p_{t_n}(z_n)+q_{t_n,\varepsilon_n}(z_n) = p_{t_0}(z_0) + q_{t_0,0}(z_0) =  p_{t_0}(z_0)$, which is a contradiction. Then, there exists a connected neighborhood $K$ such that, for every \(\varepsilon \in K\), the function \(F_\varepsilon\) is not \(D\)-nice.
\end{proof}

\begin{remark}
Fix a $C$-face diagram $D \subset \R^2$, and let $m \in \N$ be such that any mixed polynomial function
$f \colon \C^2 \to \C$ with $\supp (f) \subseteq D$ can be written in the form
\(
f(\bm{x}) = \sum_{i=1}^m c_i \bm{x}^{\bm{\nu}_i} \overbar{\bm{x}}^{\bm{\mu}_i}.
\)
Then $\C^m$ can be identified with the space of coefficients of mixed polynomial functions whose support is contained in $D$.
Define
\( 
U_D := \bigl\{ \bm{c} \in \C^m \mid \text{the mixed polynomial }
f_{\bm{c}}(\bm{x}) = f(\bm{x},\bm{c}) = \sum_{i=1}^m c_i \bm{x}^{\bm{\nu}_i} \overbar{\bm{x}}^{\bm{\mu}_i}
\text{ is IKND and } \supp (f) \subseteq D \bigr\}.
\)
Since the IKND condition is open, the set $U_D$ is an open subset of $\C^m$.
Now define
\(
U_{D\text{-nice}} := \bigl\{ \bm{c} \in U_D \mid f_{\bm{c}}(\bm{x}) \text{ is $D$-nice} \bigr\}.
\)
Then we have the decomposition
\(
U_D = U_{D\text{-nice}} \cup ((U_{D\text{-nice}})^c \cap U_D ).
\)
Proposition~\ref{prop: niceisopen} implies that both $U_{D\text{-nice}}$ and $(U_{D\text{-nice}})^c \cap U_D$ are open subsets of $\C^m$.
Therefore, $U_D$ is not connected. In particular, there is no continuous path of IKND mixed polynomial functions, with respect to the fixed $C$-face diagram $D$, that connects a $D$-nice mixed polynomial function with a non $D$-nice one.
\end{remark}

In \cite{AraujoBodeSanchez}, Ara\'ujo dos Santos, Bode, and Sanchez Quiceno studied the link of  INND mixed polynomial functions of two variables that are $\Gamma$-nice; they proved that the link of the singularity depends only on the terms of the Newton boundary and the isotopic type has a representative in the form $\mathbf{L}([L_1, L_2, \dots, L_{N-1}], L_{N})$ (Definition \ref{nestedtori} and \ref{nestedlink}),  where the links $L_{i}$ are constructed from the terms of $f$ on the Newton boundary. 
\begin{definition}\label{nestedtori}\cite{AraujoBodeSanchez}
Let \((L_1,L_2, \dots, L_{\ell})\) be a sequence of possible empty links, $L_1$ in $\C\times S^1$, $L_{i}, i=2,3,\dots, \ell$, in $\C^*\times S^1$  where each $L_i$ is parametrized by 
\begin{equation*}
\bigcup_{j=1}^{M_i}\{(u_{i,j}(\tau),\rme^{\rmi t_{i,j}(\tau)})\mid  \,\tau\in[0,2\pi]\}\subset\mathbb{C}\times {S}^1,
\end{equation*}
then we define $[L_1,L_2,\ldots,L_{\ell}]$ on $\C \times S^1$ as follows:
\begin{equation*}
\bigcup_{i=1}^{\ell}\bigcup_{j=1}^{M_i}\{(\varepsilon^{k_i} u_{i,j}(\tau),\rme^{\rmi t_{i,j}(\tau)})\mid  \,\tau\in[0,2\pi]\}\subset\mathbb{C}\times S^1,
\end{equation*}
for some sufficiently small $\varepsilon>0$, 
\end{definition}  
\begin{definition}[\cite{AraujoBodeSanchez}]\label{nestedlink}
Let $K_1$ be a link in $\C\times S^1$ that is either empty or parametrized by
\begin{equation}\label{linkparametrization}
\bigcup_{j=1}^M\{(u_{j}(\tau),\rme^{\rmi t_{j}(\tau)})\mid\tau\in[0,2\pi]\}
\end{equation}
for appropriate functions $u_j$, $t_j$,
and $K_2$ be a link in $S^1\times\C$ that is either empty or parametrized by
\begin{equation*}
\bigcup_{j=1}^{M'}\{(\rme^{\rmi \varphi_{j}(\tau)},v_{j}(\tau))\mid \tau\in[0,2\pi]\}
\end{equation*}
for appropriate functions $v_j$, $\varphi_j$.
We define the link $\widetilde{K_i}$, $i=1,2$, in $S^3$, which is empty if $K_i$ is empty and otherwise parametrized by 
\begin{equation*}
\bigcup_{j=1}^M\left\{\left(\varepsilon u_{j}(\tau),\sqrt{1-\varepsilon^2|u_j(\tau)|^2}\rme^{\rmi t_{j}(\tau)}\right)\mid \tau\in[0,2\pi]\right\}
\end{equation*}
and
\begin{equation*}
\bigcup_{j=1}^{M'}\left\{\left(\sqrt{1-\varepsilon^2|v_j(\tau)|^2}\rme^{\rmi \varphi_{j}(\tau)},\varepsilon v_{j}(\tau)\right)\mid \tau\in[0,2\pi]\right\}
\end{equation*}
for some small $\varepsilon>0$, respectively.
Then we write $\mathbf{L}(K_1,K_2)$ for the link in $S^3$ given by $\widetilde{K_1}\cup\widetilde{K_2}$.
\end{definition}     
In \cite{Bode2025}, Bode determined the topological conditions on the links
$L_{i}$ that appear in
$\mathbf{L}([L_1, L_2, \dots, L_{N-1}], L_N)$.
Furthermore, in \cite[Theorem~1.6]{Bode2025}, he gave a complete
characterization of the links that arise as links of $\Gamma$-nice and
INND mixed polynomial functions.

In what follows, we present a slightly different proof of
\cite[Theorem~1.2]{AraujoBodeSanchez}, showing that the link of a
$\Gamma$-nice and INND mixed polynomial (equivalently,
$\Gamma_{\mathrm{inn}}$-nice and IKND) is of the form
\(
\mathbf{L}([L_1, L_2, \dots, L_{N-1}], L_N).
\)
Our proof is based on the construction of a link-constant piecewise analytic family of mixed polynomials.
 
\begin{proposition} \label{th:piecewise family}
    Let $f$ be an IKND mixed polynomial function that is $\Gamma_{\rm inn}$-nice. Then, there is a link-constant piecewise analytic family $ \{ f_\varepsilon\}_{\varepsilon \in [0,1]}$, with $f_0 = f $ such that the link of $f_1$ is isotopic to \(\mathbf{L}([L_1, L_2, \dots, L_{N-1}], L_{N}),\) where $L_1 \subset \C \times S^1$, $L_i \subset \C^* \times S^1, \ i=2,3,\dots, N-1 $, $L_{N} \subset S^1 \times \C$.
\end{proposition}
\begin{proof}
    Let $\mathcal{P}_{\rm inn}(f) = \{ \bm{w}_1, \dots, \bm{w}_N \}$. We distinguish two cases according to the value of $N$, namely $N>1$ and $N=1$.
    
    \underline{Case I:} $N>1$. 
    The $C$-face diagram $\Gamma_{\rm inn}(f)$ has at least two compact 1-faces. We write the variables in polar coordinates as $x_1=r_1\rme^{\rmi t_1}$ and $x_2=r_2\rme^{\rmi t_2}$. For $f$ and  each weight vector $\w_i=(w_{i,1},w_{i,2}) \in \mathcal{P}_{\rm inn}(f)$, we define the associated functions:
\[g_i(x_1,\rme^{\rmi t_2}):=\lim_{r_2 \to 0} r_2^{-\frac{d(\w_i;f)}{w_{i,2}}}f( r_{2}^{\,k_i} x_1,r_2\rme^{\rmi t_2})\]
where $i=1,2,\dots,N-1$, and $g_i: \C \times S^1 \to \C$.
Similarly, we define
\[h_{N}(\rme^{\rmi t_1},x_2):=\lim_{r_1 \to 0}  r_1^{-\frac{d(\w_N;f)}{w_{N,1}}}f(r_1 \rme^{\rmi t_1},r_1^{\frac{1}{k_N}}x_2) \]
where $h_{N}: S^1 \times \C \to \C$.
  
	We now consider the following subsets: 
	\begin{align*}
		& L_{1}:= V(g_1)\cap  (\C \times S^1) \subset \C \times S^1,\\
		& L_{i}:= V(g_i)\cap (\C^* \times  S^1) \subset \C^* \times S^1, \ i \in \{2,3,\dots,N-1\}, \\
		& L_{N}:=V(h_N)\cap (S^1 \times \C)\subset S^1 \times \C .  
	\end{align*}
    Since $f$ is IKND and $\Gamma_{\rm inn}$-nice, each $L_i$ is compact and smooth link in the corresponding ambient space. Applying the same arguments as in \cite[Theorem~1.2]{AraujoBodeSanchez}, with $\mathcal{P}_{\rm inn} (f)$ in place of $\mathcal{P}(f)$ we conclude that the link $L_f $ is isotopic to $\textbf{L}([L_1, \dots, L_{N-1}],L_N)$. Therefore, in this case we may take the constant family $\{f_\varepsilon \}_{\varepsilon \in [0,1]} $, defined by $f_\varepsilon = f$ for all $\varepsilon.$
    
    \underline{Case II:}  $N=1$.  Suppose that $\text{supp}(f) \cap \Gamma_{\rm inn}(f) \cap \R^{\{1\}} = \emptyset$. We consider the family $\{F_\varepsilon \}_{\varepsilon \in [0,1]}$ defined by $F(\bm{x}, \varepsilon)= f+\varepsilon M(\bm{x}_1)$, such that $\supp (M) = \{(m_,0)\} \subset \text{Int} (\Gamma_{\rm inn}(f)+(\mathbb{Z}_{\geq 0})^2$).
    By Theorem~\ref{th:nonice}, the deformation $F$ is link-constant. Let $S_{\rm inn} \subset \mathbb{Z}^2$ be such that $\Gamma(S_{\rm inn}) = \Gamma_{\rm inn} (f)$, and define $D=\Gamma((S_{\rm inn} \setminus \R^{\{1\}}) \cup\supp (M))$. Then $\mathcal{P}(D)=\{\bm{w}_1, \bm{w}_2\} $. Since $f$ is IKND, we can conclude that $f_{w_2} - M$ is KND and that $\supp (f_{w_2} - M) = \{( m' , 1) \}  $. Using the same argument as in \cite[Corollary~4.11]{AraujoBodeSanchez} one proves that $\Sigma((f+M)_{\Delta(\bm{w}_2;D)}) \cap V((f+M)_{\Delta(\bm{w}_2;D)}) \cap (\C^*)^I = \emptyset $ for $I = \{1\}$ and $I=\{1,2\} $. Then, $f + M$ is IKND with respect to $D$. Hence, applying the same argument as in Case I to $\mathcal{P}(D)$, we conclude that the link type of $f+M$ has the desired form. Consequently, $\{ F_\varepsilon \}_{\varepsilon \in [0,1]}$ is a link-constant piecewise analytic family.

 Suppose now that $\text{supp}(f) \cap \Gamma_{\rm inn}(f) \cap \R^{\{1\}} \not= \emptyset$. Let $M = f_\Delta $, where $\Delta =  \Gamma_{\rm inn}(f) \cap \R^{\{1\}} =\{(a,0)\}$. Let $n= \deg_{x_1}(M(x_1))$. Then $M$ can be written as $M(x_1) = \sum_{i=0}^n c_i  x_{1}^{\,i} \overbar{x_{1}}^{a-i} $, where the coefficients $c_i,\ i=0,\dots, n-1,$ are (possibly zero). Equivalently,
\[M(x_1)=\overbar{x_{1}}^a\left(\sum_{i=0}^n c_i \left(\frac{x_1}{\overbar{x_{1}}}\right)^i\right),\]
Hence, 
\[M(x_1)=c_{n} \overbar{x_{1}}^a \prod_{i=0}^n\left(\frac{x_1}{\overbar{x_{1}}}-\gamma_i\right),\]
where $\gamma_1,\dots ,\gamma_n$ are the roots of the polynomial $\overbar{x_{1}}^{-d}M(x_1)$ in the variable $z = \frac{x_1}{\overbar{x_{1}}}$. Set $I=\{i \mid |\gamma_i|=1\}.$ By Remark \ref{remark:charKNDvertex}-(ii), the polynomial $M$ is KND if and only if $I = \emptyset$. If $M$ is KND, then we may apply the same argument as in Case I with $\mathcal{P}_{\rm inn}(f) = \{ \bm{w}_1 \} $, obtaining that the link of $f$ is isotopic to $\mathbf{L}([L_1])$. Then we take the constant family $\{f_\varepsilon \}_{\varepsilon \in [0,1]} $ with $f_\varepsilon = f$ for all $\varepsilon$. Assume now that $M$ is not KND, that is, $I \not= \emptyset $. Suppose moreover that each $\gamma_i$, $i \in I $, is a simple root. Define

\[\theta(x_1,\varepsilon)= \varepsilon \left(\sum_{i=0}^n i c_i x_{1}^{\,i} \overbar{x_{1}}^{a-i}\right).\] 
Applying the implicit function theorem to $\overbar{x_{1}}^{-a}(M(x_1) + \theta(x_1,\varepsilon))$ (which is a complex function of variable $(\frac{x_1}{\overbar{x_{1}}}, \varepsilon) $) at the point $(\gamma_i, 0) $, we deduce that $\overbar{x_{1}}^{-a}(M(x_1) + \theta(x_1,\varepsilon))$ has no roots in the unit circle for all sufficiently small real $ \varepsilon\not= 0 $. Consequently, $M(x_1) + \theta(x_1,\varepsilon)$ is KND for all sufficiently small real $\varepsilon\neq 0$. By Theorem~\ref{th:nonice}, there exists a neighborhood $K \subset \C$ of the origin such that the deformation $F(\bm{x},\varepsilon) = f(\bm{x}) + \theta( x_1,\varepsilon)$ is link-constant along $K$. Let $\varepsilon_1\in K\cap \R$ be sufficiently small so that $M(x_1)+\theta(x_1,\varepsilon_1)$ is KND. Then, the family $\{F_{\delta \varepsilon_1}\}_{\delta \in [0,1]}$ is link-constant. Moreover, $F(\bm{x},\varepsilon_1) = f(\bm{x}) + \theta( x_1,\varepsilon_1)$ is IKND with respect to $\Gamma_{\rm inn}(f)$. Applying the same argument as in Case I to $\mathcal{P}_{\rm inn}(f) = \{ \bm{w}_1\}$, we conclude that the link type of $F_{\varepsilon_1}$ has the desired form. Thus, we obtain a link-constant family $\{F_{\delta \varepsilon_1}\}_{\delta \in [0,1]}$.

Suppose that there exists $i \in I$ such that $\gamma_i $ is not a simple root. Since having only simple roots is an open property on the coefficients, there exists  $\varepsilon_2  \in K $ sufficiently small such that $\overbar{x_{1}}^{-d}(M(x_1) + \hat{\theta}(x_1,\varepsilon_2))  $ viewed as a polynomial in the variable $\frac{x_1}{\overbar{x_{1}}} $ has only simple roots, where $\hat{\theta}_\varepsilon(x_1)$ is a mixed polynomial function satisfying $\supp (\hat{\theta}_\varepsilon )= \supp (M) $. Define $\hat{F}(\bm{x},\varepsilon) = f(\bm{x},\varepsilon) + \hat{\theta}(x_1,\varepsilon) $. By argument above, there exists a family $\{\hat{f}_\delta\}_{\delta \in [0,1]}$ such that $\hat{f}_0 = \hat{F}_{\varepsilon_2}$ and the link of $\hat{f}_1$ has the desired form. Therefore, the desired piecewise family is obtained  by concatenating the families $\{\hat{F}_{\delta\varepsilon_2}\}_{\delta \in [0,1]} $ and $\{\hat{f}_\delta\}_{\delta \in [0,1]}$.
\end{proof}
\begin{ex}\label{ex:nice1}
Consider $f$ as Example~\ref{ex:nice}. Thus, $\mathcal{P}_{\rm inn }(f)=\{\w_1=(1,1), \w_2=(2,3)\}$. The mixed polynomial $f$ is $\Gamma_{\rm inn}$-nice. By Corollary~\ref{cor:nicecase}, the link of $f$ is isotopic to the link of $f_{\Gamma_{\rm inn}}
    $, which is represented by the link $\mathbf{L}([L_1],L_{2})$, where 
	$$L_1=V(g_1)\cap (\C \times S^1) 
	\text{ and }L_{2}=V(h_{2})\cap (S^1 \times \C).$$
	Here, $g_1(x_1,\rme^{\rmi t_2})=x_1 \rme^{- 3 \rmi t_2}$ and $h_2(\rme^{\rmi t_1},x_2)=\overbar{x_{2}}\rme^{\rmi t_1}( \rme^{3 \rmi t_1}+ \overbar{x_{2}}^2)$.
	The link $L_1= \{0\} \times S^1$ and $L_{2}$ splits in two components $ S^1 \times\{0\}$ and the trefoil $L$ given by the roots $\rme^{-\frac{3}{2} \rmi t_1}$ and $-\rme^{-\frac{3}{2} \rmi t_1}$. 
\end{ex}
Finally, we generalize \cite[Theorem~1.9]{Bode2025} to the full class of IKND mixed polynomial functions. This result shows that, in order to obtain a complete characterization of the links arising from IKND mixed polynomial functions, it suffices to study the subclass of convenient and KND mixed polynomials. This remains an open problem and presents significant challenges, which motivate the development of new tools for the study of links of singularities. Such tools may also prove useful in higher-dimensional settings.  
\begin{proposition}\label{cor:classoka}
A link type $L\subset S
^3$ arises as the link of a convenient KND mixed polynomial function if and only if it is the link of an IKND mixed polynomial function.
\end{proposition}
\begin{proof}
    $(\Rightarrow)$ This implication follows directly from
    Proposition~\ref{prop: knd co -> iknd}. 

    $(\Leftarrow)$ We shall construct a link-constant family $\{f_\varepsilon\}_{\varepsilon \in [0,1] }$ such that $f_0 =f$ and $f_1$ is convenient KND mixed polynomial. There are four cases to consider 
    
    $$\begin{cases}
        \text{supp}(f) \cap \Gamma_{\rm inn}(f) \cap \R^{\{1\}} = \emptyset \text{ and }
        \text{supp}(f) \cap \Gamma_{\rm inn}(f) \cap \R^{\{2\}} = \emptyset \\
        \text{supp}(f) \cap \Gamma_{\rm inn}(f) \cap \R^{\{1\}} \not= \emptyset \text{ and }
        \text{supp}(f) \cap \Gamma_{\rm inn}(f) \cap \R^{\{2\}} = \emptyset \\
        \text{supp}(f) \cap \Gamma_{\rm inn}(f) \cap \R^{\{1\}} = \emptyset \text{ and }
        \text{supp}(f) \cap \Gamma_{\rm inn}(f) \cap \R^{\{2\}} \not= \emptyset \\
        \text{supp}(f) \cap \Gamma_{\rm inn}(f) \cap \R^{\{1\}} \not= \emptyset \text{ and }
        \text{supp}(f) \cap \Gamma_{\rm inn}(f) \cap \R^{\{2\}} \not= \emptyset .
    \end{cases}$$
    We explain Case II, the remaining cases are analogous.
    
    Consider the family $\{F_\varepsilon \}_{\varepsilon \in [0,1]}$ defined by $F(\bm{x}, \varepsilon)= f+\varepsilon M(\bm{x}_2)$, where $M$ is a mixed monomial satisfying $\supp (M) = \{(0,b)\} \subset \mathrm{Int} (\Gamma_{\rm inn}(f)+(\mathbb{R}_{\geq 0})^2$) and such that $M$ is KND.
    By Theorem~\ref{th:nonice}, the deformation $F$ is link-constant. By the construction in the proof of Theorem~\ref{th:piecewise family} there exists a link-constant piecewise family $\{\hat{f}_\varepsilon\}_{\varepsilon \in [0,1]}$ such that $\hat{f}_0 = F_1$ and $\hat{f}_1$ is convenient and KND. Therefore, the link of $f$ is isotopic to the link of $\hat{f}_1$.
    
\end{proof}

  \bibliographystyle{amsplain}
  \renewcommand{\refname}{ \large R\normalsize  EFERENCES}
\bibliography{sn-bibliography}
\Addresses

\end{document}